\documentclass{article}
\usepackage[utf8]{inputenc}
\usepackage{amsmath,amsthm}
\usepackage{amsfonts}
\usepackage{graphicx}
\usepackage{authblk}
\usepackage{biblatex}

\graphicspath{ {./images/} }

\newtheorem{theorem}{Theorem}[section]
\newtheorem{lemma}{Lemma}[section]
\newtheorem{corollary}{Corollary}[section]
\newtheorem{proposition}{Proposition}[section]

\theoremstyle{definition}
\newtheorem{definition}{Definition}[section]
\newtheorem{example}{Example}[section]
\newtheorem{remark}{Remark}[section]

\title{New four-vertex type theorems for spherical polygons}
\author{Samuel Pacitti Gentil }
\affil{\begin{small} Pontifícia Universidade Católica, Rio de Janeiro

sampacitti@gmail.com \end{small}}

\date{}

\begin{document}

\maketitle

\begin{abstract}
    We prove discrete analogs of four-vertex type theorems of spherical curves, which imply corresponding results for space polygons. The smooth theory goes back to the work of Beniamino Segre and, more recently, by Mohammad Ghomi, and consists of theorems that state, for a given closed spherical curve, nontrivial lower bounds on the number of spherical inflections plus the number of self and/or antipodal intersections counted with multiplicity. We study these concepts and results adapted to the case of spherical polygons and prove, using only discrete tools, the corresponding theorems.
\end{abstract}

\section{Introduction}

The original four-vertex theorem, proved by Mukhopadhyaya in 1909 (see \cite{Mukhopadhyaya}), states that a regular closed plane curve which is also simple and convex must have at least four points at which its curvature function attains a local maximum or minimum (these are the vertices of the curve). This theorem has various generalizations, some of which being theorems about regular closed space curves.

One of these theorems, proved by Segre in 1968 (see \cite{Segre}), states that a regular closed space curve without any pair of parallel tangents with same orientation must have at least four points at which the torsion vanishes (these points are called the flattenings of the curve). This can be reformulated in the following way: the tangent indicatrix of a space curve is the spherical curve described by the unit tangent vectors of the space curve after translating their base point to the origin. Since flattenings of the space curve correspond to spherical inflections of its tangent indicatrix, Segre's theorem is a corollary of the following theorem, also proved by Segre: a regular simple and closed spherical curve which is not contained in any closed hemisphere must have at least four spherical inflections. An easy corollary of this result states that, if such a spherical curve is also symmetric with respect to the origin, then it must have at least 6 spherical inflections (this result is equivalent to a theorem on projective curves due to Möbius - see \cite{Mobius}).

More recently, Ghomi (see \cite{Ghomi}) proved that Segre's theorem on spherical curves also admits an improvement to 6 on the lower bound of spherical inflections if such curves do not pass through pairs of antipodal points. Moreover, he proves that all these previous theorems are particular cases of results that take into account the number of self-intersections, ``antipodal intersections" (i.e., pairs of antipodal points through which the curve passes) and singular points (i.e., cusps) of the spherical curve.

Although theorems on smooth curves are interesting in their own right, one can also study the discrete analog of them. Four-vertex type theorems for polygons (i.e., closed polygonal lines) have the advantage of being easier to state and having in their proofs relatively simple combinatorial arguments. Moreover, a theorem about polygons provides, in the limit, a theorem about smooth curves. We also notice that discretizing is nontrivial in the sense that there might be different discrete results (relative to a smooth one) which are not equivalent to each other. Therefore, it makes sense to discretize Segre's theorem and its extensions.

Although a proof of Segre's original result for spherical polygons was already found by Panina (see \cite{Panina}), her proof uses the original theorem by Segre (i.e., the smooth case). As far as we know, a proof using only discrete tools was first found in \cite{Gentil}, where other interesting results related to the geometry of spherical polygons were also obtained. The main theorem of \cite{Gentil} states that a simple spherical polygon not contained in a closed hemisphere must have at least 4 (discrete) spherical inflections. If, moreover, such polygon is symmetric, then this lower bound on the number of inflections can be improved to 6. The proof in \cite{Gentil} consists of an induction argument on the number of vertices, and its most difficult step is the proof the existence of a vertex whose deletion does not cause the resulting polygon to be contained in a closed hemisphere nor to have self-intersections. Crucial to the proof of this result are a triangulation technique and some elementary algebraic relations that encode the hypotheses of the theorem.

In this paper we state and prove discrete analogs of Ghomi's results. In order to do this, we adapt the tools that we developed in \cite{Gentil} and deduce an elementary algebraic relation to express an ``antipodal intersection" of a spherical polygon. We first prove that a spherical polygon, not contained in any closed hemisphere and not having self nor antipodal intersections, must have at least 6 (discrete) spherical inflections.

Later we deal with spherical polygons which, although not contained in any closed hemisphere, might have self or antipodal intersections. Denote by $I$, $D^+$ and $D^-$ the numbers of (discrete) inflections, self-intersections and antipodal intersections of a spherical polygon, respectively. Put $D:= D^+ + D^-$. Using a ``cutting-and-pasting" technique, we prove that, for any polygon not contained in any closed hemisphere, the following inequalities hold:
$$ 2D + I \geq 6$$
and
$$ 2D^+ + I \geq 4.$$
Moreover, if the spherical polygon is symmetric with respect to the origin, then we have
$$ 2D^+ + I \geq 6.$$

\section{Previous definitions and results}

We begin this section with the original theorems proved by Segre (see \cite{Segre}):
\begin{definition}
\label{segre_curve}
A closed curve $\gamma: \mathbb{S}^1 \rightarrow \mathbb{R}^3$ is called a \textit{Segre curve} if it has non-vanishing curvature and if, for any $t_1 \neq t_2 \in \mathbb{S}^1$, the tangent vectors $\gamma'(t_1)$ and $\gamma'(t_2)$ do not point to the same direction.
\end{definition}

\begin{definition}
A \textit{flattening} or a \textit{inflection point} is a point $p=\gamma(t_0)$ of a curve, for some $t_0 \in \mathbb{S}^1$, at which $\tau(t_0)=0$.
\end{definition}

\begin{theorem}
\label{segre_theorem}
\textit{(Segre)} Any Segre curve has at least 4 flattenings.
\end{theorem}

We can reformulate the previous definitions and result in the following way:

\begin{definition}
Given a smooth curve $\gamma$ in $\mathbb{R}^3$ with non-vanishing curvature, translate the unit tangent vector at each point of the curve to a fixed point \textbf{0}. The endpoints of the translated vectors describe then a curve on the unit sphere $\mathbb{S}^2$. We call this curve the \textit{tangent indicatrix} of $\gamma$.
\end{definition}

Therefore, the notion of a Segre curve can be reformulated as a closed curve such that its tangent indicatrix is embedded in $\mathbb{S}^2$ (i.e., smooth and without self-intersections). Moreover, Theorem \ref{segre_theorem} now reads:

\begin{theorem}
\label{segre_theorem_reformulated}
Let $\gamma$ be a closed curve in $\mathbb{R}^3$. If its tangent indicatrix is embedded in $\mathbb{S}^2$, then $\gamma$ has at least 4 flattenings.
\end{theorem}

Since flattenings of the original space curve correspond to spherical inflection points (i.e., points at which the tangent spherical line has a greater order of contact with the curve) of the corresponding tangent indicatrix, we have that theorem \ref{segre_theorem_reformulated} is a corollary of the following result:

\begin{theorem}
\label{segre_theorem_spherical}
    \textit{(Segre)} If a spherical closed curve is smooth and simple, and moreover is not contained in any closed hemisphere, then it has at least 4 spherical inflection points.
\end{theorem}

An easy consequence of theorem \ref{segre_theorem_spherical} is the following result, which is equivalent to a theorem by Möbius (see \cite{Mobius}):

\begin{theorem}
\label{mobius_theorem}
    If a spherical closed curve is smooth and simple and, moreover, is symmetric with respect to the origin, then it has at least six spherical inflection points.
\end{theorem}

Now we discretize the previous definitions and theorems. All the discrete results in the remaining of this section were already proved in \cite{Gentil}.

\begin{definition}
Given a polygon $P = [v_1,v_2,...,v_n]$ in $\mathbb{R}^3$ (i.e., a closed polygonal line in $\mathbb{R}^3$, where we consider the indices $i$ modulo $n$), denote by $u_i$ the unit tangent vector with the same direction of the edge $e_i$, i.e.,
$$ u_i = \frac{e_i}{|e_i|} = \frac{\overrightarrow{v_i v_{i+1}}}{|\overrightarrow{v_i v_{i+1}}|} = \frac{v_{i+1}-v_i}{|v_{i+1}-v_i|}.$$
We define the \textit{(discrete) tangent indicatrix} of $P$ as the \textit{closed spherical polygonal line}, i.e., the \textit{spherical polygon}
$$ Q = [u_1,u_2,...,u_n],$$
whose edges are the spherical segments (with minimal length) joining $u_i$ and $u_{i+1}$. This definition goes back to the work of Banchoff (see \cite{Banchoff}).
\end{definition}

We can finally define the discrete counterpart of Definition \ref{segre_curve}:

\begin{definition}
A polygon $P$ is a \textit{Segre polygon} if its tangent indicatrix $Q$ does not have self-intersections.
\end{definition}

\begin{definition}
A \textit{flattening} of a polygon is a triple $\{v_i,v_{i+1},v_{i+2}\}$ such that $v_{i-1}$ and $v_{i+3}$ are on the same side of the plane generated by vertices $v_i$, $v_{i+1}$ and $v_{i+1}$ (see figure \ref{figure_flattening_non_flattening}).
\end{definition}

\begin{figure}
    \centering
    \includegraphics[scale=0.3]{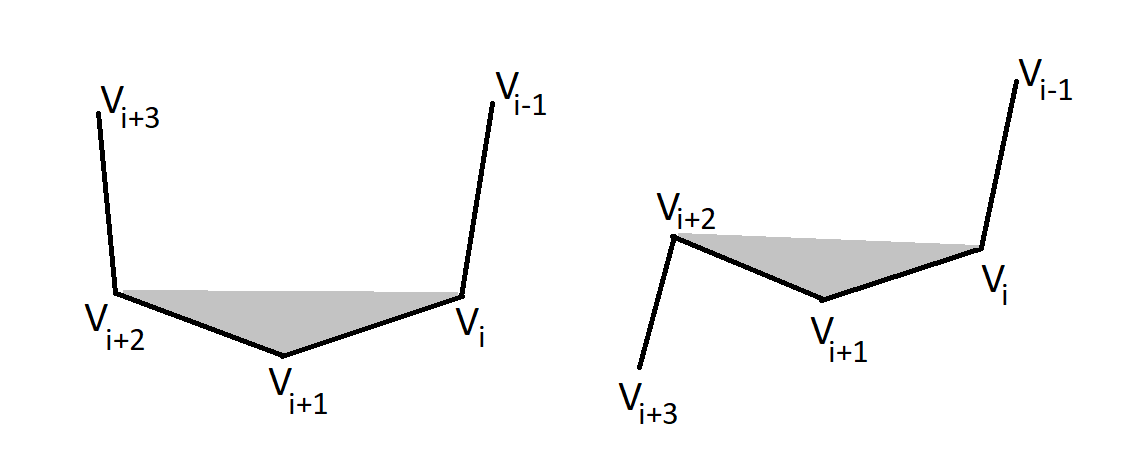}
    \caption{A flattening on the left, a non-flattening on the right}
    \label{figure_flattening_non_flattening}
\end{figure}

\begin{theorem}
\label{segre_theorem_discrete}
A Segre polygon with at least 4 vertices has at least 4 flattenings.
\end{theorem}

\begin{definition}
Given a spherical polygon $Q \subset \mathbb{S}^2$, a \textit{(spherical) inflection} of $Q$ is a pair $\{u_i,u_{i+1}\}$ such that $u_{i-1}$ and $u_{i+2}$ are in different sides of the plane spanned by $\{u_i,u_{i+1}\}$. Equivalently, $u_{i-1}$ and $u_{i+2}$ are in different hemispheres determined by the spherical line spanned by $\{u_i,u_{i+1}\}$ (see figure \ref{figure_spherical_flattening} on the right).

In terms of determinants, this means that $\epsilon_{i-1} = [u_{i-1},u_i,u_{i+1}]$ and $\epsilon_i = [u_i,u_{i+1},u_{i+2}] = [u_{i+2},u_i,u_{i+1}]$ have opposite signs.
\end{definition}

\begin{remark}
\label{remark_flattening_inflection}
The previous definition implies that, if the triple $\{v_i,v_{i+1},v_{i+2}\}$ is a flattening, then the vectors $e_{i-1}$ and $e_{i+2}$ point to different sides of the plane generated by $\{v_i, e_i, e_{i+1}\}$. This in turn happens if and only if $u_{i-1}$ and $u_{i+2}$ are on different sides of $span\{u_i,u_{i+1}\}$, i.e., $\{u_i,u_{i+1}\}$ is an inflection (see figure \ref{figure_spherical_flattening}).
\end{remark}

\begin{figure}
    \centering
    \includegraphics[scale=0.3]{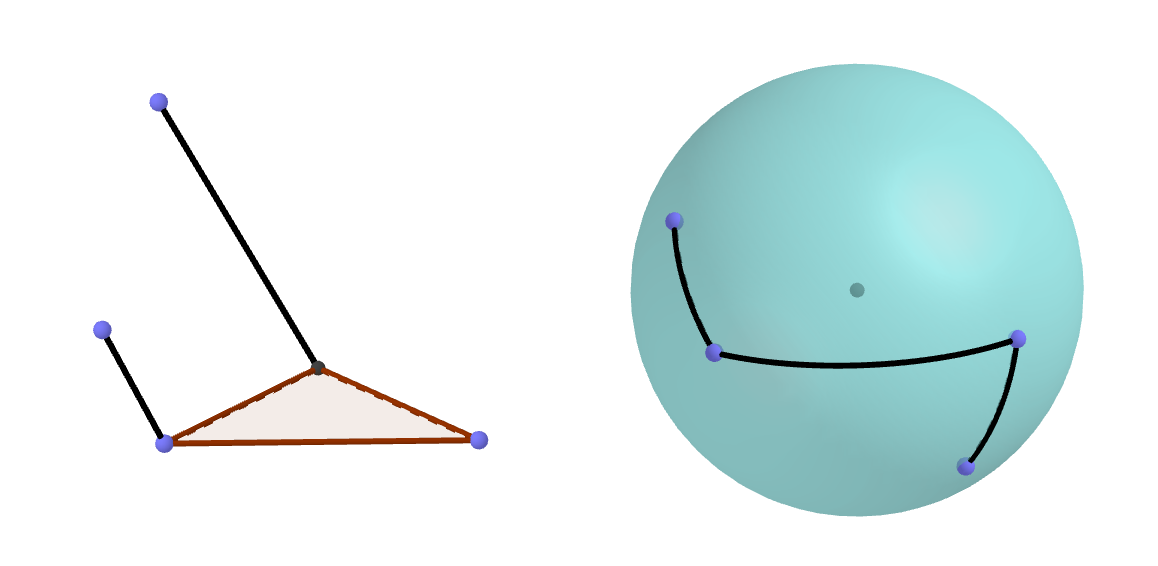}
    \caption{Flattening of a polygon $P$ and the corresponding inflection of the tangent indicatrix $Q$.}
    \label{figure_spherical_flattening}
\end{figure}

By remark \ref{remark_flattening_inflection}, we have that theorem \ref{segre_theorem_discrete} follows from the following theorem:

\begin{theorem}
\label{segre_theorem_discrete_spherical}
Let $Q=[u_1,...,u_n] \in \mathbb{S}^2$ ($n\geq 4$) be a spherical polygon not contained in any closed hemisphere and without self-intersections. Then $Q$ has at least four spherical inflections.
\end{theorem}

Theorem \ref{segre_theorem_discrete_spherical} states that, if a spherical polygon $Q=[u_1,...,u_n]$ is not contained in any closed hemisphere and does not have self-intersections, then the cyclic sequence $(\epsilon_1,\epsilon_2,...,\epsilon_n$) has at least 4 sign changes (where each $\epsilon_i$ is defined as $[u_i,u_{i+1},u_{i+2}]$).

We introduce the following terminology:

\begin{definition}
A set of points $Q = \{u_1,...,u_n\} \subset \mathbb{S}^2 $ ($n \geq 4$), not in the same spherical line, is said to be \textit{balanced} or in \textit{balanced position} if its points are not in the same closed hemisphere.

A point $u_i$ of a balanced set is said to be \textit{essential} if the set $\{u_1,...,\hat{u}_i,...,u_n\}$ is not balanced. Otherwise $u_i$ is \textit{nonessential}. For a spherical polygon $Q = [u_1,..,u_n]$, the same definition applies to $Q$ considered as a set of vertices.
\end{definition}

The terminology used above is justified because of the following result. In order to state it, we denote by $[u,v,w]$ the determinant of vectors $u,v,w \in \mathbb{R}^3$ and by the equivalence sign $\simeq$ the fact that two determinants have the same sign.

\begin{proposition}
\label{proposition_characterization}
Let $u_1,u_2,u_3,u_4$ be four points of $\mathbb{S}^2$ such that any triple of them is linearly independent. Then $u_1$, $u_2$, $u_3$ and $u_4$ are not on the same closed hemisphere if and only if the following relation holds:
$$sign[u_1,u_2,u_3] = sign[u_1,u_3,u_4] = - sign[u_1,u_2,u_4] = - sign[u_2,u_3,u_4].$$
\end{proposition}

Unless explicitly stated, we assume that any triple of unit vectors are not on the same plane (and therefore are not on the same spherical line). Hence ``$\simeq -$" can always be read as ``$\not \simeq$". Moreover, to simplify notation, we can drop the letter ``$u$" and write only the subscripts. Therefore, the conclusion of proposition \ref{proposition_characterization} can be written as
$$ [1,2,3] \simeq [1,3,4] \not \simeq [1,2,4] \simeq [2,3,4].$$

\begin{lemma}
\label{lemma_all_S2}
If a set $X = \{u_1,...,u_n\} \subset \mathbb{S}^2$ ($n \geq 4$) is in balanced position, then any point of the sphere is a positive linear combination of points of $X$.
\end{lemma}

\begin{lemma}
\label{lemma_pick_up}
Let $Q = \{u_1, u_2,..., u_n \}$ be a finite set of points on the sphere $\mathbb{S}^2$, with $n \geq 5$, not all of them on the same hemisphere. Then:
\begin{itemize}
    \item [(a)] there are four points $u_{i_1}$, $u_{i_2}$, $u_{i_3}$ and $u_{i_4}$ of $Q$ in balanced position;
    \item [(b)] the set of nonessential vertices of $Q$ has at least $n-3$ elements;
\end{itemize}
\end{lemma}

We can also express algebraically the fact that two edges of a spherical polygon intersect (see figure \ref{fig:intersections_labeled} on the left):

\begin{proposition}
\label{self_intersection_determinants}
A spherical polygon $Q \subset \mathbb{S}^2$ has a self-intersection at edges $\overrightarrow{u_i u_{i+1}}$ and $\overrightarrow{u_j u_{j+1}}$ (where $j \neq i+1$ and $i \neq j+1$) if and only if the relation
$$ [i,i+1,j] \simeq [i+1,j,j+1] \not \simeq [i,i+1,j+1] \simeq [i,j,j+1]$$
holds.
\end{proposition}

\begin{figure}
    \centering
    \includegraphics[scale=0.3]{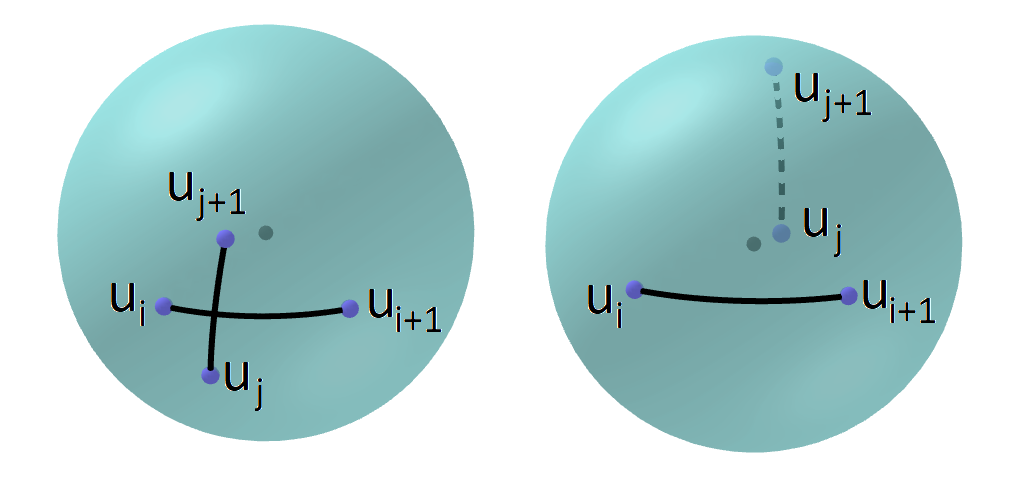}
    \caption{A self-intersection (left) and a antipodal intersection (right)}
    \label{fig:intersections_labeled}
\end{figure}

Given a spherical polygon $Q$ and any of its vertices $u_i$, denote by $Q-u_i$ the polygon $[u_1,...,\hat{u_i},...,u_n]$, obtained from $Q$ by deleting the vertex $u_i$ along with the edges $\overrightarrow{u_{i-1} u_i}$ and $\overrightarrow{u_i u_{i+1}}$, and adding the edge $\overrightarrow{u_{i-1} u_{i+1}}$ to connect vertices $u_{i-1}$ and $u_{i+1}$.

\begin{definition}
A spherical polygon $Q = [u_1,...,u_n]$ is \textit{simple} if it does not have self-intersections. A vertex $u_i$ is said to be \textit{good} if the spherical polygon $Q - u_i$ is simple. Otherwise $u_i$ is said to be \textit{bad}.
\end{definition}

\begin{lemma}
\label{lemma_resulting_polygon}
Let $Q = [u_1,u_2,...,u_n]$ be a balanced, simple spherical polygon, ($n\geq 4$). Then $Q$ has at least 4 good vertices.
\end{lemma}

\begin{lemma}
\label{lemma_inflections_not_increase}
Given a simple spherical polygon $Q$, let $u_i$ be a good vertex of $Q$. Then the number of spherical inflections of $Q$ is greater or equal to the number of spherical inflections of the resulting spherical polygon $Q-u_i$.
\end{lemma}

\begin{proof}
    (of theorem \ref{segre_theorem_discrete_spherical}) The case $n=4$ is an immediate application of propositions \ref{proposition_characterization} and \ref{self_intersection_determinants}, in which case both hypotheses and the conclusion of the theorem is verified. Now, given a balanced and simple spherical polygon $Q$ with $n\geq 5$ vertices, there is (by lemmas \ref{lemma_pick_up}.b and \ref{lemma_resulting_polygon} combined) at least one good, nonessential vertex $u_i$ of $Q$. By lemma \ref{lemma_inflections_not_increase}, the number of inflections of $Q$ is greater or equal to the number of inflections of $Q-u_i$. By the induction hypothesis, however, the number of inflections of $Q-u_i$ is greater or equal to 4.
\end{proof}

An easy corollary of theorem \ref{segre_theorem_discrete} is the following result, which is the discrete analog of theorem \ref{mobius_theorem}:

\begin{theorem}
    \label{discrete_mobius_theorem}
    Let $Q=[u_1,...,u_{2n}] \in \mathbb{S}^2$ ($2n \geq 6$) be a spherical symmetric polygon not contained in any closed hemisphere and without self-intersections. Then $Q$ has at least six spherical inflections.
\end{theorem}

\section{Polygons without self nor antipodal intersections}

An improvement on the lower bound on spherical inflections of Segre's theorem \ref{segre_theorem} was obtained by Ghomi (see \cite{Ghomi}). He proved (in the smooth setting) that, under one additional condition (namely, that such curves do not have antipodal intersections), the lower bound on the number of inflections can be improved to 6:

\begin{theorem}
\label{special_segre_smooth}
    Let $\gamma$ be a $\mathcal{C}^2$ closed spherical curve, not entirely contained in any closed hemisphere. If, for any pair of points $t \neq s \in \mathbb{S}^1$, $\gamma(t) \neq \pm \gamma(s)$ (i.e., $\gamma$ does not have self nor antipodal intersections), then $\gamma$ must have at least 6 (spherical) inflections.
\end{theorem}

Theorem \ref{special_segre_smooth} is actually a particular case of a theorem also proved by Ghomi, which will be stated in section \ref{section_with_intersections}.

The main goal of the present section until the end of section \ref{case_ess_q_2} is to state a discrete analog of theorem \ref{special_segre_smooth} and prove it. In order to do it, we use determinants again to express algebraically when the edge of a spherical polygon intersects the antipode of another edge. The proof of our theorem then proceeds on induction on the number of vertices of the polygon. The most difficult step of the proof is to prove the existence of a vertex such that, after its deletion, the resulting polygon will not be contained in a closed hemisphere nor has self or antipodal intersections.

In order to simplify the steps of our proof, we will assume from now on until the end of section \ref{case_ess_q_2} that each one of the spherical polygons being considered is such that any three of its vertices are not in the same (spherical) line, i.e., great circle.

We will also need the notion of intersection of an edge with the antipode of another edge (see figure \ref{fig:intersections_labeled} on the right). We will call such an intersection as an \textit{antipodal intersection}. Notice that, given edges $\overrightarrow{u_i u_{i+1}}$ and $\overrightarrow{u_j u_{j+1}}$ (assume that $i=1$ and $j=5$ for the sake of simplicity of notation) such that one intersects the antipode of the other, the spherical line spanned by each one intersects the other one. This implies that

$$ [1,2,5] \not \simeq [1,2,6] \text{ and } [1,5,6] \not \simeq [2,5,6]. $$

However, this condition also applies to the case of ordinary intersection of edges. The difference now is that, in the case of antipodal intersections, the spherical line spanned by $u_1$ and $u_6$ separates $u_2$ and $u_5$, while the spherical line spanned by $u_2$ and $u_5$ separates $u_1$ and $u_6$. Hence

$$ [1,6,2] \not \simeq [1,6,5] \text{ and } [2,5,1] \not \simeq [2,5,6], $$
i.e.,
$$ [1,2,6] \not \simeq [1,5,6] \text{ and } [1,2,5] \not \simeq [2,5,6]. $$

Therefore, if edge $\overrightarrow{u_1 u_2}$ intersects the antipode of $\overrightarrow{u_5 u_6}$, we have that
$$ [1,2,5] \simeq [1,5,6] \not \simeq [1,2,6] \simeq [2,5,6],$$

which is equivalent to the condition of the vertices of edges $\overrightarrow{u_1 u_2}$ and $\overrightarrow{u_5 u_6}$ being in balanced position, by proposition \ref{proposition_characterization}. We have therefore proved

\begin{proposition}
\label{antipodal_intersection_determinants}
A spherical polygon $Q \subset \mathbb{S}^2$ has a antipodal intersection at edges $\overrightarrow{u_i u_{i+1}}$ and $\overrightarrow{u_j u_{j+1}}$ (where $j \neq i+1$ and $i \neq j+1$) if and only if the following relation holds:
$$ [i,i+1,j] \simeq [i,j,j+1] \not \simeq [i,i+1,j+1] \simeq [i+1,j,j+1].$$
This, on its turn, happens if and only if the set of vectors $\{u_i,u_{i+1},u_j,u_{j+1} \}$ is in balanced position.
\end{proposition}

We can finally state the discrete analog of theorem \ref{special_segre_smooth}:

\begin{theorem}
\label{special_segre_discrete}
    Let $Q=[u_1,...,u_n] \in \mathbb{S}^2$ ($n \geq 6$) be a spherical polygon in balanced position. If $Q$ does not have self-intersections nor antipodal intersections, then it has at least six inflections.
\end{theorem}

Notice that the hypotheses of theorem \ref{special_segre_discrete} do not even hold for polygons with $4$ or $5$ vertices: since a balanced spherical polygon has $4$ of its vertices in balanced position, these vertices will necessarily be consecutive in any spherical polygon with $4$ or $5$ vertices. For spherical polygons with $6$ vertices, however, the hypotheses of theorem \ref{special_segre_discrete} can be met (see figure \ref{fig:polygon_6_vertices} for an example).

\begin{figure}
    \centering
    \includegraphics[scale=0.3]{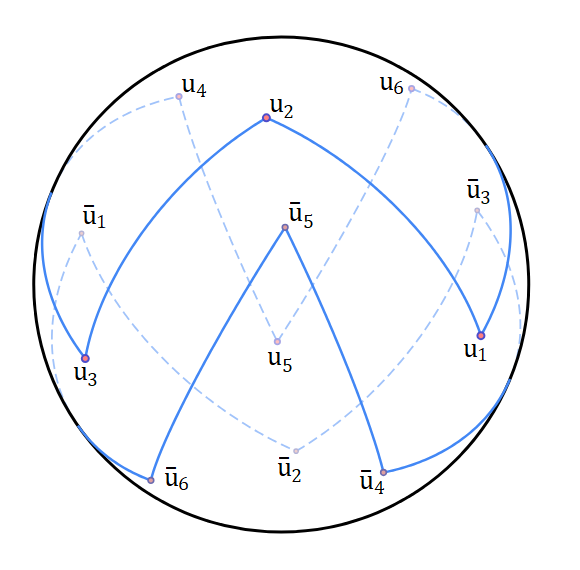}
    \caption{A balanced polygon $Q$ with $6$ vertices, without self nor antipodal intersections, together with its reflected polygon $\overline{Q}$. Notice that a polygon not having antipodal intersections is equivalent to this same polygon not intersecting its reflected version.}
    \label{fig:polygon_6_vertices}
\end{figure}

We can prove \ref{special_segre_discrete} by induction on the number $n \geq 6$ of vertices: for the case $n=6$ we can make explicit use of propositions \ref{self_intersection_determinants} and \ref{antipodal_intersection_determinants}. For the induction step we need to find a vertex $v$ with the property that the polygon $Q' = Q - \{u_i\}$ (obtained by deleting $u_i$ and its adjacent edges and connecting its adjacent vertices with a new edge) is such that:
\begin{itemize}
    \item $Q'$ is not contained in any closed hemisphere (recall that in this case $u_i$ is called \textit{nonessential});
    \item $Q'$ does not have self-intersections (recall that in this case $u_i$ is called \textit{good});
    \item $Q'$ does not have antipodal intersections.
\end{itemize}

By lemma \ref{lemma_inflections_not_increase}, the number of inflections of $Q$ is greater or equal to the number of inflections of $Q'$. By the induction hypothesis, $Q'$ has at least six inflections. It remains, therefore, to show that:
\begin{itemize}
    \item the result holds for polygons with $n=6$ vertices satisfying the hypotheses of theorem \ref{special_segre_discrete};
    \item there is a vertex that can be eliminated so that the resulting polygon still satisfies the hypotheses of theorem \ref{special_segre_discrete}.
\end{itemize}

The base case of our induction argument is given in the following proposition:

\begin{proposition}
    Let $Q=[u_1,u_2,u_3,u_4,u_5,u_6] \in \mathbb{S}^2$ be a balanced spherical polygon without self nor antipodal intersections. Then $Q$ has 6 inflections.
\end{proposition}

\begin{proof}
    Since $Q$ is balanced, four of its vertices are in balanced position. Since $Q$ does not have antipodal intersections, the indices of these vertices cannot be two pairs of consecutive indices (which includes in particular the case of four consecutive indices). Therefore the indices must be such that three of them are consecutive and the remaining one is isolated in the cyclic sequence $(1,2,3,4,5,6)$ (for example, $\{1,2,3,5\}$ or $\{2,4,5,6\}$ are in principle valid, but $\{1,2,4,5\}$ and $\{2,3,4,5\}$ are not). After a cyclic rearrangement of the indices of $Q$, we may assume that the indices of these vertices are $\{1,2,3,5\}$. Therefore
    
    $$ [1,2,3] \simeq [1,3,5] \not \simeq [1,2,5] \simeq [2,3,5].$$
    
    In this case, the vertices $u_1$, $u_2$, $u_3$ and $u_5$ subdivide the sphere in four spherical regions (triangles). The vertices $u_4$ is such that its antipode must be in one of these four regions, i.e., $u_4$ must be in balanced position with three of the four other vertices. The only triple that works is $\{u_1,u_3,u_5\}$, since any other triple, being in balanced position with $u_4$, would then form an antipodal intersection. The fact of $\{u_1,u_3,u_4,u_5\}$ being in balanced position then implies that
    
    $$ [3,4,5] \simeq [1,3,5] \not \simeq [1,4,5] \simeq [1,3,4].$$

    Again, by the same argument as with the vertex $u_4$, $\{u_1,u_3,u_5,u_6\}$ is in balanced position. Therefore

    $$ [1,3,5] \simeq [1,5,6] \not \simeq [3,5,6] \simeq [1,3,6].$$

    Notice that the determinant $[1,3,5]$ appears in all cases. Suppose without loss of generality that $[1,3,5] > 0$. Hence
    \begin{itemize}
        \item $[1,3,5]$, $[1,2,3]$, $[3,4,5]$ and $[1,5,6]$ are positive;
        \item $[1,2,5]$, $[2,3,5]$, $[1,4,5]$, $[1,3,4]$, $[3,5,6]$ and $[1,3,6]$ are negative.
    \end{itemize}

    Since $[1,2,3]$, $[3,4,5]$ and $[1,5,6]$ are positive, we just need to show that $[2,3,4]$, $[4,5,6]$ and $[1,2,6]$ are negative.
    
    Suppose by contradiction that one of these determinants, say $[1,2,6]$, is positive. From this the following two facts follow:
    \begin{itemize}
        \item $[2,5,6]$ is positive. In fact, if $[2,5,6]$ were negative, then we would have
        $$ [1,2,5] \simeq [2,5,6] \not \simeq [1,2,6] \simeq [1,5,6],$$
        i.e., edges $\overrightarrow{u_1 u_2}$ and $\overrightarrow{u_5 u_6}$ would have a (usual) self-intersection, contrary to hypothesis.
        \item $[2,3,6]$ is positive. In fact, if $[2,3,6]$ were negative, then we would have
        $$ [1,2,3] \simeq [1,2,6] \not \simeq [2,3,6] \simeq [1,3,6],$$
        i.e.,
        $$ [2,3,6] \simeq [3,6,1] \not \simeq [2,3,1] \simeq [2,6,1],$$
        i.e., the edges $\overrightarrow{u_2 u_3}$ and $\overrightarrow{u_6 u_1}$ would have a (usual) self-intersection, contrary to hypothesis.
    \end{itemize}
    Now, since $[2,5,6]$ and $[2,3,6]$ are positive, we have that
    $$ [2,3,5] \simeq [3,5,6] \not \simeq [2,3,6] \simeq [2,5,6],$$
    i.e., the edges $\overrightarrow{u_2 u_3}$ and $\overrightarrow{u_5 u_6}$ have a (usual) self-intersection. But this contradicts our hypothesis on $Q$.
    
    Our conclusion then is that the determinant $[1,2,6]$ must be negative.
    
    Now, the fact that the determinant $[1,2,6]$ is negative actually implies that $[2,3,4]$ and $[4,5,6]$ are also negative when one considers a certain symmetry between the determinants of the six-vector configuration. Notice that the bijection of the sets of vertices of $Q$ into itself, defined by $u_i \mapsto u_{i+2} \mod 6$, preserves the sign of all determinants considered (for example, $[1,3,6] < 0$, and $[3,5,2] = [2,3,5] <0$). Geometrically, we are now looking to polygon $Q$ but with a cyclically different order: $Q'=[u_5,u_6,u_1,u_2,u_3,u_4]$. It follows that our previous argument (for the determinant $[1,2,6]$) works exactly the same but to the polygon $Q'$, implying that the determinant $[4,5,6]$ is negative. Applying this bijection one more time we can use the argument again but to polygon $Q''=[u_3,u_4,u_5,u_6,u_1,u_2]$, implying that the determinant $[2,3,4]$ is negative.
\end{proof}

Now, given a spherical polygon $Q$ in balanced position without self nor antipodal intersections, we must prove the existence of a vertex $v_i$ such that the resulting polygon $Q' = Q - v_i$ is also in balanced position and does not have self nor antipodal intersections. An idea to prove the existence of such vertex would be similar to the idea that we used to prove the existence of a nonessential, good vertex in a simple polygon in balanced position: recall that there are always $n-3$ nonessential vertices and $4$ good vertices, therefore there is at least one vertex in both subsets.

In our present case, however, there are three different features that the resulting polygon $Q'$ must have. Therefore a counting argument such as the one before would not necessarily work as easily now. A more hopeful strategy is to consider again two subsets of the vertices of $Q$:
\begin{itemize}
    \item vertices $v_i$ such that $Q'= Q - v_i$ is balanced, i.e., nonessential vertices;
    \item vertices $v_i$ such that $Q'= Q - v_i$ does not have self nor antipodal intersections. From now on, such vertices will be called \textit{excellent}.
\end{itemize}

We know already that the number of nonessential vertices is always greater or equal to $n-3$. And what about the minimal number of excellent vertices?

\begin{proposition}
\label{proposition_at_least_2_excellents}
    Let $Q=[u_1,...,u_n]$ ($n \geq 6$) be a balanced spherical polygon, without self nor antipodal intersections. Then $Q$ has at least 2 excellent vertices.
\end{proposition}
\begin{proof}
    Since $Q$ does not have self nor antipodal intersections, both $Q$ and its reflection $\overline{Q}$ through the origin subdivide the sphere into three different regions. Here it makes sense to talk about the \textit{interior} of $Q$: it is one of the two regions determined by $Q$ which does not contain the reflected polygon $\overline{Q}$.

    Since $Q$ is balanced, the interior of $Q$ can be triangulated with vertices of $Q$. Since the dual graph of this triangulation (i.e., the graph whose vertices are dual to the triangles of the triangulation and whose edges are dual to the edges used in the triangulation but not contained in the polygon $Q$) is a nontrivial tree, it has at least two leaves, i.e., two triangles whose set of vertices are of the form $\{u_{i-1},u_i,u_{i+1}\}$ and $\{u_{j-1},u_j,u_{j+1}\}$. Since both triangles $\triangle(u_{i-1},u_i,u_{i+1})$ and $\triangle(u_{j-1},u_j,u_{j+1})$ are entirely contained in the interior of $Q$, both edges $\overrightarrow{u_{i-1} u_{i+1}}$ and $\overrightarrow{u_{j-1} u_{j+1}}$ do not intersect any other edge of $Q$ nor $\overline{Q}$. Therefore vertices $u_i$ and $u_j$ are excellent, by definition.
\end{proof}

At first sight we could hope to improve the lower bound on the number of nonessential or excellent vertices. A reasonable guess would be, for instance, that:
\begin{itemize}
    \item the lower bound for nonessential vertices could be improved from $n-3$ to $n-2$;
    \item the lower bound for excellent vertices could be improved from $2$ to $3$.
\end{itemize}

In this case, by the same counting argument as before, we would find our nonessential, excellent vertex. There are, however, balanced polygons without self or antipodal intersections but with exactly $n-3$ nonessential vertices. And, among these latter polygons, some of them do not have more than $3$ excellent vertices. Therefore, we must proceed differently.

Denote by $Ess(Q)$ and $Exc(Q)$ the the number of vertices of $P$ that are essential and excellent, respectively. Let us first rephrase what we know:  $Ess(Q) \leq 3$ and $Exc(Q) \geq 2$. Our strategy will be as follows: besides the cases where $Ess(Q)$ equals $0$ or $1$ (where the existence of a nonessential and excellent vertex follows immediately from proposition \ref{proposition_at_least_2_excellents}), we must study separately the cases where $Ess(Q)=2$ and $Ess(Q)=3$ and prove, in each case, the existence of an excellent vertex among the nonessential ones.

For what follows, we will need the following definitions. A \textit{(closed) lune} of the sphere is the intersection of two closed hemispheres $H_1$ and $H_2$, each of them determined by distinct spherical lines (i.e., great circles) $l_1$ and $l_2$. The intersection $l_1 \cap l_2$ is a pair of antipodal points, which are called the \textit{cusps} of the lune. The intersection of each of the spherical lines $l_1$ and $l_2$ with the boundary of the lune are called the \textit{sides} of the lune.

Now, let $u,v,w \in \mathbb{S}^2$ be three noncollinear points (in the spherical sense). Consider the lune which has as cusps the vertices $u$ and $\overline{u}$ and whose sides pass through $v$ and $w$. We will denote such lune by $L(u;v,w)$.

\begin{figure}
    \centering
    \includegraphics[scale=0.3]{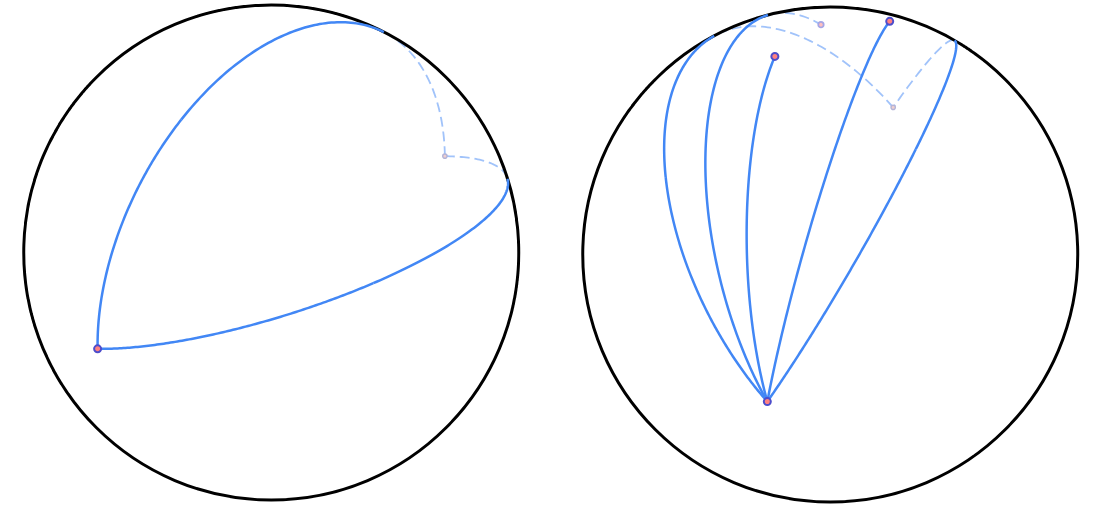}
    \caption{Examples of lunes. The spherical segment that connects any interior point of the lune to any of the two cusps is entirely contained in the closed lune.}
    \label{fig:lunes}
\end{figure}

\begin{lemma}
\label{lemma_lune}
    Let $L$ be a lune, $p$ any of its cusps and $q$ any point of the interior of the lune. Then the open segment $\overrightarrow{p q}$ is entirely contained in the interior of the lune.
\end{lemma}
\begin{proof}
    Let $l_1$ and $l_2$ be the two spherical lines that determine the lune. Denote by $\overline{p}$ the antipodal point to $p$. Consider the line $l$ spanned by the points $p$ and $q$. If the open segment $\overrightarrow{p q}$ were not contained entirely in the interior of the lune, then it would have to intersect one of the sides of the lune, and consequently intersect one of the spherical lines (say $l_1$), at a point distinct from $p$ and $\overline{p}$. But then the spherical lines $l$ and $l_1$ would intersect at three different points. This would imply that the lines $l$ and $l_1$ are the same, contrary to the hypothesis that $q$ is an interior point of the lune.
\end{proof}

\section{The case where $Ess(Q)=3$}

We will first treat the case where $Ess(Q)=3$:

\begin{proposition}
\label{proposition_essentials_3_vertex}
    Let $Q=[u_1,...,u_n]$ ($n \geq 7$) be a spherical polygon in balanced position and without self nor antipodal intersections, with $Ess(P)=3$. Then there is at least one nonessential, excellent vertex.
\end{proposition}

For the proof of \ref{proposition_essentials_3_vertex} we will need some lemmas. Since $Q$ is balanced, there are four vertices of $Q$ which are in balanced position. By hypothesis, exactly three of them are essential. Let $u_i$, $u_j$ and $u_k$ be these vertices (with $i<j<k$). Since all the remaining vertices are nonessential, they must all be contained in the triangular region spanned by $\overline{u}_i$, $\overline{u}_j$ and $\overline{u}_k$, i.e., the antipodes of the essential vertices. See figure \ref{fig:essentials_3_triangle}. Notice that this triangular region is the intersection of any two of the three following lunes (see figure \ref{fig:triangle_and_lunes} for an example):
\begin{itemize}
    \item the one whose cusps are $u_i$ and $\overline{u}_i$ and whose sides pass through $\overline{u}_j$ and $\overline{u}_k$, i.e., $L(u_i;\overline{u}_j,\overline{u}_k)$;
    \item the one whose cusps are $u_j$ and $\overline{u}_j$ and whose sides pass through $\overline{u}_k$ and $\overline{u}_i$, i.e., $L(u_j;\overline{u}_k,\overline{u}_i)$;
    \item and the one whose cusps are $u_k$ and $\overline{u}_k$ and whose sides pass through $\overline{u}_i$ and $\overline{u}_j$, i.e., $L(u_k;\overline{u}_i,\overline{u}_j)$.
\end{itemize}

\begin{figure}
    \centering
    \includegraphics[scale=0.3]{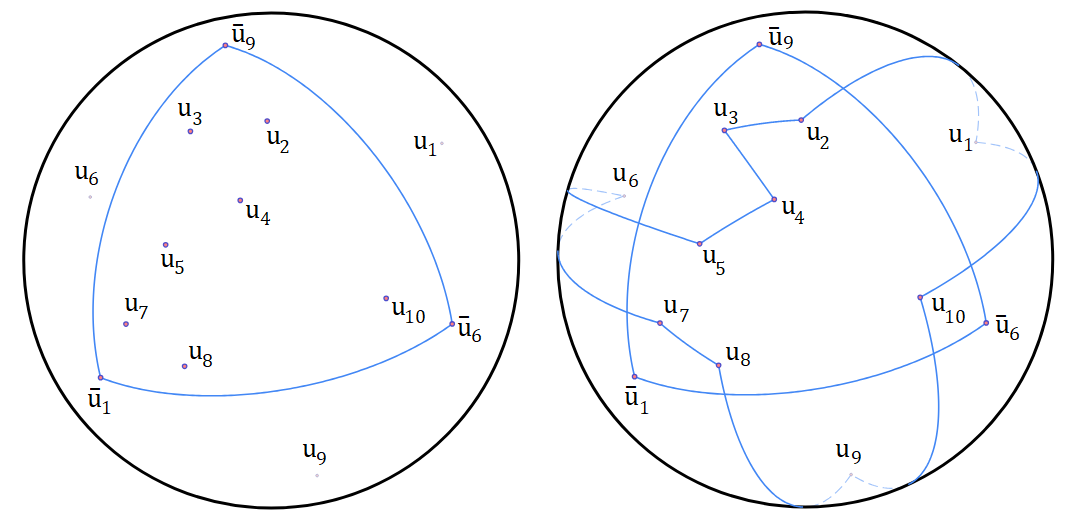}
    \caption{When $Ess(Q)=3$, all nonessential vertices are in the  triangular region spanned by the antipodes of the essential vertices.}
    \label{fig:essentials_3_triangle}
\end{figure}

\begin{figure}
    \centering
    \includegraphics[scale=0.2]{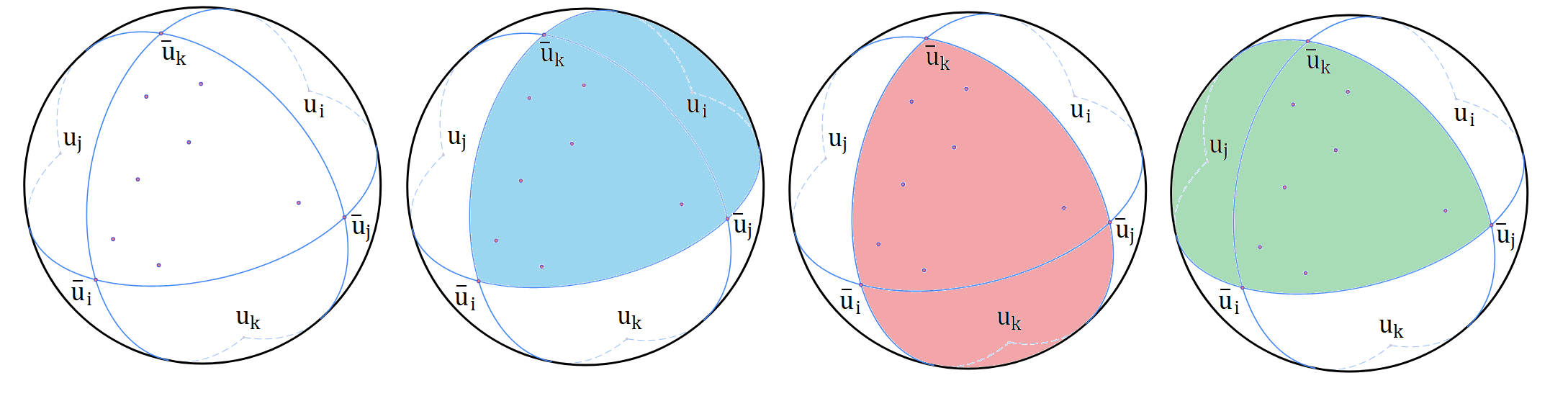}
    \caption{The triangle $\triangle(\overline{u}_i,\overline{u}_j,\overline{u}_k)$ as the intersection of the lunes $L(u_i;\overline{u}_j,\overline{u}_k)$ (in blue), $L(u_j;\overline{u}_k,\overline{u}_i)$ (in green) and $L(u_k;\overline{u}_j,\overline{u}_i)$ (in red).}
    \label{fig:triangle_and_lunes}
\end{figure}

\begin{lemma}
\label{lemma_essentials_3_isolated}
    Let $Q=[u_1,...,u_n]$ ($n \geq 6$) be a spherical polygon in balanced position and without self nor antipodal intersections, with $Ess(Q)=3$. Then the triple of essential vertices of $Q$ does not have a pair of consecutive vertices.
\end{lemma}
\begin{proof}
    If there were at least one pair of consecutive vertices among the essential ones, then we would have one of the two situations:
    \begin{itemize}
        \item two of them are consecutive and the other one is isolated. After a cyclic rearrangement, we can relabel these vertices as $u_1$, $u_2$ and $u_i$, where $i \neq 3,n$. Since $u_i$ must be connected to vertices $u_{i-1}$ and $u_{i+1}$, which are in the lune $L(u_i;\overline{u}_1,\overline{u}_2)$, the respective edges must then be contained in this same lune, by lemma \ref{lemma_lune}. This implies that these edges intersect the open segment $\overrightarrow{\overline{u}_1 \overline{u}_2}$, which means that $Q$ has two antipodal intersections, contrary to hypothesis (see figure \ref{fig:essential_3_impossible_cases} on the left for an example);
        \item the three vertices are consecutive. After a cyclic rearrangement, we can relabel them as $u_1$, $u_2$ and $u_3$. Since $u_3$ must be connected to vertex $u_4$, the corresponding edge must be (by lemma \ref{lemma_lune} again) entirely contained in the lune $L(u_3;\overline{u}_1,\overline{u}_2)$, and therefore intersects the open segment $\overrightarrow{\overline{u}_1 \overline{u}_2}$. An analogous argument also implies that the edge connecting vertices $u_1$ and $u_n$ intersects the open segment $\overrightarrow{\overline{u}_2 \overline{u}_3}$. In other words, $Q$ has two antipodal intersections, contrary to hypothesis (see figure \ref{fig:essential_3_impossible_cases} on the right for an example).
    \end{itemize}
\end{proof}

\begin{figure}
    \centering
    \includegraphics[scale=0.3]{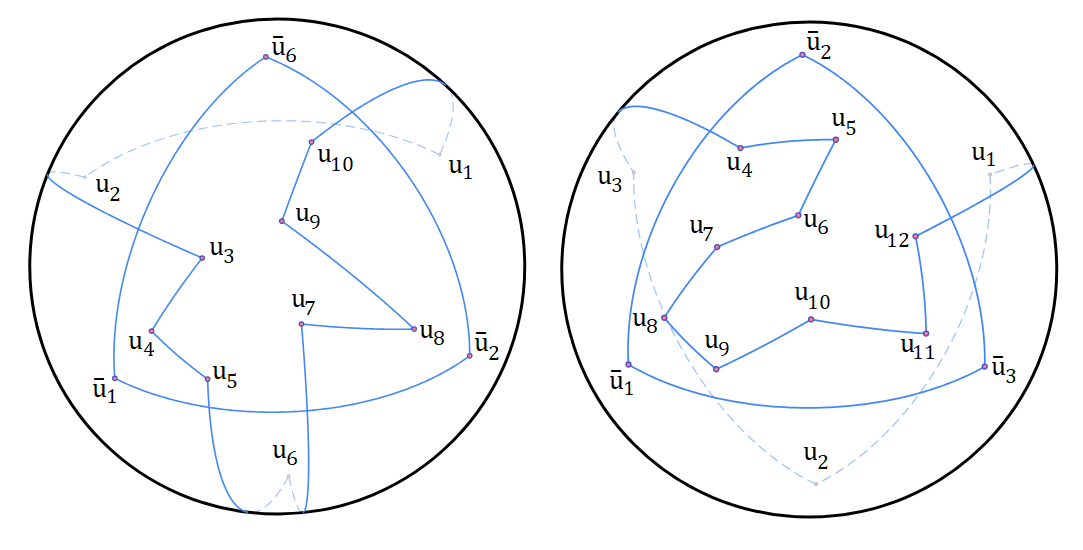}
    \caption{Spherical polygons where the triple of essential vertices always has at least a pair of consecutive vertices (on the left, the triple is $\{u_1,u_2,u_6\}$, while on the right it is $\{u_1,u_2,u_3\}$. In both examples, this implies at least two antipodal intersections.}
    \label{fig:essential_3_impossible_cases}
\end{figure}

Recall that the triangle $\triangle(\overline{u}_i,\overline{u}_j,\overline{u}_k)$ is the intersection of lunes $L(u_i;\overline{u}_j,\overline{u}_k)$, $L(u_j;\overline{u}_k,\overline{u}_i)$ and $L(u_k;\overline{u}_i,\overline{u}_j)$. Denote by $U(Q)$ the union of these three lunes. Notice that the boundary of $U(Q)$ is the polygon $[u_i,\overline{u}_k,u_j,\overline{u}_i,u_k,\overline{u}_j]$, which is symmetric. Therefore the reflection of $U(Q)$ is the union of other three lunes, namely, $L(\overline{u}_i;u_j,u_k)$, $L(\overline{u}_j;u_k,u_i)$ and $L(\overline{u}_k;u_i,u_j)$ (see figure \ref{fig:essentials_3_union_of_lunes} for an example).

\begin{figure}
    \centering
    \includegraphics[scale=0.2]{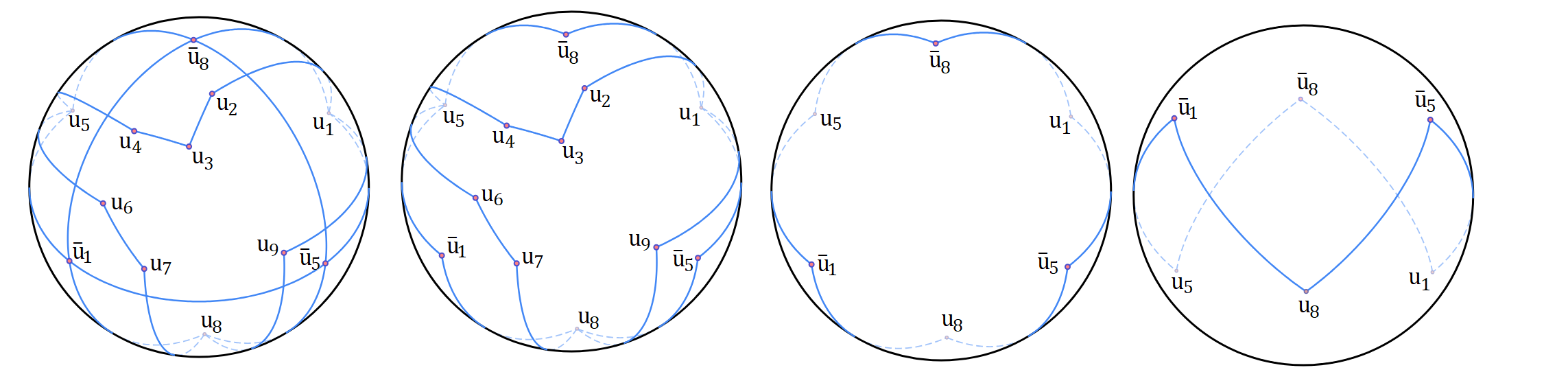}
    \caption{Given a polygon $Q$ with $Ess(Q)=3$, the region $U(Q)$ is the union of the lunes $L(u_i;\overline{u}_j,\overline{u}_k)$, $L(u_j;\overline{u}_k,\overline{u}_i)$ and $L(u_k;\overline{u}_i,\overline{u}_j)$.}
    \label{fig:essentials_3_union_of_lunes}
\end{figure}

\begin{lemma}
\label{lemma_polygon_in_union_of_lunes}
    Let $Q=[u_1,...,u_n]$ ($n \geq 7$) be a spherical polygon in balanced position and without self nor antipodal intersections, with $Ess(Q)=3$. Then
    \begin{itemize}
        \item[(a)] $Q \subset U(Q)$;
        \item[(b)] the vertices $u_i$,$u_j$ and $u_k$ are in the topological boundary of $U(Q)$, while all other points of $Q$ (including interior points of its edges) are in the topological interior of $U(Q)$.
    \end{itemize}
\end{lemma}
\begin{proof}
    (b) The first claim is immediate. Any other vertex of $Q$ is in the interior of $\triangle(\overline{u}_i,\overline{u}_j,\overline{u}_k)$ and hence in the interior of $U(Q)$. For an open edge $e$ of $Q$, it either:
    \begin{itemize}
        \item has as endpoints two nonessential vertices. Hence it is contained in the interior of $ \triangle(\overline{u}_i,\overline{u}_j,\overline{u}_k)$ and therefore in the interior of $U(Q)$;
        \item has as endpoints a nonessential vertex and a essential vertex. Hence it is, by lemma \ref{lemma_lune}, entirely contained in the interior of the corresponding lune and therefore in the interior of $U(Q)$.
    \end{itemize}
    (a) follows immediately from (b).
\end{proof}

Now we can finally prove proposition \ref{proposition_essentials_3_vertex}:
\begin{proof}
   (of proposition \ref{proposition_essentials_3_vertex}) By lemma \ref{lemma_essentials_3_isolated}, the three essential vertices are not consecutive. Denote these vertices by $u_i$, $u_j$ and $u_k$, where $i<j<k$ and $j \neq i+1$, $k \neq j+1$. In principle, the vertices which are connected to any of the essential vertices could be connected via a polygonal line to any of the others. In figure \ref{fig:essentials_3_pos_configurations} in the center we see a possible configuration, while in figure \ref{fig:essentials_3_imp_configurations} any of the connecting polygonal lines lead to the resulting polygon having either a self-intersection or being split into two different polygons. Therefore these vertices are connected in the following way: the vertex adjacent to $u_i$ which is nearest to $\overline{u}_k$ and the vertex adjacent to $u_j$ which is nearest to $\overline{u}_k$ will be connected to each other via a polygonal line; do the same but with indices $i$, $j$ and $k$ permuted (see figure \ref{fig:essentials_3_pos_configurations} right).

   \begin{figure}
    \centering
    \includegraphics[scale=0.2]{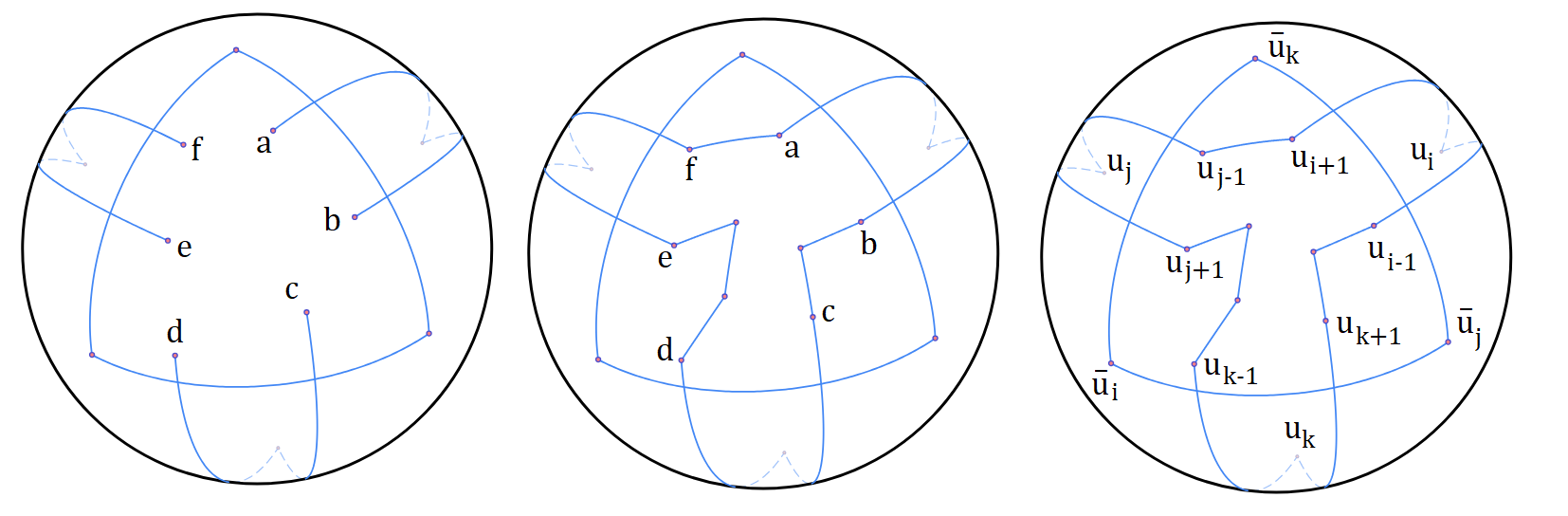}
    \caption{Given the vertices which are adjacent to the essential ones, the polygonal lines connecting them must be such that $a$ is connected to $f$, $e$ to $d$ and $c$ to $b$.}
    \label{fig:essentials_3_pos_configurations}
\end{figure}

\begin{figure}
    \centering
    \includegraphics[scale=0.2]{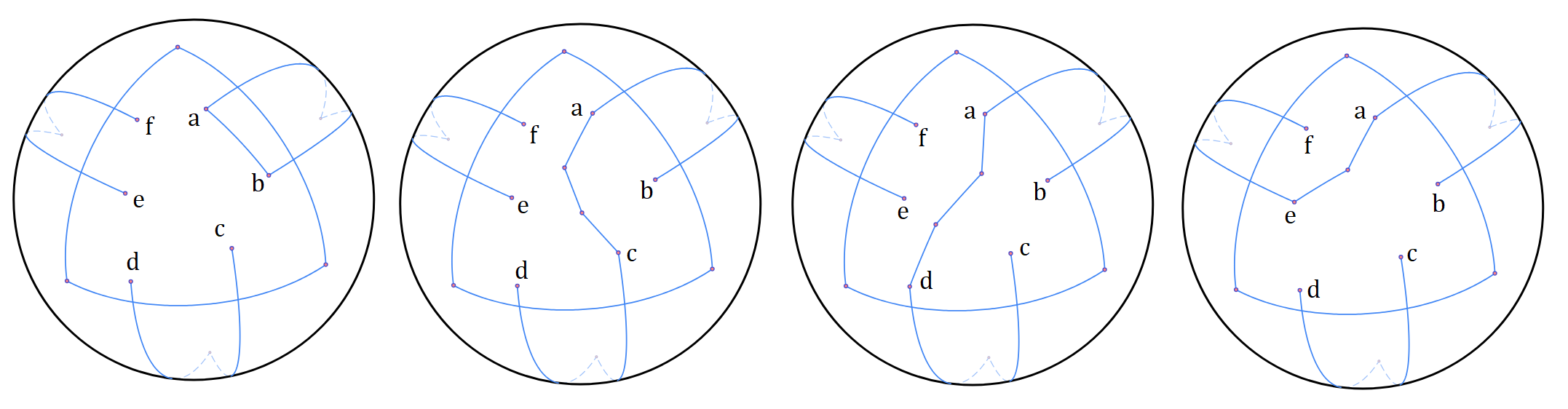}
    \caption{If vertex $a$ is connected via a polygonal line to vertex $b$, then $Q$ is split into two polygons. If $a$ is connected via a polygonal line to any of the vertices $c$, $d$ or $e$, then $Q$ will either have a self-intersection or becomes split into two polygons again.}
    \label{fig:essentials_3_imp_configurations}
\end{figure}

   It might happen that $u_{i+1}=u_{j-1}$, $u_{j+1}=u_{k-1}$ or $u_{k+1}=u_{i-1}$. However, since $Q$ has at least $7$ vertices, there is at least one pair of consecutive vertices inside the spherical triangle $\triangle(\overline{u}_i,\overline{u}_j,\overline{u}_k)$. We might assume without loss of generality that $u_{i+1}$ and $u_{j-1}$ are distinct and that therefore the polygonal line with vertices $u_i$, $u_{i+1}$, ..., $u_{j-1}$ and $u_j$ has at least four vertices.
   
   Consider then the (closed) polygon $\Tilde{Q}=[u_i,u_{i+1},...,u_{j-1},u_j]$ and the region $R$ enclosed by it which contains the vertex $\overline{u}_k$. Since $\Tilde{Q}$ has at least $4$ vertices, any triangulation of $R$ has at least two triangles. The dual graph of this triangulation is then a nontrivial tree and has therefore at least two leaves $\triangle_1$ and $\triangle_2$. Because all vertices $u_{i+1}$,...,$u_{j-1}$ are in the same triangular region, all the edges added to form the triangulation must be contained in one of the two lunes (if its endpoints are nonessential, it is inside the triangular region by convexity; if one of the endpoints is essential, the claim follows from lemma \ref{lemma_lune}) and therefore the only triangle that contains the vertex $\overline{u}_k$ is the one with side $\overrightarrow{u_j u_i}$, while the other triangles are inside of one of the two lunes. Therefore at least one leaf of this triangulation (say $\triangle_1$) has two of its sides as edges of the original polygon $Q$. Figures \ref{fig:essentials_3_example_1_triangulation} and \ref{fig:essentials_3_example_2_triangulation} show some examples.

   \begin{figure}
       \centering
       \includegraphics[scale=0.17]{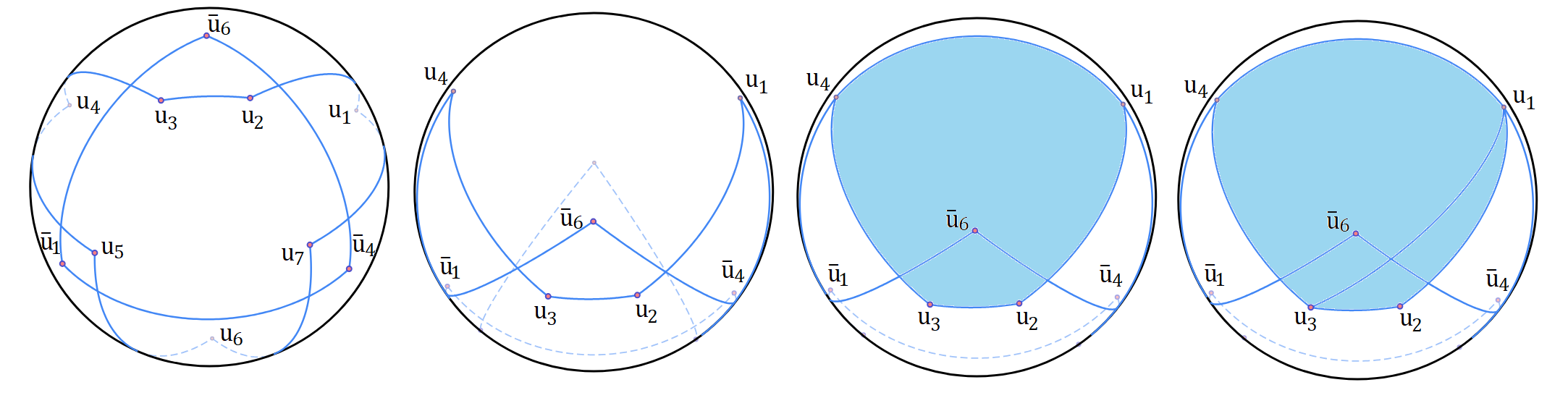}
       \caption{A polygon $Q$ with $7$ vertices and its corresponding polygon $\Tilde{Q}$ with $4$ vertices: in the given triangulation we have $u_2$ as an excellent vertex.}
       \label{fig:essentials_3_example_1_triangulation}
   \end{figure}

   \begin{figure}
       \centering
       \includegraphics[scale=0.17]{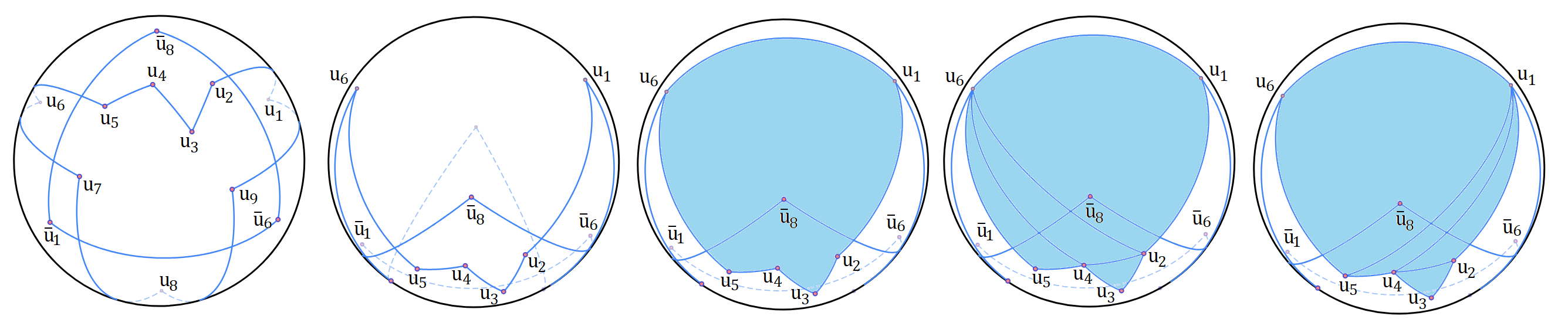}
       \caption{A polygon $Q$ with $9$ vertices and its corresponding polygon $\Tilde{Q}$ with $6$ vertices. Here we depict two different ways of triangulating the region $R$: in the first one we have $u_3$ and $u_5$ as excellent vertices, while in the second one we have only $u_3$ as an excellent vertex.}
       \label{fig:essentials_3_example_2_triangulation}
   \end{figure}

   We claim that the intermediate vertex of $\triangle_1$ (i.e., the vertex of $\triangle_1$ connected to the edges of $Q$) is an excellent vertex of $Q$. Consider the side of $\triangle_1$ that is not an edge of $Q$ and denote its open segment by $s$. It is clear that $s$ does not intersect any other edge of $Q$. Now, $s$ has as endpoints either:
   \begin{itemize}
       \item two nonessential vertices of $Q$ and therefore is entirely contained in the interior of $U(Q)$, by the same argument in the proof of lemma \ref{lemma_polygon_in_union_of_lunes}(b);
       \item a nonessential vertex and an essential vertex of $Q$, in which case the same edge must be in the interior of $U(Q)$, again by the same argument in the proof of lemma \ref{lemma_polygon_in_union_of_lunes}(b).
   \end{itemize}
   Since $\overline{Q}$ satisfies the same hypotheses of $Q$ (since it is its reflected image), lemma \ref{lemma_polygon_in_union_of_lunes}(a) applied to $\overline{Q}$ implies that $\overline{Q}$ is contained in $U(\overline{Q})$ (i.e., the reflection of $U(Q)$), which is disjoint from $int \text{ } U(Q)$. Therefore $s$ cannot intersect any edge of $\overline{Q}$.

   Since $s$ does not intersect no other edge of $Q$ or $\overline{Q}$, the intermediate vertex of $\triangle_1$ is excellent. Because this vertex is not one of the essential ones, the proposition is proved.
\end{proof}

\section{The case where $Ess(Q)=2$}
\label{case_ess_q_2}

Now we turn our attention to balanced spherical polygons without self nor antipodal intersections with $Ess(Q)=2$. The proof of the existence of a nonessential, excellent vertex turns out to be more involved in this case.

\begin{proposition}
\label{proposition_essentials_2_vertex}
    Let $Q=[u_1,...,u_n]$ ($n \geq 7$) be a spherical polygon in balanced position and without self nor antipodal intersections, with $Ess(P)=2$. Then there is at least one nonessential, excellent vertex.
\end{proposition}

As it was the case for proposition \ref{proposition_essentials_3_vertex}, we will also need some lemmas for the proof of proposition \ref{proposition_essentials_2_vertex}. Since $Q$ is balanced, there are four vertices of $Q$ in balanced position. By hypothesis, exactly two of them are essential. Let $u_i$ and $u_j$ be the essential vertices (with $i<j$) and $u_k$ and $u_l$ be nonessential ones among these four vertices (the indices $k$ and $l$ do not hold any specific position with respect to $i$ and $j$). All the nonessential vertices must be contained in the triangular regions $\triangle(\overline{u_i},\overline{u_j},\overline{u_k})$ and $\triangle(\overline{u_i},\overline{u_j},\overline{u_l})$, and each of these regions must have at least two vertices in its interior (for if one of them had only one, then this sole vertex would be a third essential vertex, contradicting our assumption). Figure \ref{fig:essentials_2_examples} shows some examples. 

Moreover, by the next lemma, the spherical convex hull of the set of nonessential vertices must be contained in the union of these two triangular regions. Figure \ref{fig:essentials_2_valid_convex_hull} shows a polygon for which this latter property holds, while figure \ref{fig:essentials_2_invalid_convex_hull} shows a polygon for which this property does not hold.

\begin{lemma}
\label{lemma_nonessentials_are_in_union_of_triangles}
    Let $Q=[u_1,...,u_n]$ ($n \geq 7$) be a spherical polygon in balanced position and without self nor antipodal intersections, with $Ess(P)=2$. Then the spherical convex hull of the set of nonessential vertices must be contained in the union of $\triangle(\overline{u_i},\overline{u_j},\overline{u_k})$ and $\triangle(\overline{u_i},\overline{u_j},\overline{u_l})$.

    In particular, the spherical convex hull of the nonessential vertices does not contain the antipodal points of the essential vertices.
\end{lemma}
\begin{proof}
    If it were not the case, then the spherical convex hull of the nonessential vertices would intersect one of the following open spherical segments: $\overrightarrow{\overline{u}_i \overline{u}_k}$, $\overrightarrow{\overline{u}_k \overline{u}_j}$, $\overrightarrow{\overline{u}_j \overline{u}_l}$ or $\overrightarrow{\overline{u}_l \overline{u}_i}$. But this would mean that a pair of nonessential vertices intersects one of these segments, say $\overrightarrow{\overline{u}_j \overline{u}_l}$, which implies that these two nonessential vertices, together with $u_j$ and $u_l$, are in balanced position. In other words, $u_i$ would be nonessential, contrary to hypothesis.
\end{proof}

\begin{figure}
    \centering
    \includegraphics[scale=0.25]{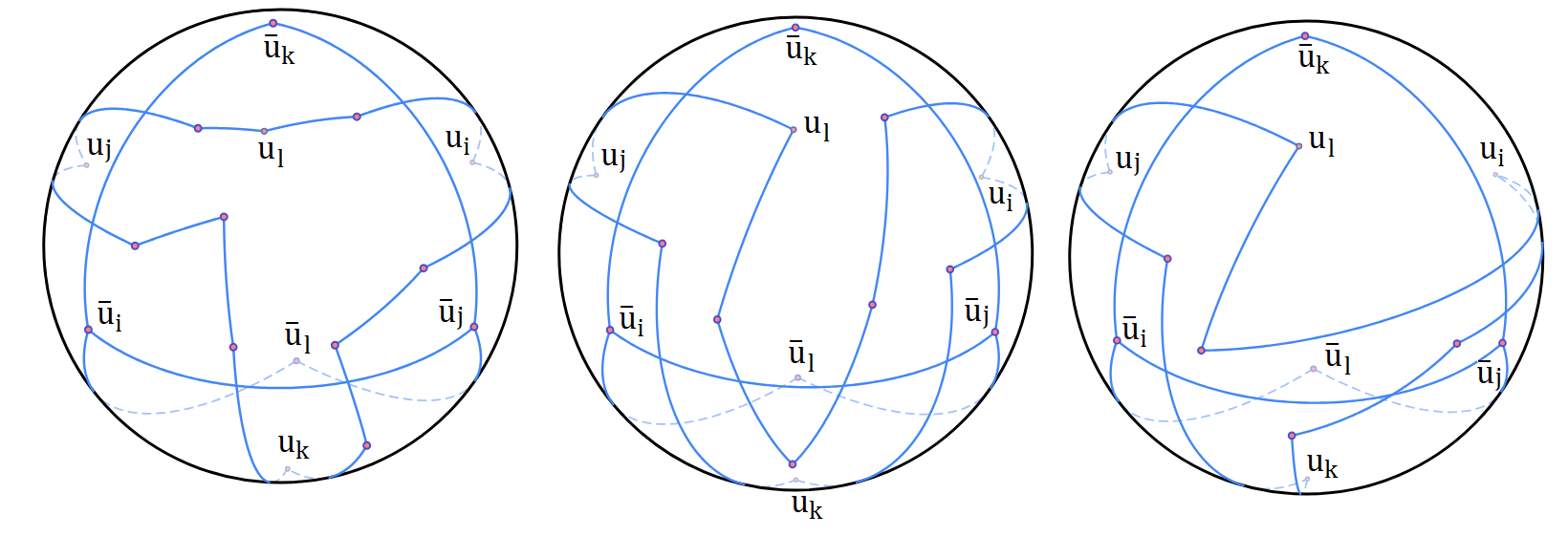}
    \caption{Examples of polygons for which the number of essential vertices is exactly $2$.}
    \label{fig:essentials_2_examples}
\end{figure}

\begin{figure}
    \centering
    \includegraphics[scale=0.25]{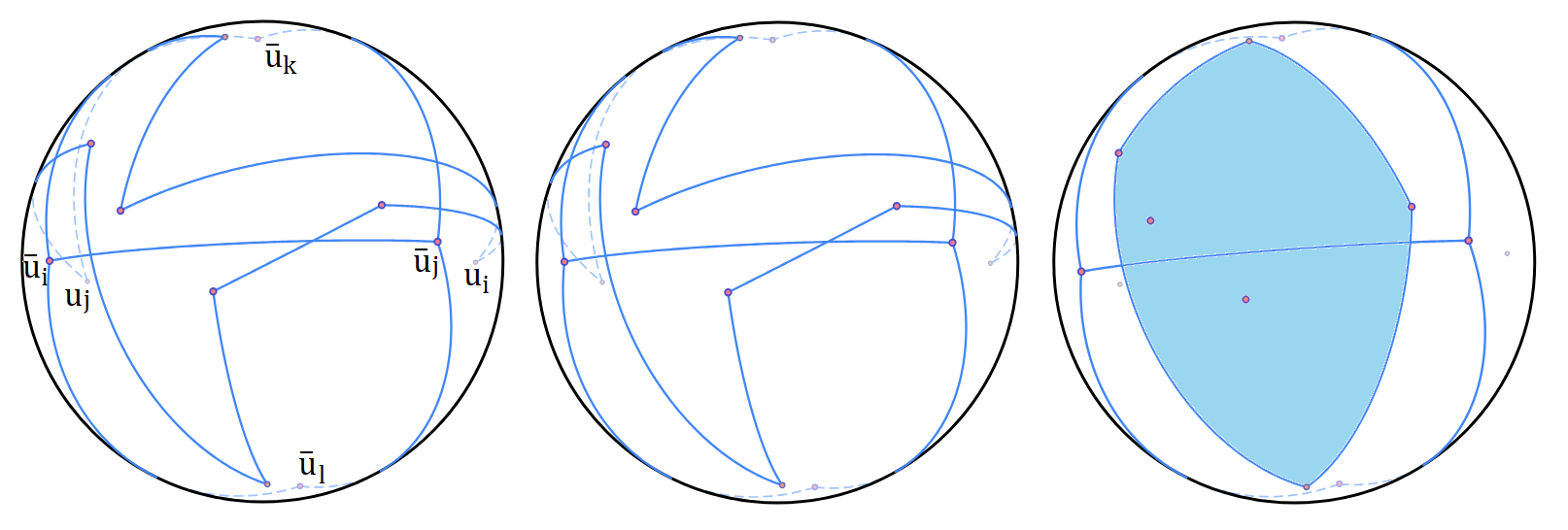}
    \caption{Valid spherical convex hull}
    \label{fig:essentials_2_valid_convex_hull}
\end{figure}

\begin{figure}
    \centering
    \includegraphics[scale=0.25]{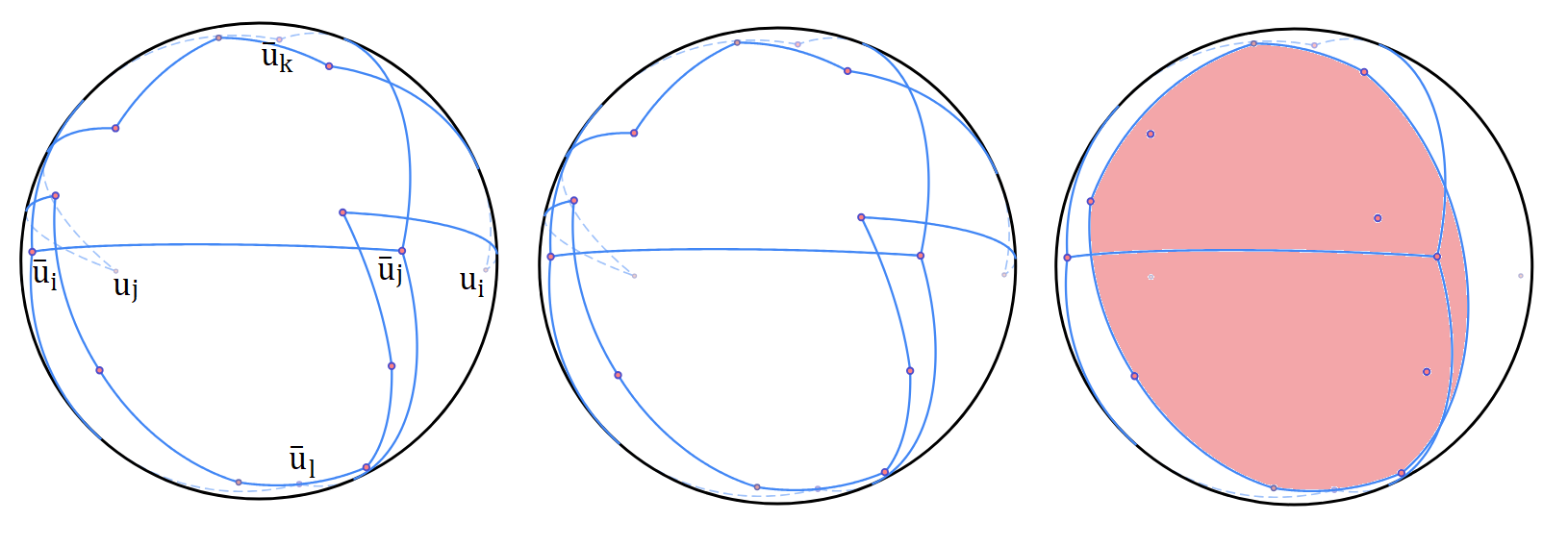}
    \caption{Invalid spherical convex hull}
    \label{fig:essentials_2_invalid_convex_hull}
\end{figure}

As in the case where $Ess(Q)=3$, the essential vertices cannot be consecutive:

\begin{lemma}
    Let $Q=[u_1,...,u_n]$ ($n \geq 6$) be a spherical polygon in balanced position and without self nor antipodal intersections, with $Ess(P)=2$. Then the essential vertices cannot be consecutive.
\end{lemma}
\begin{proof}
    Suppose that the essential vertices $u_i$ and $u_j$ are consecutive. After a cyclic rearrangement, we can relabel them as $u_1$ and $u_2$. Let $u_k$ and $u_l$ be nonessential vertices of $Q$ in balanced position with $u_1$ and $u_2$. There are two possibilities (see figure \ref{fig:essentials_2_impossible_cases}):
    \begin{itemize}
        \item $u_3$ and $u_n$ are in the same triangle. Without loss of generality we can consider this triangle to be $\triangle(\overline{u}_1,\overline{u}_2,\overline{u}_k)$. Since polygon $Q$ also has vertices in the triangular region $\triangle(\overline{u}_1,\overline{u}_2,\overline{u}_l)$ and the spherical convex hull of the nonessential vertices is contained in the union of the triangular regions, it must cross the edge segment $\overrightarrow{\overline{u}_1 \overline{u}_2}$ at least twice, i.e., it has at least two antipodal intersections, contrary to hypothesis (see figure \ref{fig:essentials_2_impossible_cases} on the left for an example);
        \item $u_3$ and $u_n$ are in different triangles. Since polygon $Q$ must close itself and the spherical convex hull of the nonessential vertices is contained in the union of the triangular regions, it must cross the edge segment $\overrightarrow{\overline{u}_1 \overline{u}_2}$ at least once, i.e., it has at least one antipodal intersection, contrary to hypothesis (see figure \ref{fig:essentials_2_impossible_cases} on the right for an example).
    \end{itemize}
\end{proof}

\begin{figure}
    \centering
    \includegraphics[scale=0.3]{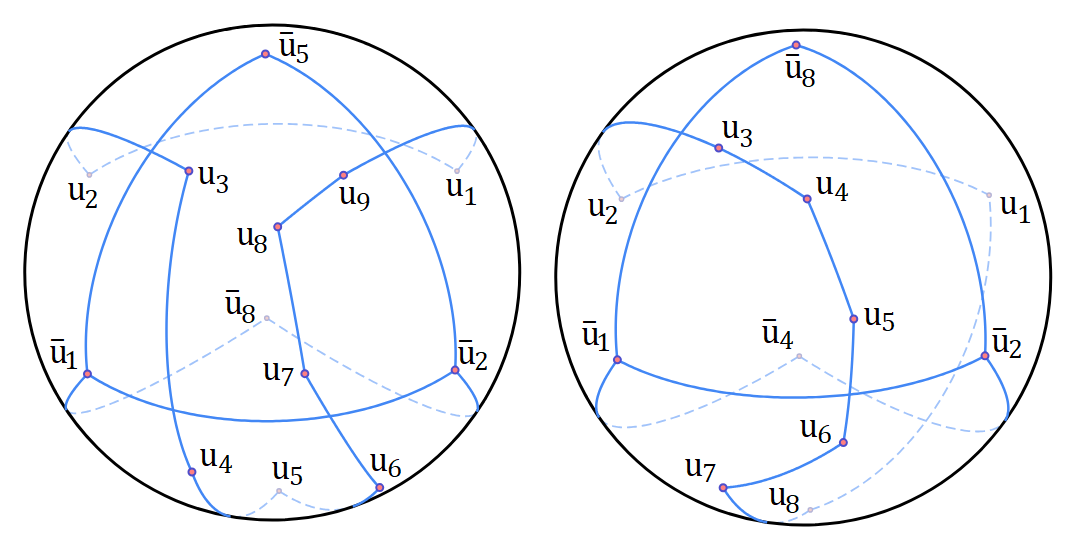}
    \caption{If the only $2$ essential vertices were consecutive, $Q$ would necessarily have at least one antipodal intersection.}
    \label{fig:essentials_2_impossible_cases}
\end{figure}

In the different examples of figure \ref{fig:essentials_2_examples} we see that the vertices of $Q$ connected to $u_i$ (i.e., $u_{i-1}$ and $u_{i+1}$) or $u_j$ (i.e., $u_{j-1}$ and $u_{j+1}$) are both in the same triangle. We might wonder if it is possible for a polygon satisfying our hypothesis to have these four (or three, if $u_{i+1}=u_{j-1}$) vertices in different triangles. The next lemmas show that this cannot happen:

\begin{lemma}
\label{lemma_each_pair_same_triangle}
    Let $Q=[u_1,...,u_n]$ ($n \geq 6$) be a spherical polygon in balanced position and without self nor antipodal intersections, with $Ess(P)=2$. Let $u_i$ and $u_j$ be the essential vertices of $Q$. Then:
    \begin{itemize}
        \item[(a)] $u_{i-1}$ and $u_{i+1}$ are in the same triangular region.
        \item[(b)] $u_{j-1}$ and $u_{j+1}$ are in the same triangular region.
    \end{itemize}
\end{lemma}
\begin{proof}
    (a) First we prove that $u_{i-1}$ and $u_{i+1}$ cannot be in different triangular regions.
    
    Suppose by contradiction that one of them is in $\triangle(\overline{u}_i,\overline{u}_j,\overline{u}_l)$ while the other one is in $\triangle(\overline{u}_i,\overline{u}_j,\overline{u}_k)$. Since the spherical convex hull of the nonessential vertices of $Q$ does not cross edge segments $\overrightarrow{\overline{u}_j \overline{u}_l}$ and $\overrightarrow{\overline{u}_j \overline{u}_k}$ and moreover $Q$ must close itself, the vertex $\overline{u}_j$ would then be separated by $Q$ from vertices $\overline{u}_k$ and $\overline{u}_l$, i.e., they would be in different regions delimited by the polygon $Q$. But because $\overline{Q}$ is a continuous curve and must connect $\overline{u}_j$ to $\overline{u}_k$ and to $\overline{u}_l$ via polygonal lines, it would necessarily intersect polygon $Q$ at some edge, contrary to hypothesis (see figure \ref{fig:essentials_2_same_triangle_each_pair} for some examples).

    (b) The proof that vertices $u_{j-1}$ and $u_{j+1}$ are in the same triangular region is analogous to the proof of (a).
\end{proof}

\begin{figure}
    \centering
    \includegraphics[scale=0.3]{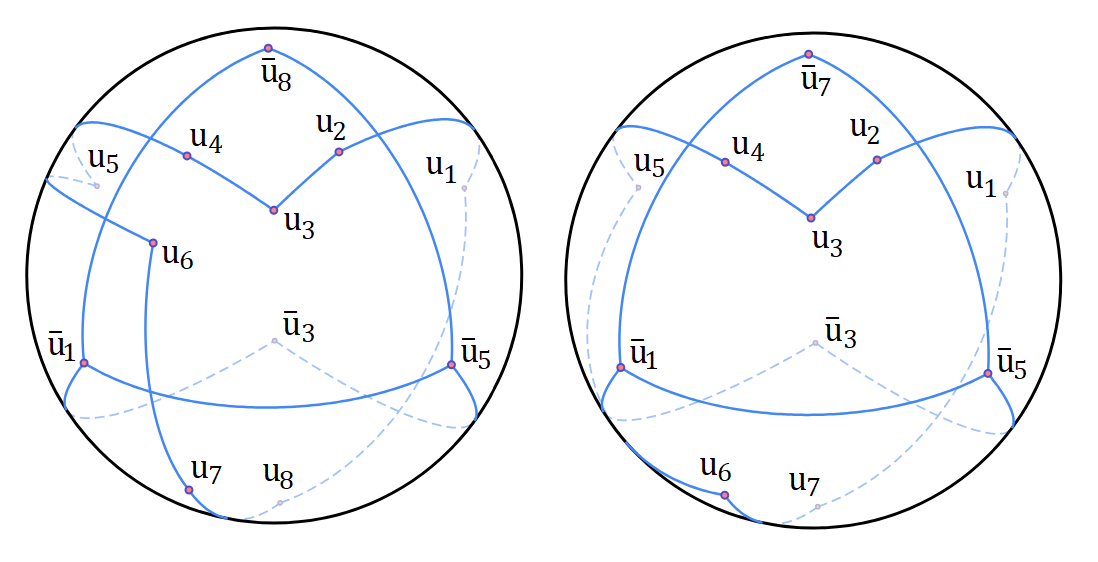}
    \caption{In both examples, $u_i=u_1$ and $u_j=u_5$. If $u_{i-1}$ and $u_{i+1}$ were in different triangles, vertex $\overline{u}_j$ would be isolated from vertices $\overline{u}_k$ and $\overline{u}_l$.}
    \label{fig:essentials_2_same_triangle_each_pair}
\end{figure}

\begin{lemma}
    Let $Q=[u_1,...,u_n]$ ($n \geq 6$) be a spherical polygon in balanced position and without self nor antipodal intersections, with $Ess(P)=2$. Let $u_i$ and $u_j$ be the essential vertices of $Q$. Then $u_{i-1}$, $u_{i+1}$, $u_{j-1}$ and $u_{j+1}$ are in the same triangular region.
\end{lemma}
\begin{proof}
    Given lemma \ref{lemma_each_pair_same_triangle}, it suffices to prove that the pair $\{u_{i-1},u_{i+1}\}$ is in the same triangular region as $\{u_{j-1},u_{j+1}\}$.

    Suppose by contradiction that $u_{i-1}$ and $u_{i+1}$ are in $\triangle(\overline{u}_i,\overline{u}_j,\overline{u}_l)$ while $u_{j-1}$ and $u_{j+1}$ are in $\triangle(\overline{u}_i,\overline{u}_j,\overline{u}_k)$. In particular, there are two polygonal lines from $u_i$ to $u_k$, each of them crossing the edge segment $\overrightarrow{ \overline{u}_l \overline{u}_j}$ only once (see figure \ref{fig:polygonal_line_lemma_1} for two different examples). Consider the polygonal line not passing through $u_j$ and denote it by $pl_{i,k}=pl(u_i,...,u_k)$. Similarly, there are two polygonal lines from $u_j$ to $u_l$, each of them crossing the edge segment $\overrightarrow{\overline{u}_i \overline{u}_k}$ only once. Consider the antipodal polygonal line of the one not passing through $u_i$ and denote it by $\overline{pl}_{j,l}= pl(\overline{u}_j,...,\overline{u}_l)$.

    We know that $pl_{i,k}$ intersects the segment $\overrightarrow{\overline{u}_l \overline{u}_j}$ once and that $\overline{pl}_{j,l}$ intersects the segment $\overrightarrow{u_i u_k}$ once. We shall prove that $pl_{i,k}$ intersects $\overline{pl}_{j,l}$, which contradicts the hypothesis of $Q$ not having antipodal intersections.

    Concatenating the polygonal line $\overline{pl}_{j,l}$ and the segment $\overrightarrow{\overline{u}_l \overline{u}_j}$ one forms a (closed) polygon $\overline{p}_{j,l}:=[\overline{u}_j,...,\overline{u}_l]$ which intersects the segment $\overrightarrow{u_i u_k}$ twice (see figure \ref{fig:polygonal_line_lemma_2}.a).
    
    The polygon $\overline{p}_{j,l}$ is a closed continuous curve which subdivides the sphere in at least two regions (more than two if it has self-intersections). Since $\overline{p}_{j,l}$ crosses segment $\overrightarrow{u_i u_k}$ twice and, except by $\overline{u}_j$ and the edges adjacent to $\overline{u}_j$, all the polygon $\overline{p}_{j,l}$ is contained in the two triangular regions, both vertices $u_i$ and $u_k$ must in the same region among the regions determined by $\overline{p}_{j,l}$.

    On the other hand, the polygonal line $pl_{i,k}$ must cross segment $\overrightarrow{\overline{u}_l,\overline{u}_j}$ once, say at a point $q$. If this point happens to be also in the polygonal line $\overline{p}_{j,l}$, we are done (see figure \ref{fig:polygonal_line_lemma_2}(b)). Suppose then that $q$ is not in $\overline{p}_{j,l}$. Since we are dealing with segments and there are no three collinear vertices, all intersections between segments are \textit{transversal}, which implies in our case that at the point of intersection $q$ the curve $pl_{i,k}$ is locally separated by segment $\overrightarrow{\overline{u}_l \overline{u}_j}$. In other words, the curve $pl_{i,k}$ passes in at least two different regions determined by $\overline{p}_{j,l}$: the one that contains vertices $u_i$ and $u_k$, and another one denoted by $R$ (see figure \ref{fig:polygonal_line_lemma_2}(c) and (d)). Therefore $pl_{i,k}$ must intersect another boundary point $q'$ of $R$. Since $q'$ cannot be in $\overrightarrow{\overline{u}_l \overline{u}_j}$ (by hypothesis), it must be a point of $\overline{pl}_{j,l}$.

    Similarly, assuming that $u_{i-1}$ and $u_{i+1}$ are in $\triangle(\overline{u}_i,\overline{u}_j,\overline{u}_k)$ while $u_{j-1}$ and $u_{j+1}$ are in $\triangle(\overline{u}_i,\overline{u}_j,\overline{u}_l)$, we derive another contradiction.
\end{proof}

\begin{figure}
    \centering
    \includegraphics[scale=0.3]{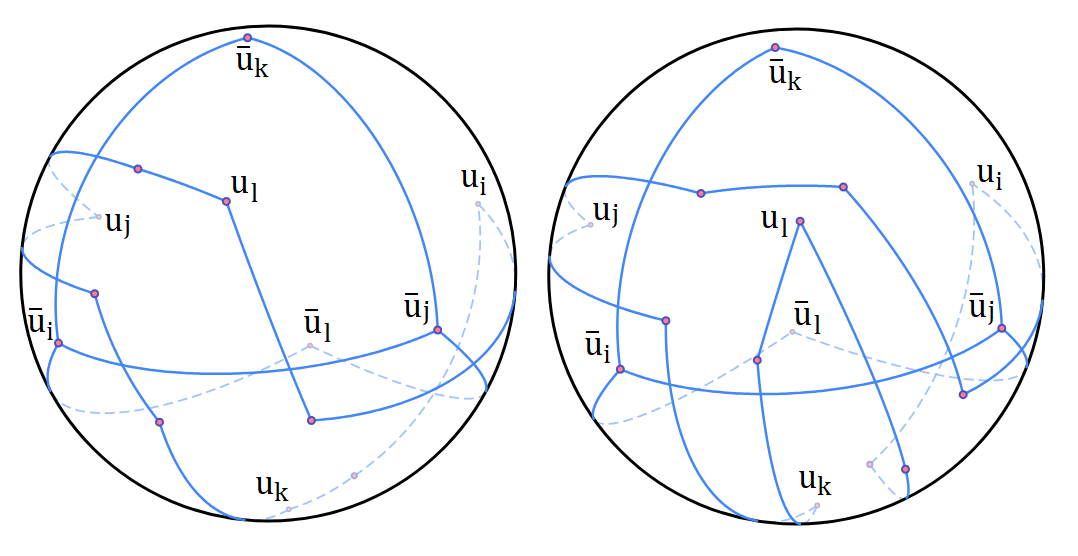}
    \caption{Two examples where $u_{i-1}$ and $u_{i+1}$ are in $\triangle(\overline{u}_i,\overline{u}_j,\overline{u}_l)$ while $u_{j-1}$ and $u_{j+1}$ are in $\triangle(\overline{u}_i,\overline{u}_j,\overline{u}_k)$.}
    \label{fig:polygonal_line_lemma_1}
\end{figure}

\begin{figure}
    \centering
    \includegraphics[scale=0.23]{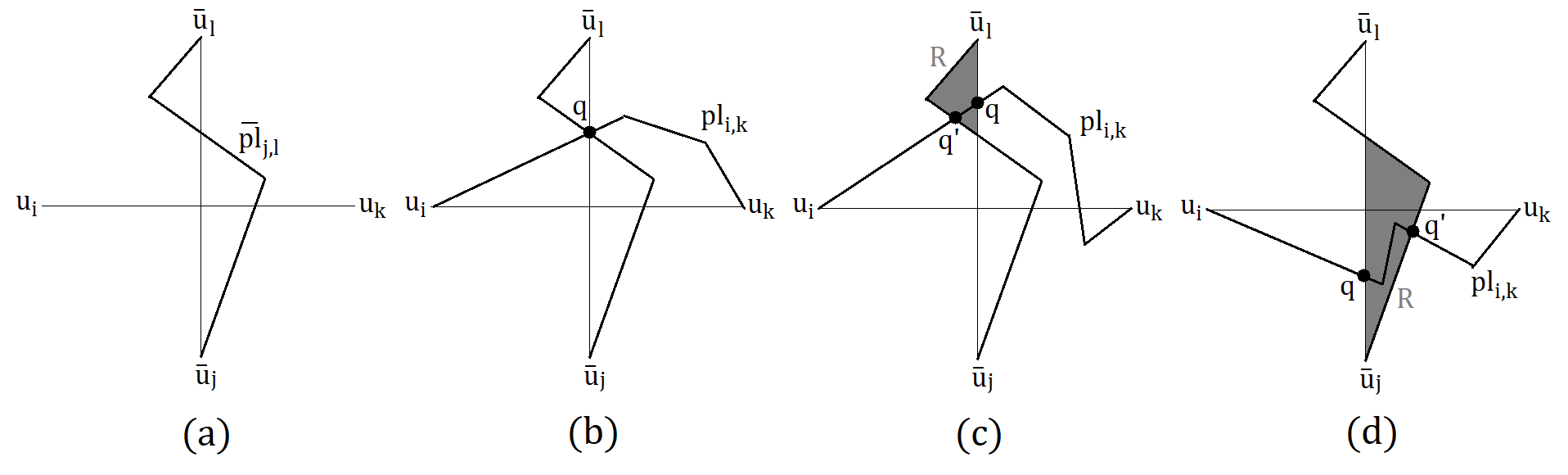}
    \caption{In (a), we see that the polygon $\overline{p}_{j,l}$ subdivides the sphere into different regions, one of them containing both vertices $u_i$ and $u_k$ in its interior. Figures (b), (c) and (d) show different possibilities regarding the polygonal line $pl_{i,k}$ and its intersections with polygon $\overline{p}_{j,l}$.}
    \label{fig:polygonal_line_lemma_2}
\end{figure}

Here is an idea to prove proposition \ref{proposition_essentials_2_vertex}: we have a geometrically similar situation as we had in the case where $Ess(Q)=3$: although two triangular regions contain more than one vertex, one of them, say $\triangle(\overline{u}_i,\overline{u}_j,\overline{u}_k)$ has ``more parts of $Q$" in the sense that the polygon $Q$ enters and exits this region more times (more precisely: the number of connected components of $Q \cap \triangle(\overline{u}_i,\overline{u}_j,\overline{u}_k)$ is equal the number of connected components of $Q \cap \triangle(\overline{u}_i,\overline{u}_j,\overline{u}_l)$ plus $2$). At first we could try to use the same argument as in proposition \ref{proposition_essentials_3_vertex}: we could triangulate a certain region enclosed by a polygon whose vertices are some essential and nonessential vertices of $Q$, and then obtain a leaf of the triangulation.

The problem with this approach in the case of $Ess(Q)=2$ is that it depends for the polygon $Q$ to have at least two consecutive vertices in $\triangle(\overline{u}_i,\overline{u}_j,\overline{u}_k)$, which might not happen here. If we try to triangulate the corresponding region as in figure \ref{fig:failed_triangulation}, all extra edges arising from the triangulation will intersect the reflected polygon $\overline{Q}$.

\begin{figure}
    \centering
    \includegraphics[scale=0.2]{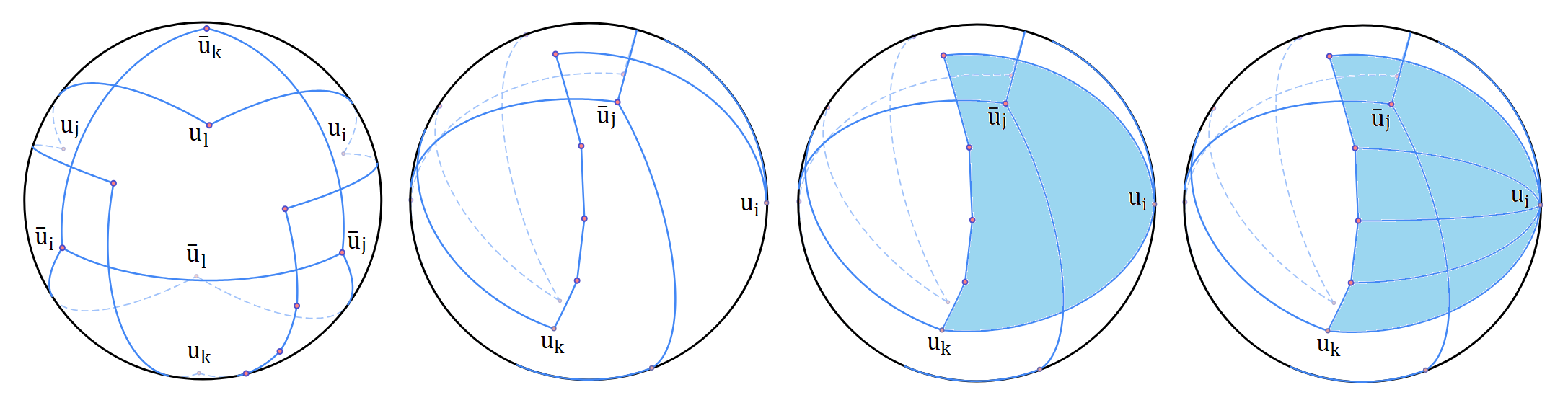}
    \caption{In this polygon $Q$, only one vertex between $u_k$ and $u_i$ is in $\triangle(\overline{u}_i,\overline{u}_j,\overline{u}_k)$. Consider the polygonal line from $u_k$ to $u_i$ and concatenate it with segment $\overrightarrow{u_i u_k}$, obtaining polygon $p_{k,i}$. Consider now the region determined by $p_{k,i}$ which contains $\overline{u}_j$. The unique triangulation on it has as leaves: the triangle whose side is the added segment $\overrightarrow{u_i u_k}$ (which does not have three consecutive vertices of the original polygon $Q$); and a triangle with $\overline{u}_j$ in its interior (which therefore will intersect the polygon $\overline{Q}$).}
    \label{fig:failed_triangulation}
\end{figure}

Therefore, in order to prove proposition \ref{proposition_essentials_2_vertex}, we must choose carefully the region that we want to triangulate. First we need the following terminology. Given a balanced spherical polygon $Q$ without self nor antipodal intersections, $Q$ separates the sphere into two regions: we say that the region which contains $\overline{Q}$ is the \textit{exterior of $Q$}, (which we denote it by $Ext(Q)$), while the other one is the \textit{interior of $Q$} (which we denote by $Int(Q)$). First notice that a striking difference between the triangulations used in the proofs of proposition \ref{proposition_at_least_2_excellents} and of proposition \ref{proposition_essentials_3_vertex} is that in the former the \textit{interior} of $Q$ is triangulated, while in the latter a subset of the \textit{exterior} of $Q$ is triangulated. Of course, in the second proof we also had to show that the one of the leaves of the triangulation did not intersect polygon $\overline{Q}$, which was true due to the fact that the corresponding triangle was contained in $U(Q)$ and that $\overline{Q}$ was contained in $U(\overline{Q})$ (recall that it only makes sense to speak of the subset $U(Q)$ when $Ess(Q)=3$).

In the case where $Ess(Q)=2$, the excellent (and nonessential) vertex to be found might come from either an internal or an external triangulation. By proposition \ref{proposition_at_least_2_excellents} and its proof, $Q$ has at least $2$ excellent vertices, the two of them coming from a triangulation of the interior of $Q$. If one of them is nonessential, we are done. Therefore we might assume that both of them are the essential vertices.

Denote by $\mathcal{H}_{S}(X)$ the spherical convex hull of a subset $X$ of the sphere, provided $X$ is contained in a closed hemisphere. Let $u_i$ and $u_j$ be the essential vertices of $Q$. In order to find an excellent vertex among the nonessential ones, we want to look at the new (simple) polygon $\Tilde{Q}:=Q - u_i - u_j$ (i.e., the polygon obtained deleting vertices $u_i$ and $u_j$ and the edges adjacent to these two vertices, and then connecting $u_{i-1}$ to $u_{i+1}$ and $u_{j-1}$ to $u_{j+1}$ by spherical segments) and its spherical convex hull $\mathcal{H}_S(\Tilde{Q})$. It is instructive to see some examples. First, a lemma:

\begin{lemma}
\label{lemma_spherical_convex_hulls_do_not_intersect}
    Let $Q=[u_1,...,u_n]$ ($n \geq 7$) be a spherical polygon in balanced position and without self nor antipodal intersections, with $Ess(P)=2$. Let $u_i$ an $u_j$ be essential vertices and $\Tilde{Q}:= Q - u_i - u_j$. Then
    $$ \mathcal{H}_S(\Tilde{Q}) \cap \mathcal{H}_S(\Tilde{\overline{Q}}) =\emptyset.$$
\end{lemma}
\begin{proof}
    Since $u_i$ and $u_j$ are essential, $\Tilde{Q}$ is contained in a (open) hemisphere $H$. By the same reason, $\overline{\Tilde{Q}}$ is also contained in a (open) hemisphere $H'$ which is disjoint from $H$. Since $\Tilde{\overline{Q}}=\overline{\Tilde{Q}}$, the result follows.
\end{proof}

\begin{example}
    Let $Q$ be the polygon of figure \ref{fig:example_triangulation_1}. In this case we have two excellent and nonessential vertices. The triangulation of $Int(\Tilde{Q})$ has two leaves, which gives $u_5$ and $u_7$ as excellent vertices of \textit{the polygon $\Tilde{Q}$}. Since $u_5$ is connected through edges of the original polygon $Q$ while $u_7$ is not, we can deduce only that $u_5$ is an excellent vertex of \textit{the original polygon $Q$}. On the other hand, $Ext(\Tilde{Q}) \cap \mathcal{H}_S(\Tilde{Q})$ is already a triangle, which gives $u_6$ as the intermediate vertex of the unique leaf of the triangulation. Notice that the edges adjacent to $u_6$ are contained in the original polygon $Q$.
    
    Is $u_6$ an excellent vertex of $Q$? If it were not excellent, then the segment $\overrightarrow{u_5 u_7}$ would intersect $\overline{Q}$ at some edge, which would imply that a vertex $\overline{u}$ of $\overline{Q}$ is in $\mathcal{H}_S(\Tilde{Q})$. But this contradicts lemmas \ref{lemma_nonessentials_are_in_union_of_triangles} and \ref{lemma_spherical_convex_hulls_do_not_intersect}. Therefore, $u_6$ is indeed excellent.
\end{example}

\begin{figure}
    \centering
    \includegraphics[scale=0.25]{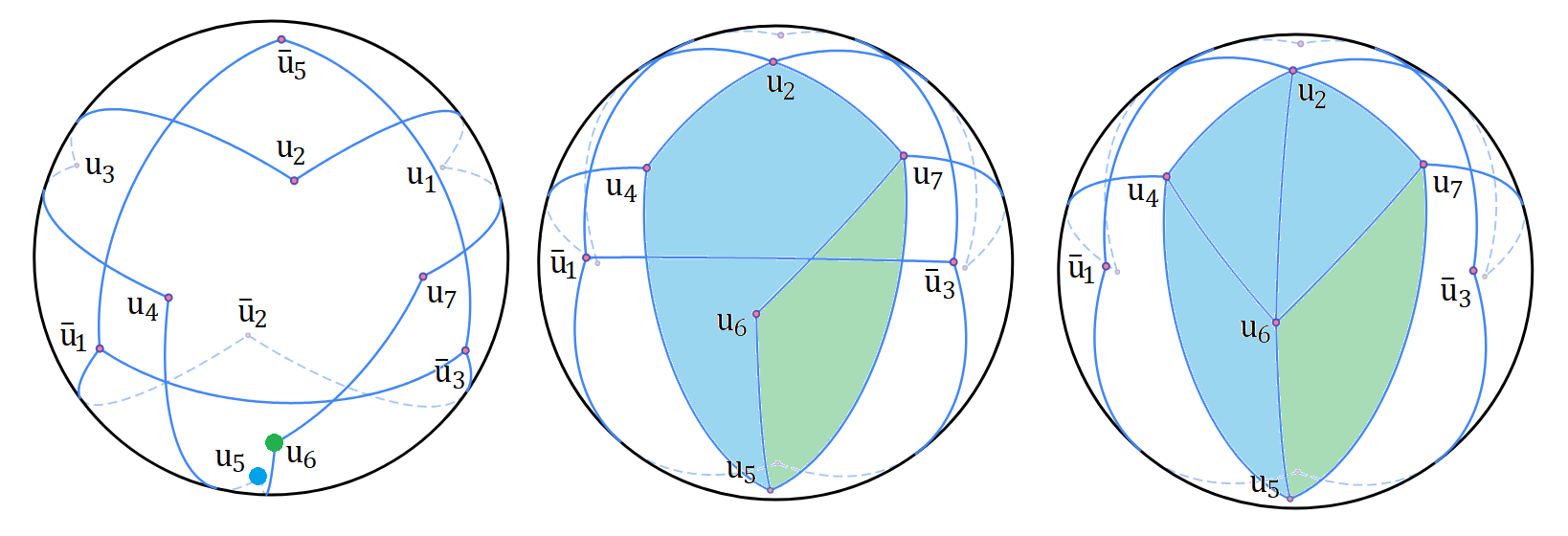}
    \caption{This polygon has $u_5$ and $u_6$ as excellent and nonessential vertices: $u_5$ (blue) can be found through an triangulation of the interior of polygon $\Tilde{Q}$ (blue), while $u_6$ (green) comes from the triangle obtained as the intersection of the exterior of $\Tilde{Q}$ with $\mathcal{H}_S(\Tilde{Q})$ (green).}
    \label{fig:example_triangulation_1}
\end{figure}

\begin{example}
    Let $Q$ be the polygon of figure \ref{fig:example_triangulation_2}. It has two excellent and nonessential vertices: $u_5$ and $u_6$. Both come from internal triangulations but, since they are consecutive, the triangulations used in each case must be different. Notice that these two triangulations also point to $u_4$ and $u_7$ as the intermediate vertices of the leaves of each triangulation, but since at least one edge adjacent to each these vertices is not of the original polygon $Q$, we cannot conclude whether or not they excellent vertices of $Q$.

    On the other hand, $u_2$ is the intermediate vertex of the unique triangle of $Ext(\Tilde{Q}) \cap \mathcal{H}_S(\Tilde{Q})$, but the two edges adjacent to it are not of the original polygon $Q$. Therefore we cannot conclude that $u_2$ is an excellent vertex of $Q$.
\end{example}

\begin{figure}
    \centering
    \includegraphics[scale=0.2]{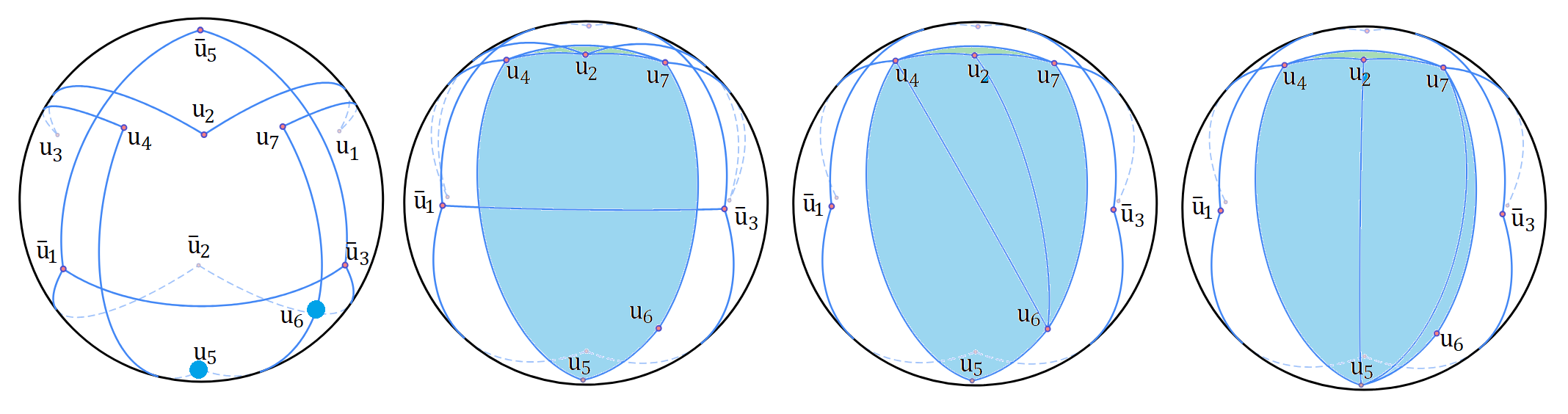}
    \caption{This polygon $Q$ has $u_5$ and $u_6$ as excellent and nonessential vertices (in blue). Different triangulations are needed in order to conclude that both of them are excellent because both vertices are consecutive. On the other hand, although $u_2$ is the intermediate vertex of the unique triangle of the triangulation of $Ext(\Tilde{Q}) \cap \mathcal{H}_S(\Tilde{Q})$, it is not an excellent vertex of $Q$.}
    \label{fig:example_triangulation_2}
\end{figure}

\begin{example}
    Let $Q$ be the polygon of figure \ref{fig:example_triangulation_3}. Both vertices $u_2$ and $u_4$ are the intermediate vertices of the leaves of the internal triangulation, but since they are adjacent to edges which are not contained in the original polygon $Q$, we cannot conclude that they are excellent vertices of $Q$.

    On the other hand, vertex $u_3$ is the intermediate vertex of the unique triangle of $Ext(\Tilde{Q}) \cap \mathcal{H}_S(\Tilde{Q})$ and, moreover, is adjacent to edges of the original polygon $Q$. To conclude that $u_3$ is indeed excellent, it suffices to see that the segment $\overrightarrow{u_2 u_4}$ does not intersect $\overline{Q}$. For if it intersected an edge of $\overline{Q}$, it would follow that that $\mathcal{H}_S(\Tilde{Q})$ contains a vertex $u$ of $\overline{Q}$ in its interior, which contradicts lemmas \ref{lemma_nonessentials_are_in_union_of_triangles} and \ref{lemma_spherical_convex_hulls_do_not_intersect}. Therefore, $u_3$ must be an excellent vertex of $Q$.
\end{example}

\begin{figure}
    \centering
    \includegraphics[scale=0.2]{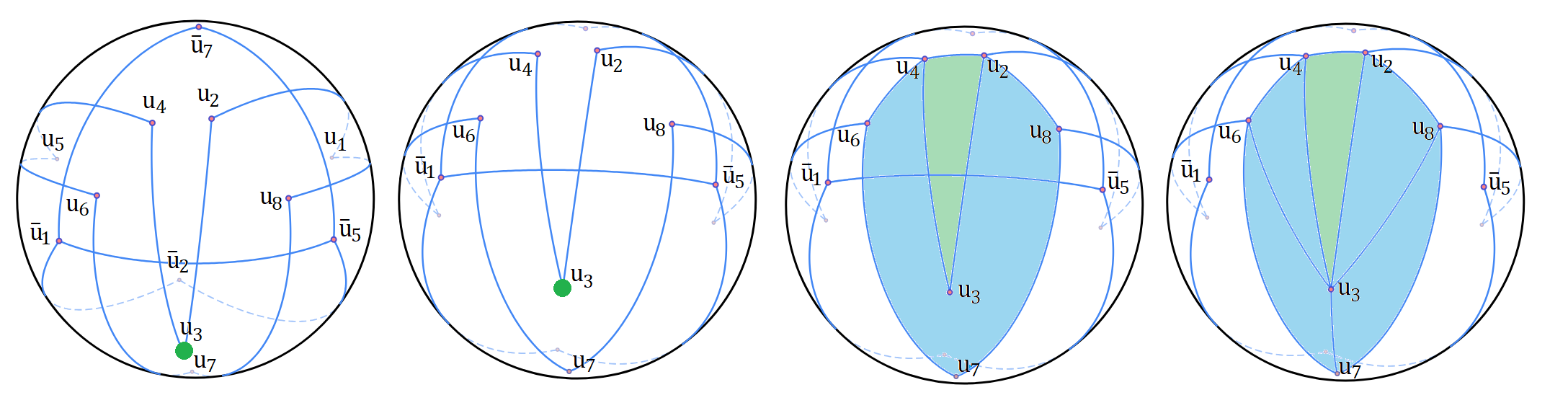}
    \caption{This polygon has only $u_3$ (in green) as an excellent vertex: it comes from the unique triangle of $Ext(\Tilde{Q}) \cap \mathcal{H}_S(\Tilde{Q})$.}
    \label{fig:example_triangulation_3}
\end{figure}

The previous examples show some general principles when looking for excellent and nonessential vertices of $Q$. When we consider the intermediate vertex of any leaf of the triangulation, it might be a vertex among $u_{i-1}$, $u_{i+1}$, $u_{j-1}$ or $u_{j+1}$. In this case it is an excellent vertex of $\Tilde{Q}$, although we cannot conclude in principle if it is an excellent vertex of $Q$. On the other hand, if the vertex is not one of these four vertices, the next lemma shows that it must be an excellent vertex of $Q$:

\begin{lemma}
\label{lemma_conditions_excellent}
    Let $Q=[u_1,...,u_n]$ ($n \geq 7$) be a spherical polygon in balanced position and without self nor antipodal intersections, with $Ess(P)=2$. Let $u_i$ an $u_j$ be the essential vertices and let $\Tilde{Q}:= Q - u_i - u_j$. Let $u_t$ be a good vertex of $\Tilde{Q}$ coming from a triangulation of either $Int(\Tilde{Q})$ or of a connected component of $Ext(\Tilde{Q}) \cap \mathcal{H}_S(\Tilde{Q})$. If $u_t$ is not one of the vertices $u_{i-1}$, $u_{i+1}$, $u_{j-1}$ and $u_{j+1}$, then $u_t$ is an excellent vertex of $Q$.
\end{lemma}
\begin{proof}
    If $u_t$ comes from a triangulation of $Int(\Tilde{Q})$, it is excellent since in this case edge $\overrightarrow{u_{t-1} u_{t+1}}$ is contained in $Int(\Tilde{Q})$ and therefore cannot intersect $\overline{Q}$.
    
    Now, suppose that $u_t$ comes from a triangulation of a connected component of $Ext(\Tilde{Q}) \cap \mathcal{H}_S(\Tilde{Q})$. To prove that $u_t$ is excellent, it suffices to show that the segment $\overrightarrow{u_{t-1} u_{t+1}}$ (i.e., the side of $\triangle_1$ not adjacent to $u_t$) does not intersect $\overline{Q}$. For if it intersected an edge of $\overline{Q}$, it would follow that that $\mathcal{H}_S(\Tilde{Q})$ contains a vertex $\overline{u}$ of $\overline{Q}$ in its interior, which contradicts lemmas \ref{lemma_nonessentials_are_in_union_of_triangles} and \ref{lemma_spherical_convex_hulls_do_not_intersect}.
\end{proof}

\begin{figure}
    \centering
    \includegraphics[scale=0.25]{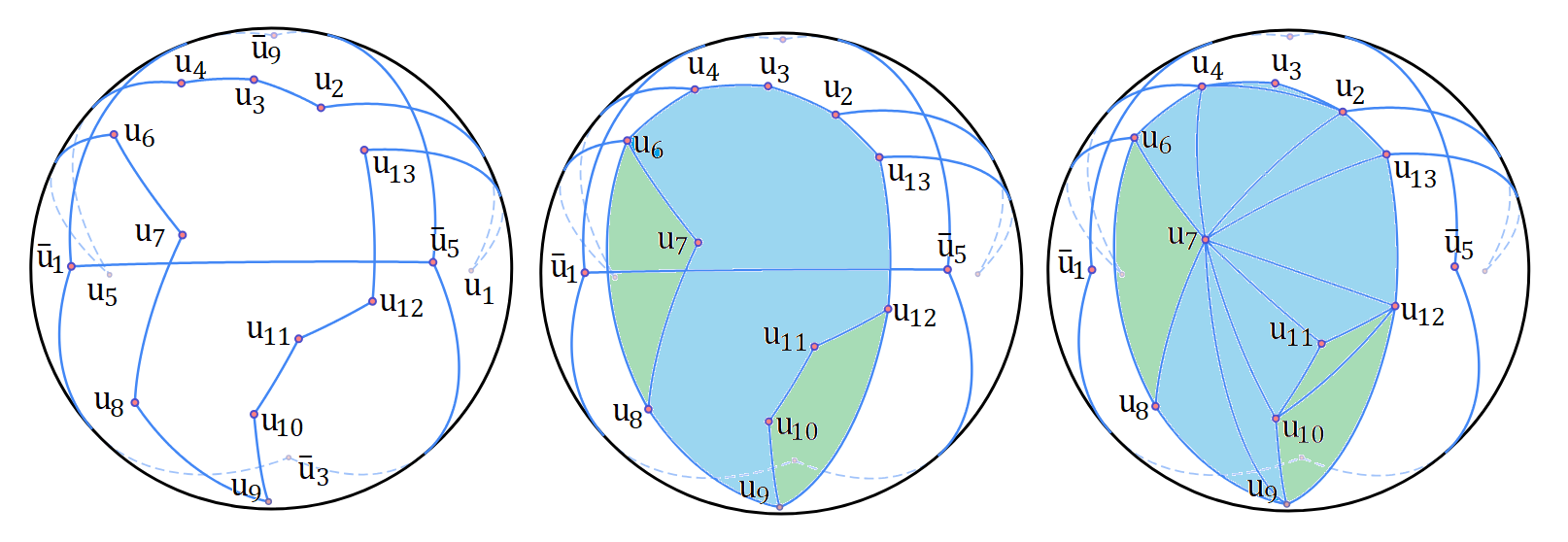}
    \caption{Vertices $u_3$ and $u_8$ are excellent and come from a triangulation of $Int(\Tilde{Q})$. Vertices $u_7$ and $u_{11}$ are also excellent, but come from a triangulation of $Ext(\Tilde{Q}) \cap \mathcal{H}_S(\Tilde{Q})$. Notice that since $u_6$, $u_8$ and $u_9$ are also at the boundary of components of $Ext(\Tilde{Q}) \cap \mathcal{H}_S(\Tilde{Q})$, they can be seen as the ``intermediate" vertices of leaves of these regions. However, each of them is adjacent to a segment that is not of the original polygon $Q$, so lemma \ref{lemma_conditions_excellent} does not apply here.}
    \label{fig:example_triangulation_4}
\end{figure}

Before proving proposition \ref{proposition_essentials_2_vertex}, we prove a lemma which states a lower bound on the number of good vertices of a \textit{polygon which is not in balanced position}. From now on, "convex" will always mean "spherically convex".

\begin{lemma}
\label{lemma_3_good_vertices}
    Let $Q=[u_1,...,u_n]$ ($n \geq 4$) be a simple spherical (or even planar) polygon which is contained in a closed hemisphere. Then $Q$ has at least 3 good vertices. More precisely:
    \begin{itemize}
        \item[(a)] If $Q$ is convex, then all of its vertices are good;
        \item[(b)] If $Q$ is not convex, then it has at least 3 good vertices.
    \end{itemize}
\end{lemma}
\begin{proof}
    (a) is clear: if a vertex $u_i$ were not good, then the segment $\overrightarrow{u_{i-1} u_{i+1}}$ would intersect the polygon $Q$ at some other edge and therefore some of its points would not be in the interior of $Q$, contradicting the convexity of $Q$.
    
    (b) If $Q$ is not convex, then $\mathcal{H}_S(Q) \cap Ext(Q)$ is nonempty. Consider any triangulation of $\mathcal{H}_S(Q)$: since $Q$ has at least 4 vertices, the dual graph of $\mathcal{H}_S(Q) \cap Int(Q)$ is nontrivial and, by a basic theorem of graph theory, has at least two leaves. The intermediate vertex of each of these triangles is a good vertex.

    Moreover, $\mathcal{H}_S(Q) \cap Ext(Q)$ is nonempty (it might have more than one connected component). Consider a connected component $\mathcal{C}$ of $\mathcal{H}_S(Q) \cap Ext(Q)$: it is either a triangle with vertices $u_{i-1}$, $u_i$ and $u_{i+1}$ (in which case $u_{i}$ is the good vertex), or another type of poligonal region. In the latter case the dual graph of any triangulation of $\mathcal{C}$ is nontrivial and, again by a basic theorem of graph theory, has at least two leaves. Since at most one of these triangles might have as side the segment of $\partial \mathcal{H}_S(Q)$ which is not a side of $Q$, at least one of the leaves has two sides contained in $Q$. The vertex of $Q$ which is adjacent to these sides is therefore good.
\end{proof}

Now we can finally prove proposition \ref{proposition_essentials_2_vertex}:

\begin{proof}
    (of proposition \ref{proposition_essentials_2_vertex})
    Let $p$ be the polygonal line from $u_{i+1}$ to $u_{j-1}$.
    
    \textbf{Step 1}: First we prove the proposition in the case where $i+2 < j$, i.e., where $p$ is not a sole point (see figure \ref{fig:proof_similar_to_ess_3}.a). Here we proceed analogously to the proof of the proposition for $Ess(Q)=3$: we consider the (closed) polygon $[u_i,u_{i+1},...,u_{j-1},u_j]$ and denote by $R$ the region enclosed by it which contains the vertex $\overline{u}_k$. Triangulating the region $R$, we obtain again two leaves $\triangle_1$ and $\triangle_2$, one of them (say $\triangle_1$) having as intermediate vertex a nonessential vertex (see figure \ref{fig:proof_similar_to_ess_3}.b). It must be excellent: for if it were not the case, then the edge of $\triangle_1$ which is not contained in $Q$ (which we denote by $e$) would intersect an edge of $\overline{Q}$ (which we denote by $\overline{f}$), in which case there would be two possibilities:
    \begin{itemize}
        \item The endpoints of $e$ are two nonessential vertices (see figure \ref{fig:proof_similar_to_ess_3}.c). Since $e$ intersects $\overline{f}$, one of the endpoints of $\overline{f}$ is inside of $\triangle_1$ and therefore cannot be $\overline{u}_i$ nor $\overline{u}_j$, i.e., it must be a nonessential vertex of $\overline{Q}$. But then we would have four points of $Q$ in balanced position (namely, the endpoints of $e$ and $f$), at most one of them being essential. But this contradicts our hypothesis of $Ess(Q)=2$;
        \item One endpoint of $e$ is nonessential and the other is essential (see figure \ref{fig:proof_similar_to_ess_3}.d). We can assume without loss of generality that the latter point is $u_i$. Again, since $e$ intersects $\overline{f}$, one of the endpoints of $\overline{f}$ is inside of $\triangle_1$ and therefore cannot be $\overline{u}_i$ nor $\overline{u}_j$, i.e., it must be a nonessential vertex of $\overline{Q}$. The other endpoint cannot be $\overline{u}_i$ nor $\overline{u}_j$, since in this case $\overline{f}$ would not intersect $e$. Therefore, we would have four points of $Q$ in balanced position (namely, the endpoints of $e$ and $f$), with only one of them being essential. But this contradicts our hypothesis of $Ess(Q)=2$.
    \end{itemize}

    \begin{figure}
        \centering
        \includegraphics[scale=0.2]{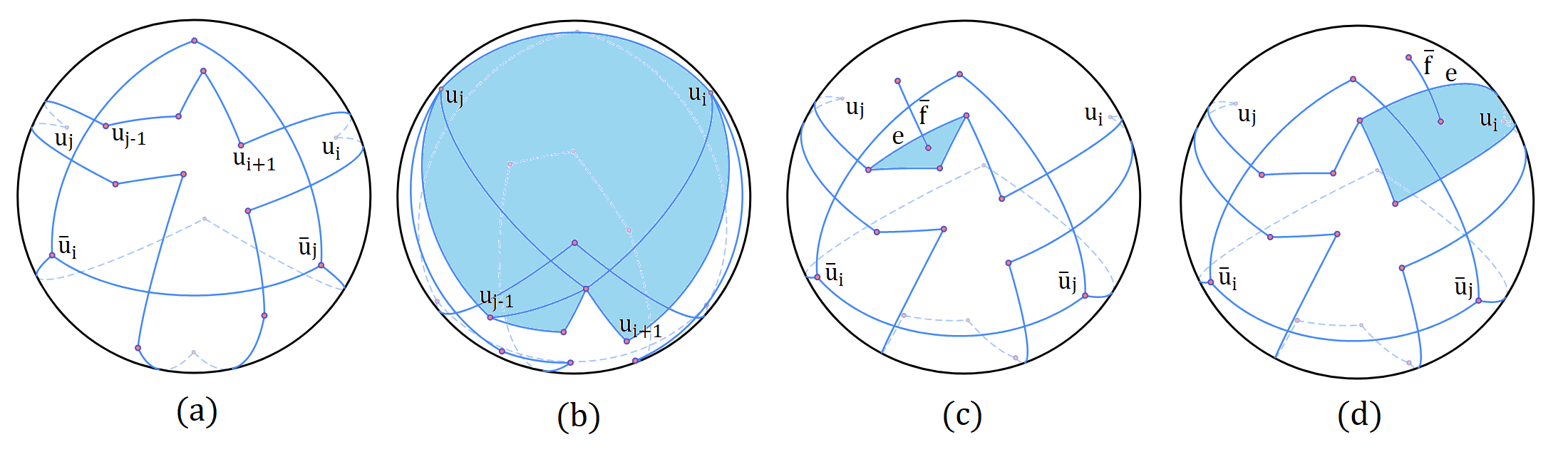}
        \caption{If the polygonal line $p$ is not a sole point, the polygon $[u_i,u_{i+1},...,u_{j-1},u_j]$ defines a region $R$ whose triangulation has at least two leaves. In this example, the intermediate vertices of each leaf of the triangulation are excellent vertices of $Q$.}
        \label{fig:proof_similar_to_ess_3}
    \end{figure}
    
    \textbf{Step 2}: Now we assume (for the rest of the proof) that $i+2=j$, i.e., the polygonal line $p$ is the sole point $u_{i+1}=u_{j-1}$ (see figure \ref{fig:q_and_q_tilde}). Denote this point by $u_h$. Denote by $p'$ the polygonal line from $u_{j+1}$ to $u_{i-1}$.
    
    Let $\Tilde{Q}:=Q - u_i - u_j$, where $u_i$ and $u_j$ are the essential vertices of $Q$ (see figure \ref{fig:q_and_q_tilde}). This new polygon is equal to the concatenation of the polygonal line $p$, the segment $\overrightarrow{u_{j-1} u_{j+1}}$, the polygonal line $p'$, and the segment $\overrightarrow{u_{i-1} u_{i+1}}$. Moreover, $\Tilde{Q}$ is a simple polygon with at least 5 vertices.

    \begin{figure}
        \centering
        \includegraphics[scale=0.2]{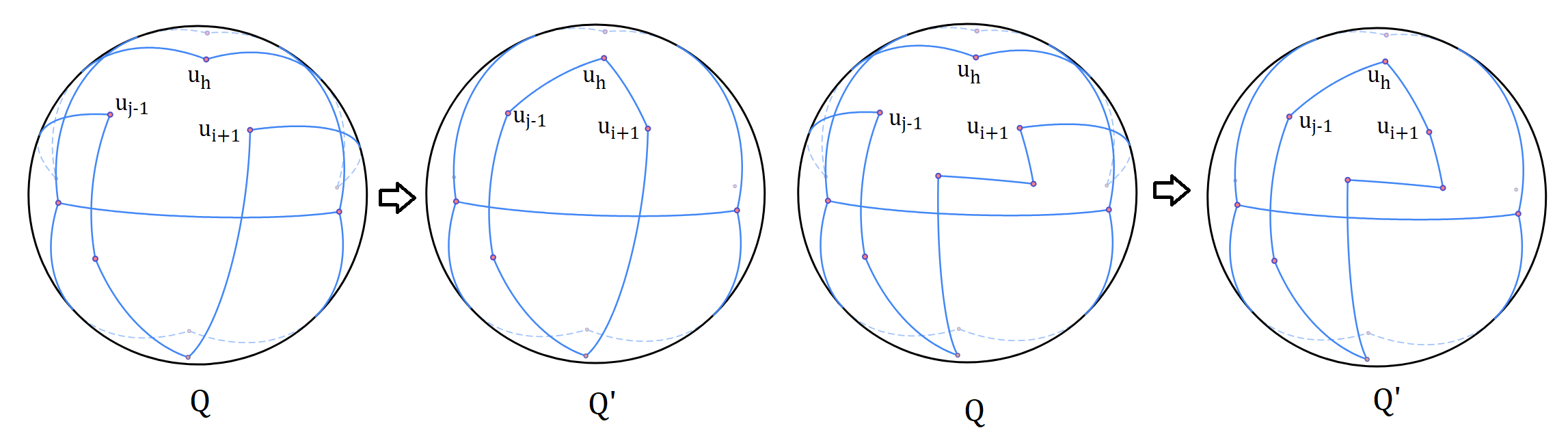}
        \caption{Two examples of a polygon $Q$ together with the corresponding $\Tilde{Q}:=Q - u_i - u_j$.}
        \label{fig:q_and_q_tilde}
    \end{figure}

    If $\Tilde{Q}$ is convex, then all of its vertices are good (by lemma \ref{lemma_3_good_vertices}.a) and therefore at least two of its vertices (namely, the ones different from $u_{i-1}$, $u_h$ and $u_{j+1}$) are also excellent vertices of the original polygon $Q$, by lemma \ref{lemma_conditions_excellent}.

    Now, if $\Tilde{Q}$ is not convex, then it has at least 3 good vertices by lemma \ref{lemma_3_good_vertices}.b. If one of these vertices is different from $u_{i-1}$, $u_h$ and $u_{j+1}$, then it is an excellent vertex of $Q$, by lemma \ref{lemma_conditions_excellent}.

    \textbf{Step 3}: Suppose now (for the rest of the proof) that the set of good vertices of $\Tilde{Q}$ (relative to the given triangulation) is precisely $\{u_{i-1},u_h,u_{j+1}\}$ (see figures \ref{fig:worst_situation_1} and \ref{fig:worst_situation_2}). By the proof of lemma \ref{lemma_3_good_vertices}.b, two of them are the intermediate vertices of leaves of a triangulation of the interior of $\Tilde{Q}$, while the other is the intermediate vertex of a leaf of a triangulation of $\mathcal{H}_S(\Tilde{Q}) \cap Ext(\Tilde{Q})$. Since $u_{i-1}$, $u_h$ and $u_{j+1}$ are consecutive in $\Tilde{Q}$, the only possibility is that $u_h$ is the intermediate vertex of the leaf of a triangulation of $\mathcal{H}_S(\Tilde{Q}) \cap Ext(\Tilde{Q})$.

    \begin{figure}
        \centering
        \includegraphics[scale=0.2]{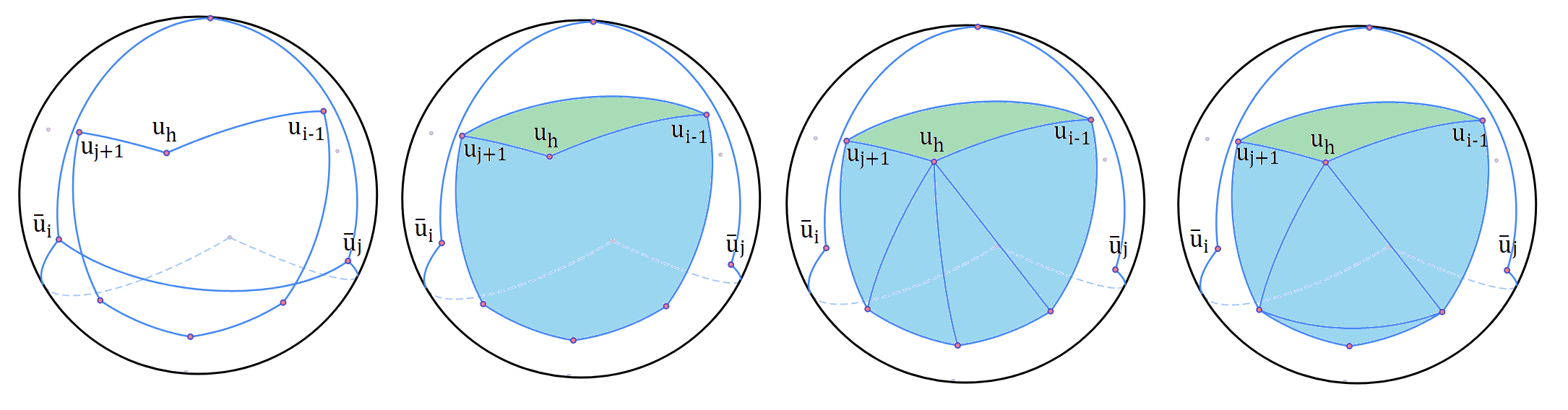}
        \caption{An example of a polygon $\Tilde{Q}$ such that the triangulation of $\mathcal{H}_S(\Tilde{Q})$ gives only $u_{i-1}$, $u_h$ and $u_{j+1}$ as good vertices of $\Tilde{Q}$. Notice that all edges which were added to form the internal triangulation must be adjacent to $u_h$. Moreover, we can ``flip" a pair of consecutive triangles in order to find another good vertex of $\Tilde{Q}$.}
        \label{fig:worst_situation_1}
    \end{figure}

    \begin{figure}
        \centering
        \includegraphics[scale=0.2]{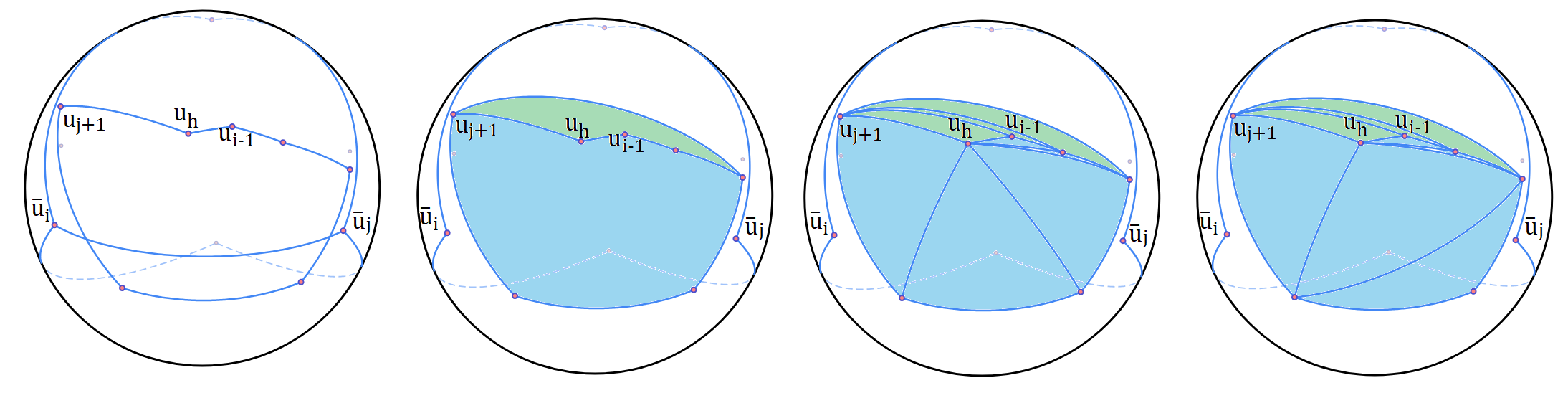}
        \caption{An example of a polygon $\Tilde{Q}$ such that the triangulation of $\mathcal{H}_S(\Tilde{Q})$ gives only $u_{i-1}$, $u_h$ and $u_{j+1}$ as good vertices of $\Tilde{Q}$. Notice that all edges which were added to form the internal triangulation must be adjacent to $u_h$. This example also shows that, besides the pairs of consecutive triangles of type (a) (which can be ``flipped"), type (b) pairs may also appear (which happens at vertices $u_h$, $u_{i-3}$, $u_{i-2}$ and $u_{i-1}$).}
        \label{fig:worst_situation_2}
    \end{figure}

    Now, since $u_{i-1}$, $u_h$ and $u_{j+1}$ are (in this order) consecutive in $\Tilde{Q}$, we have that all edges which were added to form the internal triangulation of $\Tilde{Q}$ must have $u_h$ as one of their endpoints: since $\triangle(u_{i-2},u_{i-1},u_h)$ and $\triangle(u_h,u_{j+1},u_{j+2})$ are the only leaves of the internal triangulation, all added edges must have as endpoints both a vertex in the polygonal line between $u_{i-1}$ and $u_{j+1}$ (in this order, relative to polygon $\Tilde{Q}$) and a vertex in the polygonal line between $u_{j+1}$ and $u_{i-1}$ (in this order, relative also to polygon $\Tilde{Q}$).

    Moreover, we have that the polygon $\Tilde{Q}$ is \textit{star-shaped} with respect to $u_h$: given any point of $\Tilde{Q}$ (including edge points), the segment connecting it to $u_h$ does not intersect the exterior of $\Tilde{Q}$.

    Therefore, the triangles
    $$ \triangle(u_h,u_{j+1},u_{j+2}), \triangle(u_h,u_{j+2},u_{j+3}),... \triangle(u_h,u_{i-3},u_{i-2}), \triangle(u_h,u_{i-2},u_{i-1}) $$
    are all the triangles of the internal triangulation of $\Tilde{Q}$. We want to examine each pair $\mathcal{P}$ of consecutive triangles of this sequence ($\triangle(u_h,u_{t-1},u_t)$ and $\triangle(u_h,u_t,u_{t+1})$, for each $j+1 < t < i-1$). The four vertices of $\Tilde{Q}$ which define them (i.e., $u_h$, $u_{t-1}$, $u_t$ and $u_{t+1}$) might be either in convex position or not. Because $\Tilde{Q}$ is star-shaped with respect to $u_h$, there are in principle only three possibilities (see figure \ref{fig:pair_of_triangles_3_cases}):
    \begin{itemize}
        \item[(a)] the four points are in convex position and such that $u_h$, $u_{t-1}$, $u_t$ and $u_{t+1}$ are in the counterclockwise order in $\partial \mathcal{H}_S(\mathcal{P})$;
        \item[(b)] $u_t$ is in the interior of $\mathcal{H}_S(\mathcal{P})$ and $u_h$, $u_{t-1}$ and $u_{t+1}$ are in the counterclockwise order in $\partial \mathcal{H}_S(\mathcal{P})$;
        \item[(c)] $u_h$ is in the interior of $\mathcal{H}_S(\mathcal{P})$ and $u_{t-1}$, $u_t$ and $u_{t+1}$ are in the counterclockwise order in $\partial \mathcal{H}_S(\mathcal{P})$;
    \end{itemize}

    \begin{figure}
        \centering
        \includegraphics[scale=0.2]{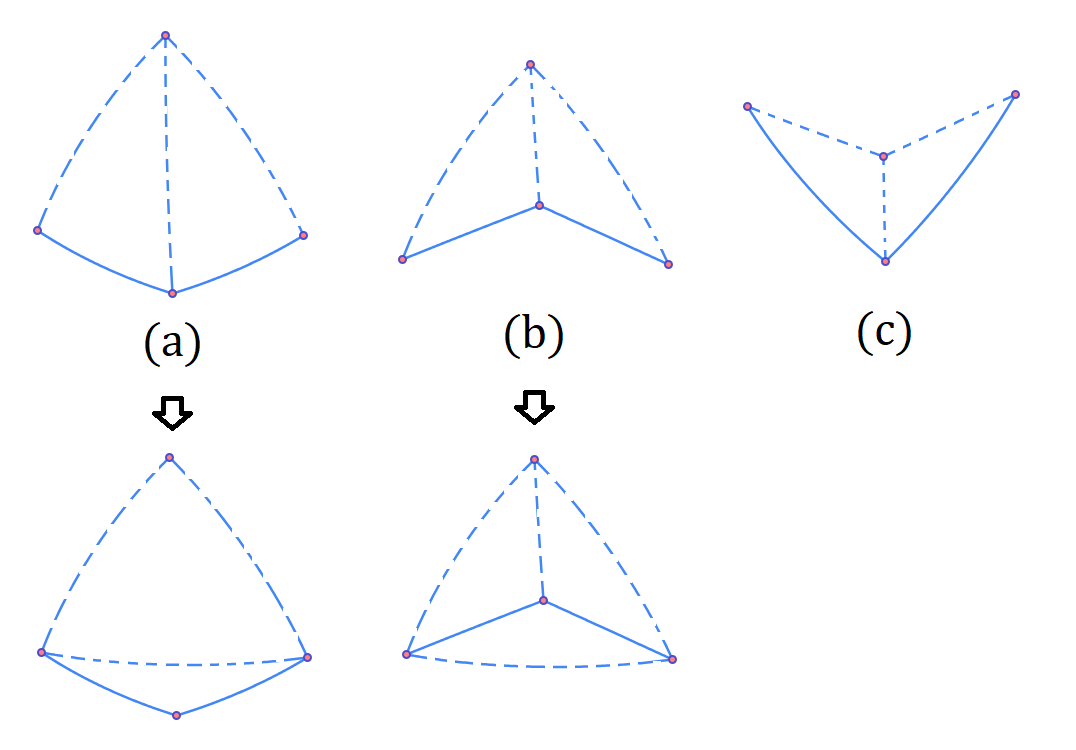}
        \caption{Three possibilities regarding a pair of consecutive triangles of the internal triangulation of $\Tilde{Q}$.}
        \label{fig:pair_of_triangles_3_cases}
    \end{figure}

    In case (a), we can just ``flip" the pair of triangles, i.e., switch the triangulation of $\mathcal{P}$ (and therefore change also the triangulation of the interior of $\Tilde{Q}$). In this case, we have that vertex $u_t$ is a good vertex of $\Tilde{Q}$.
    
    In case (b), we notice that, since $\Tilde{Q}$ is star-shaped with respect to $u_h$, we have that the triangle $\triangle(u_{t-1},u_t,u_{t+1})$ does not have vertices in its interior and, therefore, the vertex $u_t$ is a good vertex of $\Tilde{Q}$.

    It remains to show that case (c) cannot happen for every pair of triangles $\mathcal{P}$ of the internal triangulation of $\Tilde{Q}$. We have two cases to consider: the polygonal line $p'$ has either 4 vertices or more than 4 vertices.
    \begin{itemize}
        \item If the polygonal line $p'$ from $u_{j+1}$ to $u_{i-1}$ has exactly 4 vertices (see figure \ref{fig:pentagon}), then $\Tilde{Q}$ is a pentagon and, moreover, vertices $u_{j+2}$ and $u_{i-2}$ are consecutive and are both in the triangular region $\triangle(\overline{u}_i,\overline{u}_j,\overline{u}_l)$. 
        
        Here, it is impossible for vertex $u_h$ to be inside of $\triangle(u_{j+2},u_{i-2},u_{i-1})$: since $u_h$ is adjacent to $u_i$ in the original polygon $Q$, it must be ``on the left" of the oriented segment $\overrightarrow{u_{i-1} u_i}$ (i.e., the determinant $[u_{i-1},u_i,u_h]$ is positive). On the other hand, any point of the interior of $\triangle(u_{j+2},u_{i-2},u_{i-1})$ must be ``on the right" of $\overrightarrow{u_{i-1} u_i}$ (i.e., the corresponding determinant is negative), because points $u_{j+2}$ and $u_{i-2}$ are also on the right of $\overrightarrow{u_{i-1} u_i}$.

        Analogously, one proves that $\triangle(u_{j+1},u_{j+2},u_{i-2})$ cannot have $u_h$ in its interior.

        The conclusion is that in the case of $\Tilde{Q}$ being a pentagon, only cases (a) and (b) can happen, which implies the existence of at least two good vertices of $\Tilde{Q}$, both of them different from $u_h$, $u_{j+1}$ and $u_{i-1}$. Therefore, they are excellent vertices of $Q$, by lemma \ref{lemma_conditions_excellent}.

        \item If the polygonal line $p'$ from $u_{j+1}$ to $u_{i-1}$ has more than 4 vertices (see figure \ref{fig:hexagon}), then $\Tilde{Q}$ is a polygon with at least 6 vertices, whose interior is triangulated with at least 4 triangles. Therefore, there are at least 3 pairs of consecutive triangles $\mathcal{P}_1$,..., $\mathcal{P}_m$ ($m \geq 3$). Assume that there are no pairs of triangles of type (b). In this case (see figure \ref{fig:hexagon}), vertex $u_h$ cannot be in the interior of all triangles of the form $\triangle(u_{j+1},u_{j+2},u_{j+3})$, ..., $\triangle(u_{i-3},u_{i-2},u_{i-1})$, since some of them have disjoint interiors (for instance, triangles $\triangle(u_{j+1},u_{j+2},u_{j+3})$ and $\triangle(u_{i-3},u_{i-2},u_{i-1})$).

        Therefore, there must be at least one pair $\mathcal{P}$ of triangles such as in case (a) or (b). This implies the existence of at least one good vertex of $\Tilde{Q}$, which is different from $u_h$, $u_{j+1}$ and $u_{i-1}$. Hence it is also an excellent vertex of $Q$, by lemma \ref{lemma_conditions_excellent}.
    \end{itemize}

    \begin{figure}
        \centering
        \includegraphics[scale=0.2]{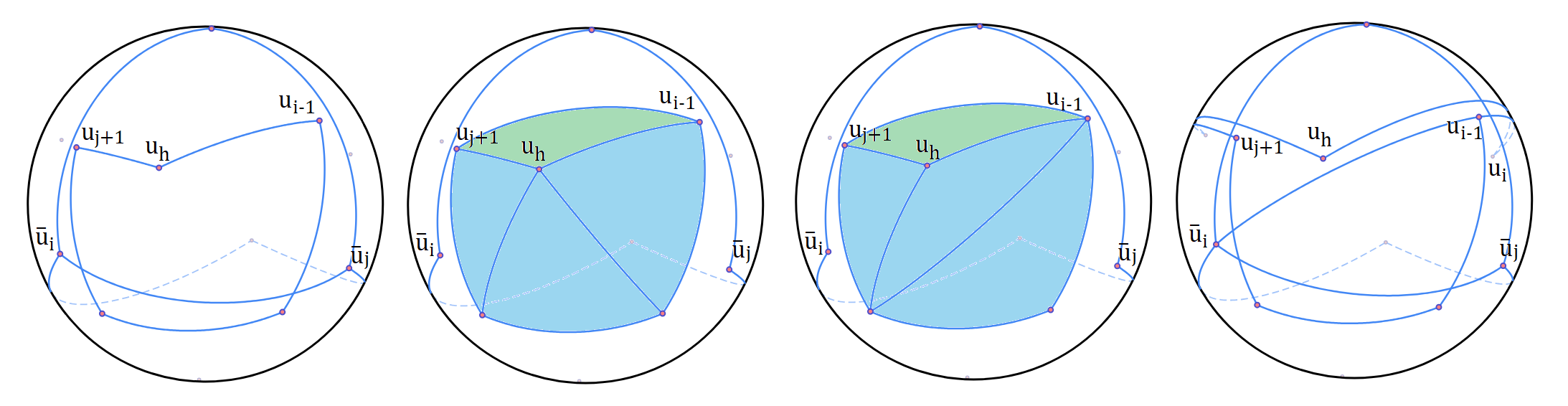}
        \caption{If $\Tilde{Q}$ is a pentagon and does not have pairs of type (b), we have that vertex $u_h$ cannot be inside of $\triangle(u_{j+2},u_{i-2},u_{i-1})$, since both sets are separated by the spherical line spanned by $u_{i-1}$ and $u_i$. Analogously, $u_h$ cannot be inside of $\triangle(u_{j+1},u_{j+1},u_{i-2})$.}
        \label{fig:pentagon}
    \end{figure}

    \begin{figure}
        \centering
        \includegraphics[scale=0.3]{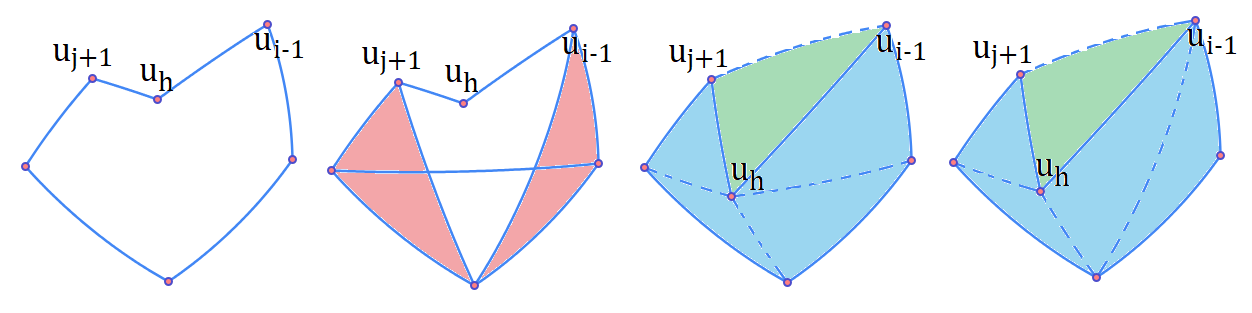}
        \caption{If $\Tilde{Q}$ has 6 or more vertices and does not have pairs of type (b), then the situation is depicted above: there will be triangles formed by consecutive vertices of $p'$ whose interiors are disjoint. Even if we move $u_h$, it cannot be inside all of these triangles.}
        \label{fig:hexagon}
    \end{figure}
\end{proof}

Since we proved propositions \ref{proposition_essentials_3_vertex} and \ref{proposition_essentials_2_vertex}, theorem \ref{special_segre_discrete} is also proved. An immediate corollary of theorem \ref{special_segre_discrete} is the following result:

\begin{corollary}
    Let $P=[v_1,...,v_n]$ ($n \geq 6$) be a space polygon whose tangent indicatrix does not self nor antipodal intersections. Then $P$ must have at least 6 flattenings.
\end{corollary}

\section{Spherical polygons with self and antipodal intersections}
\label{section_with_intersections}

From now on we prove some generalizations of the main results on spherical polygons that we already proved. Recall that theorems \ref{segre_theorem_spherical} (resp. \ref{segre_theorem_discrete_spherical}) stated that a closed smooth spherical curve (a spherical polygon, resp.), without self-intersections and not entirely contained in any hemisphere must have at least four inflections. If such curve / polygon is symmetric, then this lower bound can be improved to six (the smooth case is theorem \ref{mobius_theorem} and the discrete case is theorem \ref{discrete_mobius_theorem}).

In the previous sections, we saw that this lower bound on the number of inflections may be improved for nonsymmetric curves / polygons when the given curve / polygon does not have antipodal intersections (in the smooth case, this result was theorem \ref{special_segre_smooth} due to Ghomi, while the discrete case was theorem \ref{special_segre_discrete}).

Ghomi extended all these smooth results (see \cite{Ghomi}) by allowing the curves to have singularities (i.e., cusps) and/or double points (i.e., self and antipodal intersections). Then he obtains the following results:

\begin{theorem}
\label{generalized_special_segre_smooth}
Given a $\mathcal{C}^2$ closed spherical curve $\gamma$, not entirely contained in any closed hemisphere, let $S$ be the number of singular points of $\gamma$, $I$ be the number of its inflections, and $D$ the number of pairs of points $t \neq s \in \mathbb{S}^1$ where $\gamma(t) = \pm \gamma(s)$. Then
$$ 2(D + S) + I \geq 6.$$
\end{theorem}

\begin{theorem}
\label{generalized_segre_smooth}
Let $\gamma$ be a spherical curve under the same conditions and notations of the previous theorem. If $D^+ (\leq D)$ denotes the number of pairs of points $t \neq s \in \mathbb{S}^1$ where $\gamma(t) = \gamma(s)$, then
$$ 2(D^+ + S) + I \geq 4.$$ 
Furthermore, if $\gamma$ is symmetric, i.e., $-\gamma = \gamma$, then
$$ 2(D^+ + S) + I \geq 6.$$
\end{theorem}

Note that in the particular case where $D^+ = S = 0$, theorem \ref{generalized_segre_smooth} is theorem \ref{segre_theorem} (or theorem \ref{mobius_theorem} if such curve is symmetric). In the particular case where $D = S = 0$, theorem \ref{generalized_special_segre_smooth} is theorem \ref{special_segre_smooth}.

We have already proved the discrete version of Theorem \ref{generalized_segre_smooth} in the case where $D^+ = S = 0$ and Theorem \ref{generalized_special_segre_smooth} in the case where $D = S = 0$. It remains, therefore, to state Theorems \ref{generalized_special_segre_smooth} and \ref{generalized_segre_smooth} in the discrete setting and then prove them.

Our strategy to prove these theorems consists of a ``cutting-and-pasting" procedure: at the point of self-intersection we delete the two edges and reconnect its adjacent vertices in such a way that we obtain a new polygon, which is free from self-intersections. We then prove that the number of inflections of the original polygon plus $2$ (which refers to the self-intersection counted twice) is greater or equal than the number of inflections of the resulting polygon, for which the lower bound is known.

We still have to define the notion of a ``discrete singular point". In the smooth case, the singular point is a cusp, i.e., a point at which the curve fails to be differentiable. Therefore it makes sense to define a ``discrete cusp" as a pair of two consecutive inflections (see figure \ref{fig:discrete_cusp}).

\begin{figure}
    \centering
    \includegraphics[scale=0.4]{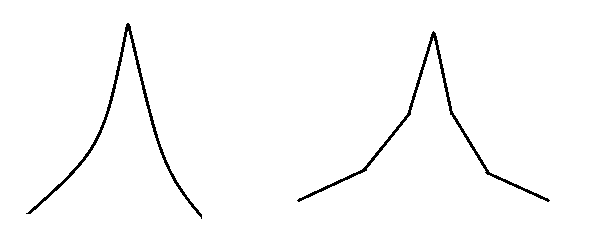}
    \caption{A cusp in the smooth/discrete setting.}
    \label{fig:discrete_cusp}
\end{figure}

Therefore, since a discrete cusp already counts as two inflections, it is enough to state theorems \ref{generalized_special_segre_smooth} and \ref{generalized_segre_smooth} without the explicit number $S$ of discrete cusps of the polygon.

Given a spherical polygon $Q$, let $D^+$ the number of double points (i.e., self-intersections) of the polygon, $D^-$ be the number of its antipodal-double points (i.e., antipodal intersections), $D := D^+ + D^-$ and $I$ be the number of its inflections. Now we can state the discrete analogs of theorems \ref{generalized_special_segre_smooth} and \ref{generalized_segre_smooth}.

\begin{theorem}
\label{generalized_special_segre_discrete}
Let $Q=[u_1,...,u_n] \in \mathbb{S}^2$ ($n\geq 6$) be a spherical polygon, not entirely contained in a closed hemisphere. Then
$$ 2D + I \geq 6.$$
\end{theorem}

\begin{theorem}
\label{generalized_segre_discrete}
Let $Q=[u_1,...,u_n] \in \mathbb{S}^2$ ($n\geq 4$) be a spherical polygon, not entirely contained in a closed hemisphere. Then
$$ 2D^+ + I \geq 4.$$
\end{theorem}

\begin{theorem}
\label{generalized_mobius_discrete}
Let $Q=[u_1,...,u_{2n}] \in \mathbb{S}^2$ ($2n\geq 6$) be a spherical polygon. If $Q$ is symmetric, then
$$ 2D^+ + I \geq 6.$$
\end{theorem}

We will first prove theorem \ref{generalized_segre_discrete}. Notice that for $D^+ \geq 2$ the result follows trivially, while for $D^+ = 0$ the result follows from Theorem \ref{segre_theorem_discrete_spherical}. Therefore, it suffices to consider the case where $D^+ = 1$. The strategy to prove the theorem in this case will be following:
we eliminate the intersecting edges and reconnect the adjacent vertices in such a way that the resulting polygon $Q'$ does not have self-intersections (we might need to add some more vertices to the polygon $P$ before this operation, in order to avoid new self-intersections). We then show that the number $2D^+ + I = 2 + I$ is greater or equal than the number $I'$ of inflections of $Q'$, which in its turn is greater or equal than $4$, by Theorem \ref{segre_theorem_discrete_spherical}.

\section{Eliminating self-intersections}

Suppose that we are given a spherical (or even planar) polygon $Q = [u_1,u_2,...,u_n]$ (with the usual orientation) with only one self-intersection, which happens at edges $\overrightarrow{u_i u_{i+1}}$ and $\overrightarrow{u_j u_{j+1}}$ (see figure \ref{fig:decomposition_polygons} - left). Suppose that no other vertex is in the spherical region spanned by $\{u_i,u_{i+1},u_j,u_{j+1}\}$, i.e., $\mathcal{C}(u_i,u_{i+1},u_j,u_{j+1}) \cap \mathbb{S}^2$ (in the case of a planar polygon, this region would be the convex hull of points $u_i$, $u_{i+1}$, $u_j$ and $u_{j+1}$).

\begin{figure}
    \centering
    \includegraphics[scale=0.4]{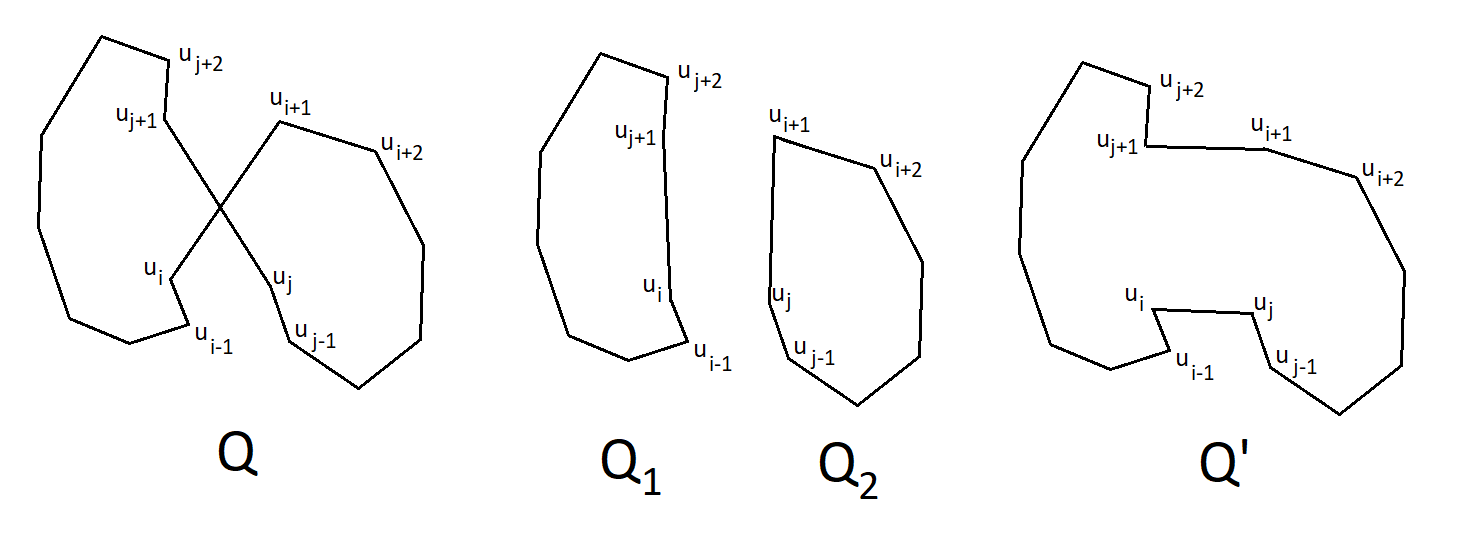}
    \caption{Two possible ways of eliminating a intersection: only one way results in a connected polygon.}
    \label{fig:decomposition_polygons}
\end{figure}

After deleting edges $\overrightarrow{u_i u_{i+1}}$ and $\overrightarrow{u_j u_{j+1}}$, there are two different ways of reconnecting these pairs of vertices (see figure \ref{fig:decomposition_polygons}):

\begin{itemize}
    \item form an edge from $u_i$ to $u_{j+1}$ and an edge of from $u_j$ to $u_{i+1}$. The result will be two new polygons $Q_1 = [u_i,u_{j+1},u_{j+2},...,u_{i-1}]$ and $Q_2 = [u_j,u_{i+1},u_{i+2},...,u_{j-1}]$, both of them with the orientation induced by that of $Q$;
    \item form an edge from $u_i$ to $u_j$ and an edge from $u_{i+1}$ to $u_{j+1}$, and reverse the orientation from vertex $u_{i+1}$ to $u_j$. This will result in only one polygon $Q' = [u_i,u_j,u_{j-1},...,u_{i+2},u_{i+1},u_{j+1},u_{j+2},...,u_{i-1}]$.
\end{itemize}

Both possibilities do not cause new self-intersections because we are assuming that no other vertex is in the region spanned by $\{u_i,u_{i+1},u_j,u_{j+1}\}$. More precisely, denoting by $w$ the self-intersection of edges $\overrightarrow{u_i u_{i+1}}$ and $\overrightarrow{u_j u_{j+1}}$, we see that:
\begin{itemize}
    \item $Q_1$ and $Q_2$ do not have self-intersections because no other vertex is in the regions spanned by $\{u_i,u_{j+1},w\}$ and $\{u_j,u_{i+1},w\}$;
    \item $Q'$ does not have self-intersections because no other vertex is in the regions spanned by $\{u_i,u_j,w\}$ and $\{u_{i+1},u_{j+1},w\}$.
\end{itemize}

Since we will need only the case where we obtain polygon $Q'$, our hypothesis might be narrowed down to the case in which no other vertex is in the regions spanned by $\{u_i,u_j,w\}$ and $\{u_{i+1},u_{j+1},w\}$.

Now we need the following lemma. Recall that $D^+$ and $I$ denote respectively the number of self-intersections and the number of inflections of $Q$ , while $D'^+$ and $I'$ denote respectively the number of self-intersections and the number of inflections of $Q'$.

\begin{lemma}
\label{never_decreasing_lemma}
Let $Q$ be a spherical (or planar) polygon with one self-intersection $w$ and such that no other vertex is in the regions spanned by $\{u_i,u_j,w\}$ and $\{u_{i+1},u_{j+1},w\}$. Let $Q'$ be the spherical (or planar) polygon obtained as above. Then
$$ 2D^+ + I \geq 2D'^+ + I'.$$
\end{lemma}

Lemma \ref{never_decreasing_lemma} does not apply  to a polygon $Q$ which does not satisfy the hypothesis concerning the regions spanned by $\{u_i,u_j,w\}$ and $\{u_{i+1},u_{j+1},w\}$ (see figure \ref{fig:counterexample_decreasing_lemma}). We will see later how to deal with such polygons.

\begin{figure}
    \centering
    \includegraphics[scale=0.4]{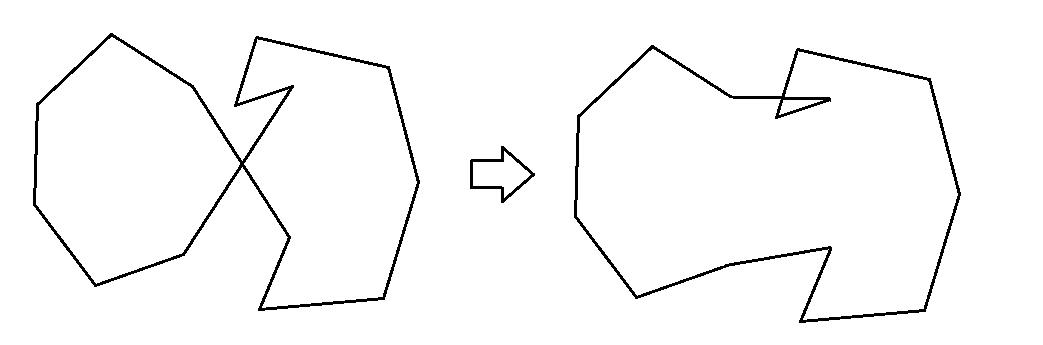}
    \caption{For the original polygon $Q$ we have $D^+=1$ and $I=2$, while for the resulting polygon $Q'$ we have $D'^+=1$ and $I'=4$.}
    \label{fig:counterexample_decreasing_lemma}
\end{figure}

Since $Q$ has only one self-intersection and $Q'$ has none, we have that $D^+=1$ and $D'^+=0$. Therefore it suffices to show that $2 + I$ is greater than or equal to $I'$, i.e., that $2+I-I'$ is greater than or equal to zero. First notice that an edge does not change its condition of being an inflection or not if we reverse its orientation (since $[u_{i+2},u_{i+1},u_i] = (-1) \cdot [u_i,u_{i+1},u_{i+2}]$ and $[u_{i+3},u_{i+2},u_{i+1}] = (-1) \cdot [u_{i+1},u_{i+2},u_{i+3}]$). Hence the operation of reversing the orientation of the polygon from vertex $u_{i+1}$ to $u_j$ does not increase nor decrease the number of inflections.

Therefore, in order to determine the number $2+I-I'$, it is enough to study what happens to the edges near the intersection:
\begin{itemize}
    \item two of the original edges disappear (namely, $\overrightarrow{u_i u_{i+1}}$ and $\overrightarrow{u_j u_{j+1}}$). These edges might be inflections, in which case they would contribute to the number $I$;
    \item two new edges appear (namely, $\overrightarrow{u_i u_j}$ and $\overrightarrow{u_{i+1} u_{j+1}}$). These edges might be inflections, in which case they would contribute to the number $I'$;
    \item the adjacent edges $\overrightarrow{u_{i-1} u_i}$ , $\overrightarrow{u_{i+1} u_{i+2}}$, $\overrightarrow{u_{j-1} u_j}$ and $\overrightarrow{u_{j+1} u_{j+2}}$ might also change their condition of being inflections or not, since we are altering one of the vertices that are adjacent to each one of these edges. Therefore theses edges might contribute either to $I$ or $I'$ (or both).
\end{itemize}

The condition of any of these edges to be an inflection or not depends on the position of the adjacent vertices $u_{i-1}$, $u_{i+2}$, $u_{j-1}$ and $u_{j+2}$. We call these vertices \textit{external}. The vertices $u_i$, $u_{i+1}$, $u_j$ or $u_{j+1}$ are \textit{internal}. Likewise, the edges whose endpoints are internal vertices ($\overrightarrow{u_i u_{i+1}}$, $\overrightarrow{u_j u_{j+1}}$, $\overrightarrow{u_i u_j}$ and $\overrightarrow{u_{i+1} u_{j+1}}$) are called \textit{internal edges}, while the remaining adjacent edges ($\overrightarrow{u_{i-1} u_i}$ , $\overrightarrow{u_{i+2} u_{i+1}}$, $\overrightarrow{u_j u_{j-1}}$ and $\overrightarrow{u_{j+1} u_{j+2}}$) are called \textit{external edges}.

Each external vertex might be in one of four different regions. For instance, vertex $u_{i-1}$ might be in one of the two regions determined by the line spanned by $\overrightarrow{u_i u_{i+1}}$ and in one of the two regions determined by the line spanned by $\overrightarrow{u_i u_j}$, which gives us four regions in total (see figure \ref{fig:gray_regions}). Now, if $u_{i-1}$ is not in any of the gray regions as in the figure \ref{fig:gray_regions} (a), then $u_{i+1}$ and $u_j$ will be on the same region determined by the line spanned by $\overrightarrow{u_{i-1} u_i}$. This implies that edge $\overrightarrow{u_{i-1} u_i}$ does not alter its condition of being an inflection or not after the cut-and-connect process, and therefore its contribution to the number $2+I-I'$ is zero. If, however, $u_{i-1}$ is in one of the gray regions, then we must also consider subcases in which $\overrightarrow{u_{i-1} u_i}$ goes from being an ordinary edge to being an inflection, and vice-versa (in which case its contribution to the number $2+I-I'$ would be $-1$ resp. $+1$). Figure \ref{fig:gray_regions}(c) shows simultaneously the gray regions for all internal vertices.

\begin{figure}
    \centering
    \includegraphics[scale=0.3]{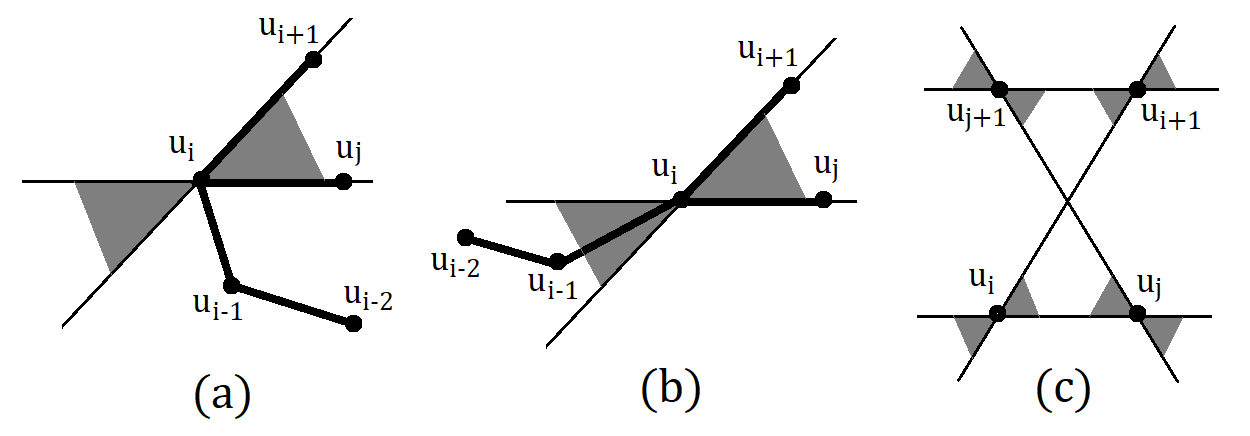}
    \caption{In (a), edge $\overrightarrow{u_{i-1} u_i}$ is not an inflection in $P$ nor in $P'$. In (b), edge $\overrightarrow{u_{i-1} u_i}$ is not an inflection in $P$, but it is an inflection in $P'$.}
    \label{fig:gray_regions}
\end{figure}

Figure \ref{fig:four_cases} shows 4 possibilities (among many others). In case (a), $\overrightarrow{u_i u_{i+1}}$ and $\overrightarrow{u_{i+1} u_{j+1}}$ are inflections, while $\overrightarrow{u_j u_{j+1}}$ and $\overrightarrow{u_i u_j}$ are not. This gives us $2+I-I'=2+0=2$.

In case (b), $\overrightarrow{u_j u_{j+1}}$ and $\overrightarrow{u_{i+1} u_{j+1}}$ are inflections, while $\overrightarrow{u_i u_{i+1}}$ and $\overrightarrow{u_i u_j}$ are not, which means that these four edges contribute zero to the number $2+I-I'$. Now, in this case we must also look at edges $\overrightarrow{u_{i-1} u_i}$ and $\overrightarrow{u_{j-1} u_j}$: both of them are originally inflections but become ordinary edges. Therefore $2+I-I'=2+2=4$.

A similar analysis shows that for case (c) the number $2+I-I'$ equals $0$. Notice that configurations (b) and (c) are the same except for the vertices $u_{i-2}$ and $u_{j-2}$, which in each case determine if edges $\overrightarrow{u_{i-1} u_i}$ and $\overrightarrow{u_{j-1} u_j}$ cease to be or become inflections. This, on its turn, determine a change in the number $2+I-I'$. Hence, in a certain sense, case (c) is a ``worse" situation than case (b) because its corresponding number $2+I-I'$ is less than that of (b), although it does not contradict the conclusion of Lemma \ref{never_decreasing_lemma}. Since we want to simplify the number of configurations to be analyzed, we will assume that if any of the external vertices $u$ is in a gray region, then the adjacent vertex to $u$ which is not internal is positioned in such a way that the corresponding external edge goes from being an ordinary edge to becoming an inflection, i.e., it will contribute $-1$ to the number $2+I-I'$ (that is, since we want to prove that the number $2+I-I'$ is equal or greater than zero, we are already considering the ``worst-case scenario").

For case (d) the number $2+I-I'$ equals $-2$. This does not contradict the validity of Lemma \ref{never_decreasing_lemma}, since in this case vertex $u_{j+2}$ is in the forbidden region and therefore the configuration does not satisfy the hypotheses.

\begin{figure}
    \centering
    \includegraphics[scale=0.45]{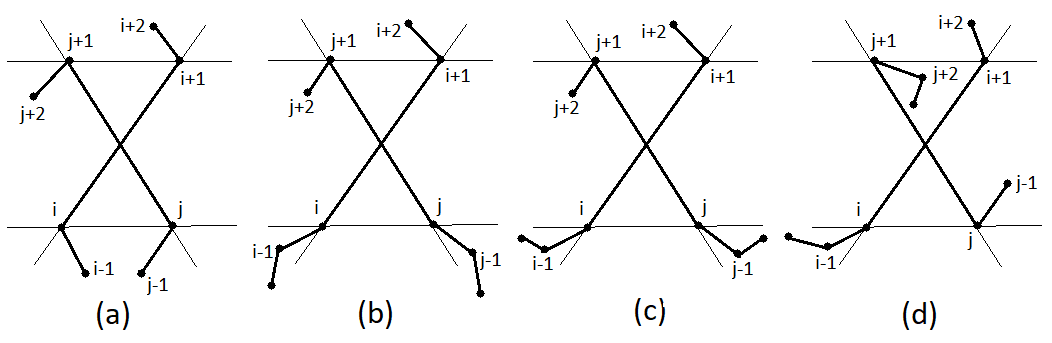}
    \caption{Four different cases}
    \label{fig:four_cases}
\end{figure}

Another important fact is that vertices $u_{j+1}$ and $u_i$ (or $u_{i+1}$ and $u_j$) might be consecutive. In this case, we can just add another vertex $u_k$ between $u_{j+1}$ and $u_i$ as in figure \ref{fig:adding_one_vertex}: just delete edge $\overrightarrow{u_{j+1} u_i}$ and add the new ones $\overrightarrow{u_{j+1} u_k}$ and $\overrightarrow{u_k u_i}$. This can be done without altering the numbers $I$ and $I'$ provided that one places vertex $u_k$ as in figure \ref{fig:adding_one_vertex}: edges $\overrightarrow{u_i u_{i+1}}$, $\overrightarrow{u_j u_{j+1}}$, $\overrightarrow{u_i u_j}$ and $\overrightarrow{u_{i+1} u_{j+1}}$ do not change their condition of being or not being an inflection, while edges $\overrightarrow{u_{j+1} u_i}$ (the deleted one), $\overrightarrow{u_{j+1} u_k}$ and $\overrightarrow{u_k u_i}$ (the new ones) are not inflections. Since the number $2+I+I'$ of the new polygon is equal to the corresponding number of the original polygon, we can therefore assume that our polygons do not have self-intersections of the type of figure \ref{fig:adding_one_vertex}-left.

\begin{figure}
    \centering
    \includegraphics[scale=0.35]{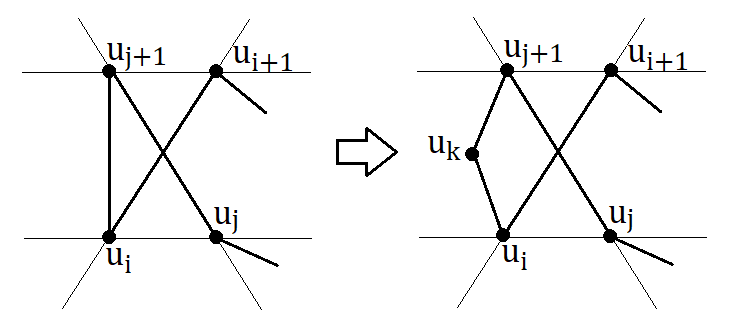}
    \caption{If vertices $u_{j+1}$ and $u_i$ are consecutive, one can add an intermediate vertex without altering the numbers $I$ and $I'$.}
    \label{fig:adding_one_vertex}
\end{figure}

We could prove Lemma \ref{never_decreasing_lemma} by looking at all $3^4=81$ possible configurations (each of the $4$ external vertices might be in $3$ different regions) and checking that the number $2+I-I'$ is always greater or equal to zero. One can notice further that many of these configurations are symmetric (by reflections) to each other. For example, in figure \ref{fig:symmetries} the configurations (a),(b),(c) and (d) can be obtained from the others by reflecting ``horizontally" and ``vertically", and inflections are reflected into each other. Hence the number $2+I-I'=2+0=2$ is the same for these 4 cases. In figure \ref{fig:symmetries} cases (e) and (f) can also be obtained from one another by a reflection, and cases (g) and (h) can also be reflected into each other.

\begin{figure}
    \centering
    \includegraphics[scale=0.4]{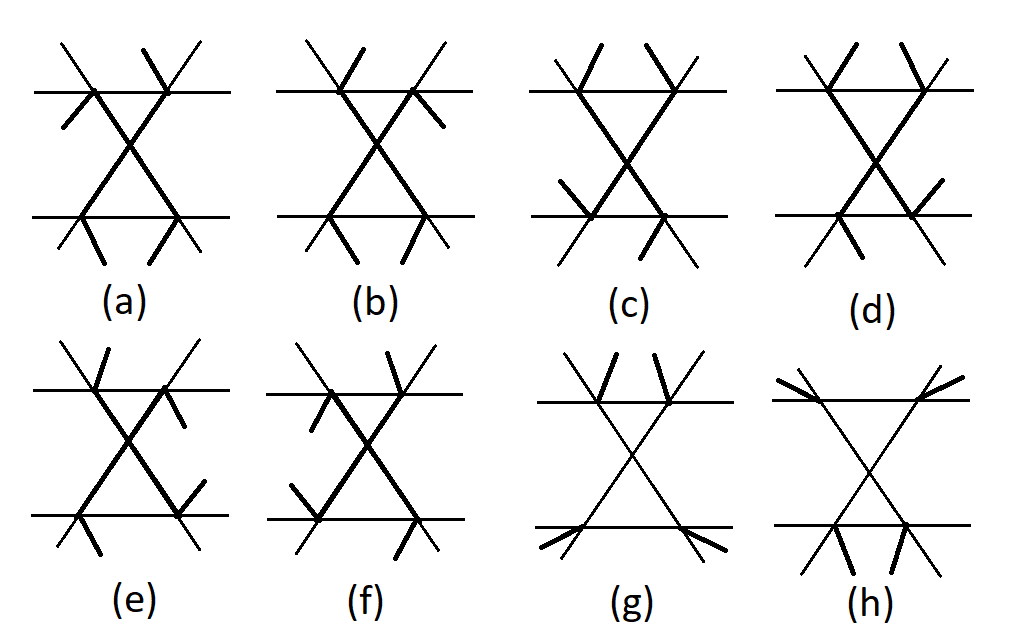}
    \caption{Cases (a),(b),(c) and (d) can be reflected one into the other, and for all of them the number $2+I-I'$ is the same, equal to $2$. Cases (e) and (f) can be reflected into each other as well, and their number $2+I-I'$ is the same, equal to $0$. Cases (g) and (h) can also be reflected into each other, and their number $2+I-I'$ is the same, equal to $0$ (recall that we assume the ``worst-case scenario", i.e., the external edges in the gray regions goes from ordinary edges to becoming inflections).}
    \label{fig:symmetries}
\end{figure}

Therefore the number of configurations to be considered can be decreased. The final number of cases is 27 and is depicted in figure \ref{fig:27_cases}. Since in all the cases the number $\gamma=2+I-I'$ is greater or equal than $0$, lemma \ref{never_decreasing_lemma} is proved.

\begin{figure}
    \centering
    \includegraphics[scale=0.3]{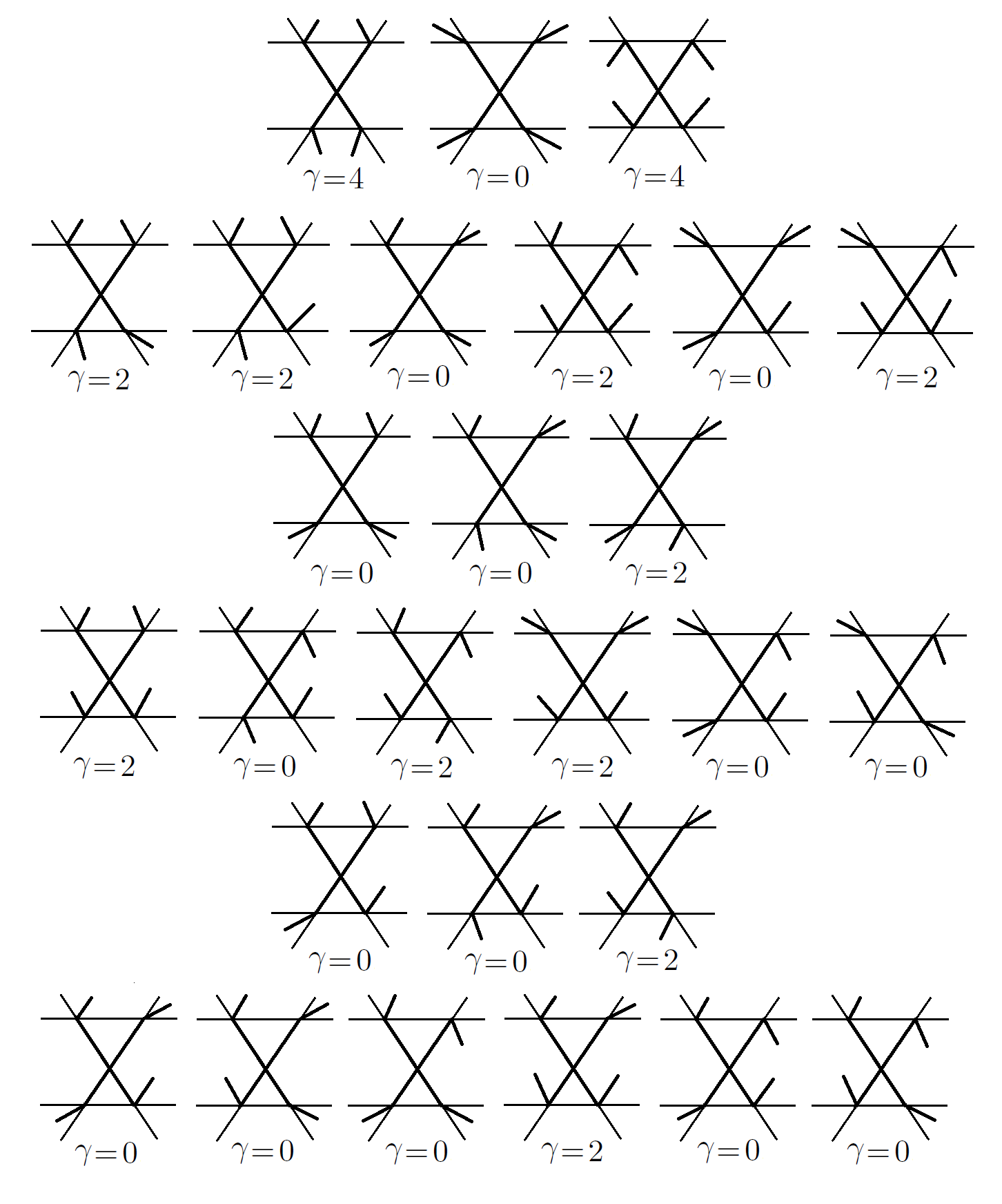}
    \caption{The $27$ possible types of configurations regarding the location of the external vertices. In all of the cases, the number $\gamma=2+I-I'$ is greater or equal to $0$.}
    \label{fig:27_cases}
\end{figure}

A more algebraic approach to checking that the number $2+I-I'$ is always greater or equal to zero may be found in the following (optional) section.

\section{An elementary algebraic approach to study all the possible configurations}

Instead of depicting all the possible configurations, we can do a more synthetic approach: we noticed before that the location of the external vertices suffices to discover the number $\gamma =2+I-I'$ (recall that for each vertex in a gray region we assume an extra contribution of $-1$ to $\gamma$). Each of these vertices does not work alone: the location of pairs of them is what determines which edges (from $\overrightarrow{u_i u_{i+1}}$, $\overrightarrow{u_j u_{j+1}}$, $\overrightarrow{u_i u_j}$ and $\overrightarrow{u_{i+1} u_{j+1}}$) are inflections or not. This suggests that it might be possible to assign to each external vertex a signed number that, all of them considered together, gives us a formula that computes the number $\gamma$. If we succeed at finding such a formula, it will be easier for a computer to check all the possible configurations. If we are lucky, however, it might be possible to prove easily (by elementary algebra) that such a formula always gives zero or positive numbers.

The idea to deduce a formula that gives the number $\gamma$ is to associate to each external vertex a pair $(x_k,y_k) \in \{ \pm 1\}^2$. At each internal vertex two different lines intersect. The corresponding external vertex might be in one of the two regions determined by one line, and in one of the regions determined by the other line. From now on we relabel (for simplicity of notation) each index of the internal vertices as in figure \ref{fig:diagram_plus_minus} (c).  If the external vertex corresponding to the internal vertex $u_k$ is in a region with plus respectively minus sign in figure \ref{fig:diagram_plus_minus} (a), then put $x_k$ equal to $+1$ respectively $-1$. If the external vertex corresponding to the internal vertex $u_k$ is in a region with a plus respectively minus sign in figure \ref{fig:diagram_plus_minus} (b), then put $y_k$ equal to $+1$ respectively $-1$. See figure \ref{fig:example_diagram} for an example. Once this is done, the following facts follow:

\begin{itemize}
    \item edge $\overrightarrow{u_4 u_1}$ is an inflection if and only if $x_1$ and $x_4$ have the same sign, and edge $\overrightarrow{u_3 u_2}$ is an inflection if and only if $x_2$ and $x_3$ have the same sign;
    \item edge $\overrightarrow{u_2 u_1}$ is an inflection if and only if $y_1$ and $y_2$ have opposite signs, and edge $\overrightarrow{u_3 u_4}$ is an inflection if and only if $y_3$ and $y_4$ have opposite signs;
    \item the edge connecting the internal and external vertices at $u_i$ is changes its condition of being an inflection if and only if it is in one of the gray regions, which on its turn happens if and only if $x_i$ and $y_i$ have the same sign;
\end{itemize}

\begin{figure}
    \centering
    \includegraphics[scale=0.4]{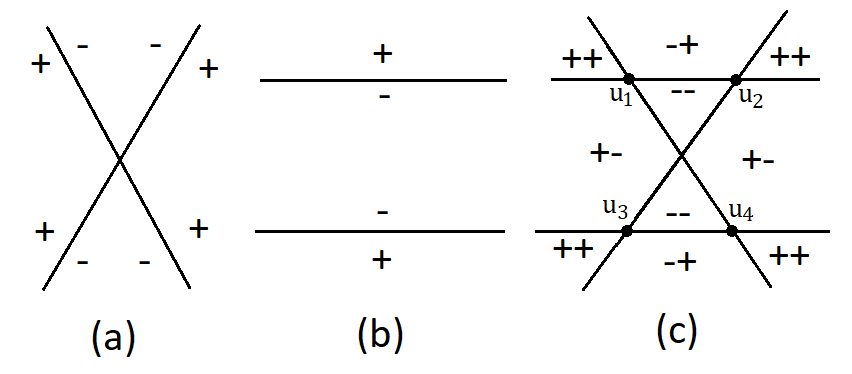}
    \caption{Assigning a pair of signed numbers to each region}
    \label{fig:diagram_plus_minus}
\end{figure}

\begin{figure}
    \centering
    \includegraphics[scale=0.4]{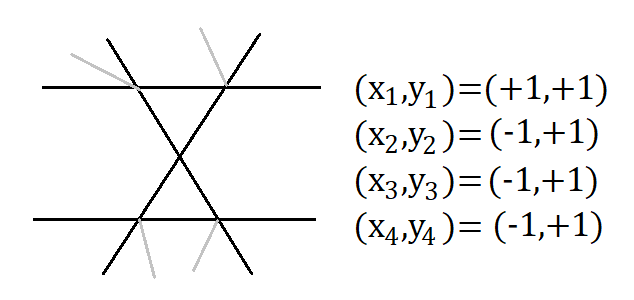}
    \caption{An example of configuration with the respective values of $x_k$'s and $y_k$'s.}
    \label{fig:example_diagram}
\end{figure}

Although the assigning was a bit arbitrary, it still conveys the symmetry we are looking for as figure \ref{fig:symmetries} shows.
Now comes the trick: we want to find a formula in terms of the $x_k$'s and the $y_k$'s that gives the number $\gamma=2+I-I'$. The previous facts imply the following numerical relations, respectively:

\begin{itemize}
    \item We want an expression that gives $1$ if edge $\overrightarrow{u_4 u_1}$ is an inflection, and $0$ if it is not. It is not hard to see that $\frac{x_1 x_4 + |x_1 x_4|}{2}$ does the work. Similarly, for the edge $\overrightarrow{u_3 u_2}$, the expression $\frac{x_2 x_3 + |x_2 x_3|}{2}$ gives $1$ when this edge is an inflection, and $0$ when it is not. Therefore the contribution of inflections among the internal edges to the number $I$ is equal to
    $$ \frac{x_1 x_4 + x_2 x_3 + |x_1 x_4| + |x_2 x_3|}{2} = \frac{x_1 x_4 + x_2 x_3 + 2}{2} = \frac{x_1 x_4 + x_2 x_3}{2} +1.$$
    \item We want an expression that gives $1$ if edge $\overrightarrow{u_2 u_1}$ is an inflection, and $0$ if it is not. Such an expression is $\frac{-y_1 y_2 + |y_1 y_2|}{2}$. Similarly, for the edge $\overrightarrow{u_3 u_4}$, the expression $\frac{-y_3 y_4 + |y_3 y_4|}{2}$ gives $1$ when this edge is an inflection, $0$ when it is not. Therefore the contribution of inflections among the internal edges to the number $I'$ is equal to
    $$ \frac{-y_1 y_2 - y_3 y_4 + |y_1 y_2| + |y_3 y_4|}{2} = \frac{-y_1 y_2 - y_3 y_4 + 2}{2} = -\frac{y_1 y_2 + y_3 y_4 }{2} +1.$$
    \item Assuming the ``worst-case scenario" concerning the change of the conditions of the external edges, we want an expression that gives $1$ if the external edge at $u_k$ is an inflection, and $0$ when it is not (for all $k\in\{1,2,3,4\})$. Such an expression is $\frac{x_k y_k + |x_k y_k|}{2}$. Therefore, the contribution of inflections among the external edges to the number $I'$ is equal to
    $$ \frac{x_1 y_1 + x_2 y_2 + x_3 y_3 + x_4 y_4 + |x_1 y_1| + |x_2 y_2| + |x_3 y_3| + |x_4 y_4|}{2} $$
    $$ = \frac{x_1 y_1 + x_2 y_2 + x_3 y_3 + x_4 y_4 + 4}{2} $$
    $$ = \frac{x_1 y_1 + x_2 y_2 + x_3 y_3 + x_4 y_4}{2} + 2.$$
\end{itemize}

Our conclusion is that the number $\gamma = 2 + I - I'$ is equal to
$$ 2 + \frac{x_1 x_4 + x_2 x_3}{2} +1 - (-\frac{y_1 y_2 + y_3 y_4 }{2} +1 + \frac{x_1 y_1 + x_2 y_2 + x_3 y_3 + x_4 y_4}{2} + 2)$$
$$ = \frac{x_1 x_4 + x_2 x_3 + y_1 y_2 + y_3 y_4 -x_1 y_1 - x_2 y_2 - x_3 y_3 - x_4 y_4}{2}.$$

Notice that this formula is invariant by permuting the indices by $(1 2)(3 4)$ (which is related by the ``y-axis reflection" symmetry) and by $(1 3) (2 4)$ (which is related by the ``x-axis reflection" symmetry). Therefore the formula gives the same number for configurations that can be obtained one from another by reflections (as we saw in the examples of figure \ref{fig:symmetries}).

It remains to prove that the number $\gamma$ is always positive or equal to zero, provided that the external vertices are not in the forbidden regions. By the previous discussion, it suffices to prove the following lemma:

\begin{lemma}
\label{polynomial_lemma}
If $(x_k,y_k) \neq (-1,-1)$ for all $k \in \{1,2,3,4\}$, then $2 \gamma$, given by
$$ x_1 x_4 + x_2 x_3 + y_1 y_2 + y_3 y_4 -x_1 y_1 - x_2 y_2 - x_3 y_3 - x_4 y_4,$$
is always greater or equal to zero.
\end{lemma}
\begin{proof}
    Consider the sequence $(x_1 y_1, x_2 y_2, x_3 y_3, x_4 y_4)$. There are five possibilities regarding the number of $+1$'s in the sequence:
    \begin{itemize}
        \item 4 of them. The hypothesis implies then that all $x_k$'s and $y_k$'s are equal to $1$, and therefore $2\gamma$ is equal to $0$.
        \item 3 of them. The hypothesis implies then that exactly one number among the $x_k$'s and $y_k$'s is equal to $-1$, which then implies that $2\gamma$ is equal to $0$.
        \item 2 of them. This immediately implies that the last four terms of the expression of $2\gamma$ add up to $0$. Moreover, by hypothesis, this implies that exactly two numbers among the $x_k$'s and $y_k$'s are equal to $-1$. This in turn implies that at least two of the first four terms of the expression of $2\gamma$ are equal to $1$, which suffices for the entire expression of $2\gamma$ to be greater or equal to $0$.
        \item 1 of them. This immediately implies that the last four terms of the expression of $2\gamma$ add up to $2$. Moreover, by hypothesis, this implies that exactly five numbers among the $x_k$'s and $y_k$'s are equal to $+1$. This in turn implies that at least one of the first four terms of the expression of $2\gamma$ is equal to $1$, which suffices for the entire expression of $2\gamma$ to be greater of equal to $0$.
        \item $0$ of them. This immediately implies that the last four terms of the expression of $2\gamma$ add up to $4$, which suffices for the entire expression of $2\gamma$ to be greater or equal to $0$.
    \end{itemize}
\end{proof}

Lemma \ref{never_decreasing_lemma} now follows immediately from lemma \ref{polynomial_lemma}.

\section{Proofs of theorems}

Now we can prove the promised theorems.

\begin{proof}
    (of Theorem \ref{generalized_segre_discrete}) It suffices to prove the theorem when $D^+ = 1$. If the polygon $Q$ satisfies at its self-intersection the hypothesis of Lemma \ref{never_decreasing_lemma} and the corresponding polygon $Q'$ is obtained as indicated, then $2D^+ + I = 2 + I \geq 2 D'^+ + I' = 0 + I' = I'$, which is greater or equal to $4$, since $Q'$ satisfies the hypotheses of Theorem \ref{segre_theorem_discrete_spherical}.

    If, however, one (or even both) of the regions determined by $\{u_i,u_j,w\}$ and $\{u_{i+1},u_{j+1},w\}$ contain in its interior other vertices of $Q$ (which would then imply that the new polygon would have a self-intersection), then we construct and intermediate polygon before ``cutting-and-pasting". We put
    $$\epsilon = \frac{1}{2} \min \{|w - u_i|_{\mathbb{S}^2}; i \in \{1,...,n\}\}$$
    (where we denote by $|\cdot|_{\mathbb{S}^2}$ the spherical distance) and denote by $S_{\epsilon}(w)$ the circle with radius $\epsilon$ centered at $w$. We then add new vertices to the polygon $Q$ in the following way (see figure \ref{fig:adding_vertices}):

    \begin{itemize}
        \item If the region determined by $\{u_i,u_j,w\}$ has some other vertex in its interior, then add to $Q$ the new vertices $u_{i+\frac{1}{3}} := \overrightarrow{u_i w} \cap S_\epsilon(w)$ and $u_{j+\frac{1}{3}} := \overrightarrow{u_j w} \cap S_\epsilon(w)$, with the ordering induced by the usual ordering of rational numbers $\mod n$;
        \item if the region determined by $\{u_{i+1},u_{j+1},w\}$ has some other vertex in its interior, then add to $Q$ the new vertices $u_{i+\frac{2}{3}} := \overrightarrow{w u_{i+1}} \cap S_\epsilon(w)$ and $u_{j+\frac{2}{3}} := \overrightarrow{w u_{j+1}} \cap S_\epsilon(w)$, with the ordering induced by the usual ordering of rational number $\mod n$.
    \end{itemize}

\begin{figure}
    \centering
    \includegraphics[scale=0.3]{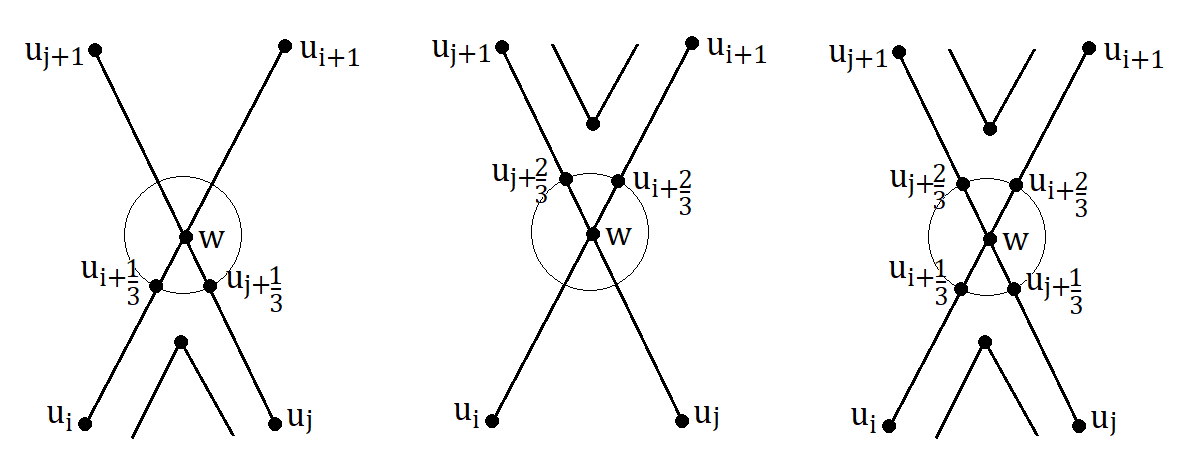}
    \caption{Adding vertices}
    \label{fig:adding_vertices}
\end{figure}

Denote the newly obtained polygon by $Q'$. Perturb the vertices of $Q'$ slightly so that it does not have three consecutive collinear vertices nor new self-intersections, and such that the number of inflections of $Q'$ is the same as the original polygon $Q$ (see figure \ref{fig:new_vertices_inflections} for examples). This can always be done. Continue denoting such polygon by $Q'$. Now, the new polygon $Q'$ satisfies the hypotheses of Lemma \ref{never_decreasing_lemma}, with $I'=I$ and $D'^+ = D^+$. Therefore we can construct from $Q'$ a new polygon $Q''$ without self-intersections as in the discussion before Lemma \ref{never_decreasing_lemma} (in the case of figure \ref{fig:adding_vertices} on the right, for instance, the internal vertices would be $u_{i+\frac{1}{3}}$, $u_{i+\frac{2}{3}}$, $u_{j+\frac{1}{3}}$ and $u_{j+\frac{2}{3}}$). By the latter result we have:
$$ 2D^+ + I = 2D'^+ + I' \geq 2D''^+ + I'' = 0 + I'' \geq 4,$$
where the last inequality holds by Theorem \ref{segre_theorem_discrete_spherical}.

\begin{figure}
    \centering
    \includegraphics[scale=0.3]{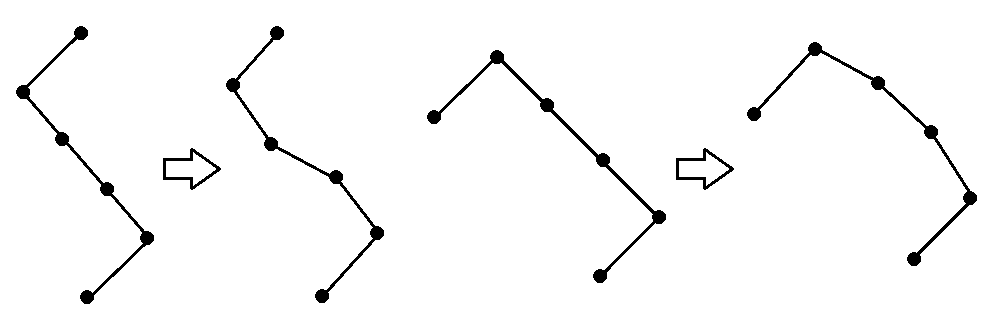}
    \caption{Perturbing a polygon with three consecutive vertices so that the number of its inflections will still be the same.}
    \label{fig:new_vertices_inflections}
\end{figure}

\end{proof}

Now we prove Theorem \ref{generalized_mobius_discrete}:

\begin{proof} (of Theorem \ref{generalized_mobius_discrete})
    If the original polygon $Q$ is symmetric, then $D^+$ is always even. If $D^+ \geq 4$, the result follows immediately. If $D^+=0$, then we apply theorem \ref{discrete_mobius_theorem}. Now, if $D^+=2$, then we can apply the same procedure as we did in this section: construct from $Q$ a new polygon $Q'$ through ``cutting-and-pasting" at each of the two self-intersections (adding new vertices if necessary), so that $Q'$ does not have self-intersections. Notice that each step of this procedure is invariant by reflection on the origin, hence $Q'$ is also symmetric. Therefore
    $$ 2D^+ + I \geq I' \geq 6,$$
    where the first inequality holds applying Lemma \ref{never_decreasing_lemma} twice and the second inequality holds by Theorem \ref{discrete_mobius_theorem}.
\end{proof}

We still need to prove theorem \ref{generalized_special_segre_discrete}. Recall that $D^+$ is the number of self-intersections of $Q$, $D^-$ is the number of antipodal intersections of $Q$ and $D=D^+ + D^-$.

\begin{proof} (of theorem \ref{generalized_special_segre_discrete})
    If for a given spherical polygon $Q$ not contained in any hemisphere the number $D^-$ equals $1$, then
    $$ 2D + I = 2(D^+ + D^-) + I = 2D^- + 2D^+ + I = 2 + 2D^+ + I \geq 2 + 4 = 6,$$
    where the inequality holds by theorem \ref{generalized_segre_discrete}. If, however, $D^- = 0$, then we still need to prove that in this case the number $2D^+ + I$ is greater or equal to $6$. Now, $Q$ either has no self-intersections or has at least one. In the former case we apply theorem \ref{special_segre_discrete}, while in the latter case we proceed as we did in the previous section: at the self-intersection we cut-and-paste the edges of $Q$ such that the resulting polygon $Q$ has no self-intersections and no antipodal intersections. In case we need to add new vertices before cutting-and-pasting, we must take now as $\epsilon$ the number
    $$ \epsilon = \frac{1}{2} \min \{|w \pm u_i|_{\mathbb{S}^2}; i \in \{1,...,n\}\}$$
    (where we denote again by $|\cdot|_{\mathbb{S}^2}$ the spherical distance), so that the proof of theorem \ref{generalized_segre_discrete} also works here. The conclusion is that $2D^+ + I$ is greater or equal to $I'$, where $I'$ is the number of inflections of the resulting polygon $Q'$, which does not have self-intersections nor antipodal intersections. By theorem \ref{special_segre_discrete}, this number is greater or equal to $6$.
\end{proof}

\section{Further remarks}

The ``cutting-and-pasting" approach used around the double points (i.e., intersection points) is also used in the smooth setting by Ghomi (see \cite{Ghomi}) to prove theorems \ref{generalized_special_segre_smooth} and \ref{generalized_segre_smooth}. In our case, however, the separate study of the 27 cases (up to symmetry) around the intersection point to prove lemma \ref{never_decreasing_lemma} is necessary since we are dealing with (spherical) polygons. The elementary algebraic argument of section 5.4 and the proof of lemma \ref{polynomial_lemma} are new.

In the smooth setting, theorems \ref{generalized_special_segre_smooth} and \ref{generalized_segre_smooth} are due to Ghomi (see \cite{Ghomi}). From these theorems, he also considers corresponding results for space curves: self/antipodal intersections in the tangent indicatrix correspond to pairs of parallel tangents with the same/inverse orientation in the original space curve. Taking into account the number of such pairs, he obtains inequalities for space curves which are analogous to the inequalities of theorems \ref{generalized_special_segre_smooth} and \ref{generalized_segre_smooth}.

As far as we know, the proofs of theorems \ref{generalized_special_segre_discrete}, \ref{generalized_segre_discrete} and \ref{generalized_mobius_discrete} in the discrete setting are new. Notice also that these theorems, although stated for spherical polygons, can be thought as applied to tangent indicatrices of space polygons. Since generic space polygons do not have ``pairs of parallel tangents", one can define pairs of \textit{parallel vertices with the same/inverse orientation} of a space polygon $P$ as pairs of vertices $v_i$ and $v_j$ such that the tangent indicatrix $Q$ of $P$ has self/ antipodal intersections between spherical edges $\overrightarrow{u_{i-1} u_i}$ and $\overrightarrow{u_{j-1} u_j}$. Taking into account the number of such pairs, we can obtain the following inequalities for space polygons which are analogous to the inequalities of theorems \ref{generalized_special_segre_discrete} and \ref{generalized_segre_discrete}. Denote by $T^+$ and $T^-$ the numbers of pairs of parallel vertices with the same/inverse orientation of a space polygon, and put $T:=T^+ + T^-$. Denote also by $F$ the number of flattenings of a space polygon.

\begin{theorem}
\label{analogous_generalized_special_segre_discrete}
Let $P=[v_1,...,v_n] \in \mathbb{R}^3$ ($n\geq 6$) be a space polygon. Then
$$ 2T + F \geq 6.$$
\end{theorem}

\begin{theorem}
\label{analogous_generalized_segre_discrete}
Let $P=[v_1,...,v_n] \in \mathbb{R}^3$ ($n\geq 6$) be a space polygon. Then
$$ 2T^+ + F \geq 4.$$
\end{theorem}

Theorems \ref{analogous_generalized_special_segre_discrete} and \ref{analogous_generalized_segre_discrete} follow immediately from theorems \ref{generalized_special_segre_discrete} and \ref{generalized_segre_discrete}, respectively.

\section*{Acknowledgments}

This study was financed in part by the Coordenação de Aperfeiçoamento de Pessoal de Nível Superior - Brasil (CAPES) - Finance Code 001. It was also financed in part by Fundação Carlos Chagas Filho de Amparo à Pesquisa do Estado do Rio de Janeiro - Rio de Janeiro, Brasil (FAPERJ).

\printbibliography

@article{Mukhopadhyaya,
  author =       "S. Mukhopadhyaya",
  title =        "{New methods in the geometry of plane arc}.",
  journal =      "Bull. Calcutta Math. Soc.",
  volume =       "\textbf{1}",
  number =       "",
  pages =        "31-37",
  year =         "1909",
  DOI =          ""
}

@article{Segre,
  author =       "B. Segre",
  title =        "{Alcune proprietà differenziali in grande delle curve                     chiuse sghembe}.",
  journal =      "Rend. Math",
  volume =       "\textbf{1}",
  number =       "no. 6",
  pages =        "237--297",
  year =         "1968",
  DOI =          ""
}

@incollection{Mobius,
    author    = {A. F. Möbius},
    title     = {Über die Grundformen der Linien der dritten Ordnung},
    origyear  = {},
    pages     = {89-176},
    editor    = {},
    booktitle = {Gesammelte Werke},
    booksubtitle = {},
    volume    = {2},
    year      = {1886},
    publisher = {Verlag von S. Hirzel},
    address   = {Leipzig},
    note      = {},
}

@article{Ghomi,
  author =       "M. Ghomi",
  title =        "{Tangent lines, inflections, and vertices of closed curves}",
  journal =      "Duke Math. J",
  volume =       "162",
  number =       "",
  pages =        "2691-2730",
  year =         "2013",
  DOI =          ""
}

@article{Panina,
  author =       "G. Panina",
  title =        "{Singularities of piecewise linear saddle spheres on $\mathbb{S}^3$}",
  journal =      "Journal of Singularities",
  volume =       "1",
  number =       "",
  pages =        "69--84",
  year =         "2010",
  DOI =          ""
}

@online{Gentil,
    author    = "S. P. Gentil and M. Craizer",
    title     = "A discrete analog of Segre's theorem on spherical curves",
    url       = "arXiv: 2208.04454"
}

@article{Banchoff,
  author =       "T.F. Banchoff",
  title =        "{Global geometry of Polygons. III. Frenet Frames and                   theorems of Jacobi and Milnor for space polygons}",
  journal =      "Jag. Jugoslav. Znan. Umjet.",
  volume =       "396",
  number =       "",
  pages =        "101--108",
  year =         "1982",
  DOI =          ""
}

\end{document}